\documentclass[12pt,a4paper]{amsart}
\usepackage{amssymb,eucal}
\usepackage[cmtip,arrow]{xy}
\usepackage{pb-diagram,pb-xy}

\calclayout

\newcommand{\+}{\protect\nobreakdash-}
\renewcommand{\:}{\colon}

\newcommand{\rarrow}{\longrightarrow}
\newcommand{\larrow}{\longleftarrow}
\newcommand{\birarrow}{\rightrightarrows}
\newcommand{\ot}{\otimes}
\newcommand{\ocn}{\odot}
\newcommand{\tim}{\rightthreetimes}

\newcommand{\bu}{{\text{\smaller\smaller$\scriptstyle\bullet$}}}
\newcommand{\lrarrow}{\mskip.5\thinmuskip\relbar\joinrel\relbar\joinrel
 \rightarrow\mskip.5\thinmuskip\relax}

\DeclareMathOperator{\Hom}{Hom}
\DeclareMathOperator{\Ext}{Ext}
\DeclareMathOperator{\Tor}{Tor}
\DeclareMathOperator{\id}{id}
\DeclareMathOperator{\im}{im}
\DeclareMathOperator{\coker}{coker}
\DeclareMathOperator{\pd}{pd}

\newcommand{\bB}{\mathbb B}
\newcommand{\bF}{\mathbb F}
\newcommand{\bG}{\mathbb G}
\newcommand{\bH}{\mathbb H}
\newcommand{\bI}{\mathbb I}
\newcommand{\bJ}{\mathbb J}
\newcommand{\bM}{\mathbb M}
\newcommand{\bP}{\mathbb P}

\newcommand{\bfB}{\boldsymbol{\mathfrak B}}
\newcommand{\bfF}{{\boldsymbol{\mathfrak F}}}
\newcommand{\bfG}{{\boldsymbol{\mathfrak G}}}
\newcommand{\bfH}{{\boldsymbol{\mathfrak H}}}

\newcommand{\A}{\mathfrak A}
\newcommand{\C}{\mathfrak C}
\newcommand{\D}{\mathfrak D}
\newcommand{\E}{\mathfrak E}
\newcommand{\I}{\mathfrak I}
\newcommand{\J}{\mathfrak J}
\newcommand{\M}{\mathfrak M}
\newcommand{\R}{\mathfrak R}
\newcommand{\sS}{\mathfrak S}
\newcommand{\U}{\mathfrak U}

\newcommand{\N}{\mathcal N}

\newcommand{\sA}{\mathsf A}
\newcommand{\sC}{\mathsf C}
\newcommand{\sD}{\mathsf D}
\newcommand{\sF}{\mathsf F}
\newcommand{\sP}{\mathsf P}
\newcommand{\sQ}{\mathsf Q}

\newcommand{\sT}{\mathsf T}

\newcommand{\boZ}{\mathbb Z}
\newcommand{\boQ}{\mathbb Q}

\newcommand{\modl}{{\operatorname{\mathsf{--mod}}}}
\newcommand{\modr}{{\operatorname{\mathsf{mod--}}}}
\newcommand{\bimod}{{\operatorname{\mathsf{--mod--}}}}
\newcommand{\contra}{{\operatorname{\mathsf{--contra}}}}
\newcommand{\discr}{{\operatorname{\mathsf{discr--}}}}

\newcommand{\ctra}{{\operatorname{\mathsf{-ctra}}}}

\newcommand{\Sets}{\mathsf{Sets}}
\newcommand{\Ab}{\mathsf{Ab}}
\newcommand{\Add}{\mathsf{Add}}
\newcommand{\Com}{\mathsf{Com}}

\newcommand{\Dh}{\mathrm{DH}}
\newcommand{\CT}{\mathrm{CT}}
\newcommand{\IL}{\mathrm{IL}}
\newcommand{\PL}{\mathrm{PL}}
\newcommand{\il}{\mathrm{il}}
\newcommand{\pl}{\mathrm{pl}}
\newcommand{\tp}{\mathrm{tp}}
\newcommand{\mh}{\mathrm{mh}}

\newcommand{\rop}{{\mathrm{op}}}
\newcommand{\sop}{{\mathsf{op}}}

\newcommand{\Section}[1]{\bigskip\section{#1}\medskip}
\theoremstyle{plain}
\newtheorem{thm}{Theorem}[section]

\newtheorem{lem}[thm]{Lemma}
\newtheorem{prop}[thm]{Proposition}
\newtheorem{cor}[thm]{Corollary}
\theoremstyle{definition}

\newtheorem{exs}[thm]{Examples}
\newtheorem{rem}[thm]{Remark}

\begin{document}

\title{Flat ring epimorphisms of countable type}
\author{Leonid Positselski}

\address{Institute of Mathematics, Czech Academy of Sciences,
\v Zitn\'a~25, 115~67 Prague~1, Czech Republic; and
\newline\indent Laboratory of Algebraic Geometry, National Research
University Higher School of Economics, Moscow 119048; and
\newline\indent Sector of Algebra and Number Theory, Institute for
Information Transmission Problems, Moscow 127051, Russia; and
\newline\indent Department of Mathematics, Faculty of Natural Sciences,
University of Haifa, Mount Carmel, Haifa 31905, Israel}

\email{positselski@yandex.ru}

\begin{abstract}
 Let $R\rarrow U$ be an associative ring epimorphism such that $U$ is
a flat left $R$\+module.
 Assume that the related Gabriel topology $\bG$ of right ideals in $R$
has a countable base.
 Then we show that the left $R$\+module $U$ has projective dimension
at most~$1$.
 Furthermore, the abelian category of left contramodules over
the completion of $R$ at $\bG$ fully faithfully embeds into the
Geigle--Lenzing right perpendicular subcategory to $U$ in the category
of left $R$\+modules, and every object of the latter abelian category is
an extension of two objects of the former one.
 We discuss conditions under which the two abelian categories are
equivalent.
 Given a right linear topology on an associative ring $R$, we consider
the induced topology on every left $R$\+module, and for a perfect
Gabriel topology $\bG$ compare the completion of a module with
an appropriate Ext module.
 Finally, we characterize the $U$\+strongly flat left $R$\+modules by
the two conditions of left positive-degree Ext-orthogonality to all left
$U$\+modules and all $\bG$\+separated $\bG$\+complete left $R$\+modules.
\end{abstract}

\maketitle

\tableofcontents

\section{Introduction}
\bigskip

\setcounter{subsection}{-1}
\subsection{{}}
 Ring epimorphisms and Gabriel topologies are a popular subject of
contemporary research in associative and commutative ring theory,
where nontrivial work is being done~\cite{AS,Hrb,AH,MS,AMSTV}.
 To be more precise, one has to say that these are two overlapping,
but different subjects: \emph{perfect} right Gabriel topologies
correspond bijectively to left \emph{flat} ring epimorphisms.
 In this paper, we are mostly dealing with this overlap, imposing
additional conditions as the need arises.

 The most important of such additional conditions is the one of
\emph{countable type}.
 A left flat ring epimorphism $u\:R\rarrow U$ is said to be of
countable type if the related Gabriel topology $\bG$ of right ideals
in $R$ has a countable base.
 In Section~\ref{perfect-of-type-lambda-secn} we discuss how ubiquitous
(perfect) Gabriel topologies of countable type are among (perfect)
Gabriel topologies in general; in Sections~\ref{fully-faithful-secn}
and~\ref{projdim-1-secn}\+-\ref{when-Delta=Lambda-secn}
we show what one can do with (mostly perfect) Gabriel topologies of
countable type; and in the final
Sections~\ref{when-Delta=Lambda-secn}\+-\ref{faithful-perfect-secn} we
combine these results, proving theorems about completions and
contramodules related to some perfect Gabriel topologies of
uncountable type.

 Another condition on which many of our results depend is that, even
when a right Gabriel topology on an associative ring $R$ is not perfect,
we often need it to have a base consisting of finitely generated right
ideals.
 For perfect Gabriel topologies, this holds automatically.

 Concerning ubiquitousness of topologies of countable type, the main
obstacle to that appears to be unrelated to the topologies being perfect
or even Gabriel.
 The right linear topology axiom~(T3) from the book~\cite{St},
characterizing filters of right ideals defining a topological ring
structure (with continuous multiplication) on an associative ring,
allows to produce an open ideal from every pair (open ideal,
element of the ring).
 When the ring is commutative, this axiom is trivial; but for rings
that are noncommutative enough and uncountable enough, it potentially
may prevent existence of countably based right linear topologies.

 In Section~\ref{perfect-of-type-lambda-secn} of this paper, we impose
an additional axiom~(T$_\lambda$) to control this problem.
 Given an associative ring $R$, an infinite cardinal~$\lambda$, and
a right Gabriel topology $\bG$ on $R$ satisfying~(T$_\lambda$) and having
a base consisting of finitely generated right ideals, we show that $\bG$
is the union of Gabriel topologies $\bH\subset\bG$ having bases of
the cardinality not exceeding~$\lambda$ consisting of finitely generated
right ideals.
 When $\bG$ is a perfect Gabriel topology, we assume additionally that
the $\bG$\+torsion in $R$ is $\lambda$\+bounded (e.~g., $\bG$ is
faithful) and then show that $\bG$ is the union of perfect Gabriel
topologies $\bP\subset\bG$ with bases of the cardinality not
exceeding~$\lambda$.

\subsection{{}}
 The main heroes of this paper are, of course, the \emph{contramodules}.
 In the previous papers~\cite{BP0,PSl0,PSl} of the present author with
collaborators, we applied contramodule techniques in order to describe
flat modules over certain commutative rings as direct summands of
transfinitely iterated extensions of flat modules of special type
(namely, the localizations of the ring with respect to various
multiplicative subsets).
 In these approaches, going back to the work of Trlifaj~\cite{Trl0,Trl}
and Bazzoni--Salce~\cite{BS,BS2}, one starts from defining certain
subclasses of flat modules, and then proceeds to show that, under
certain assumptions, the whole class of flat modules coincides with
such a subclass.

 In the present paper, we make the first steps towards the important
goal of extending the results of these papers, and particularly
of~\cite{PSl}, to noncommutative rings.
 In Section~\ref{strongly-flat-weakly-cotorsion-secn}, we define
the class of \emph{$U$\+strongly flat} left $R$\+modules for a left
flat ring epimorphism $u\:R\rarrow U$.
 Assuming that the related right Gabriel topology $\bG$ on $R$ has
a countable base, we characterize $U$\+strongly flat left $R$\+modules
by a pair of conditions which, in the additional assumption that $\bG$
has a base consisting of two-sided ideals, reduces to a direct
generalization of what was suggested, for commutative case,
in~\cite[Optimistic Conjecture~1.1]{PSl}.
 More precisely, a flat left $R$\+module $F$ is $U$\+strongly flat if
and only if the left $U$\+module $U\ot_RF$ is projective and, for
every two-sided ideal $H\subset R$ belonging to $\bG$, the left
$R/H$\+module $F/HF$ is projective.

\subsection{{}}
 The most important results of this paper are proved in
Section~\ref{projdim-1-secn}.
 First of all, it is the theorem that, for any left flat ring
epimorphism $u\:R\rarrow U$ of countable type, the flat left $R$\+module
$U$ has projective dimension at most~$1$.
 The proof uses contramodules over the completion $\R$ of
the topological ring $R$ with respect to its perfect right Gabriel
topology~$\bG$.
 Here $\R$ is viewed as a complete, separated topological ring in
the completion (``projective limit'') topology.

 Here it should be noted that such a result is easily provable for
rings $R$ in which every left ideal has a countable set of generators.
 Indeed, in these assumptions the left $R$\+module $U$ is countably
presented (as one can see from its explicit construction as
the ring/bimodule of quotients $U=R_\bG$), and a countably presented
flat module over an associative ring always has projective dimension
at most~$1$ (see~\cite[Corollary~2.23]{GT}).
 Our contramodule-based approach, while much more technical, allows
to obtain an extra generality.

 Furthermore, in Section~\ref{fully-faithful-secn} we prove that,
for any right Gabriel topology $\bG$ on an associative ring $R$ having
a countable base of finitely generated right ideals, the forgetful
functor from the abelian category of left $\R$\+contramodules
$\R\contra$ to the abelian category of left $R$\+modules
$R\modl$ is fully faithful.
 In Section~\ref{projdim-1-secn}, we compare the category of
left $\R$\+contramodules for the completion $\R$ of an associative
ring $R$ with respect to its perfect Gabriel topology $\bG$ of
countable type with the Geigle--Lenzing abelian perpendicular
subcategory $U^{\perp_{0,1}}\subset R\modl$.
 We show that the Geigle--Lenzing perpendicular subcategory (which
we also denote by $R\modl_{u\ctra}$ and call ``the full subcategory of
$u$\+contramodule left $R$\+modules'') is, generally speaking,
a wider full subcategory in $R\modl$ than $\R\contra$: a left
$R$\+module belongs to $R\modl_{u\ctra}\subset R\modl$ if and only if it
is an extension of two left $R$\+modules belonging to
$\R\contra\subset R\modl$ (i.~e., admitting left $\R$\+contramodule
structures).

 When $u\:R\rarrow U$ is an \emph{injective} left flat ring epimorphism
of countable type (or, in other words, the corresponding perfect right
Gabriel topology $\bG$ on $R$ is \emph{faithful}), we show
in Section~\ref{projdim-1-secn} that the two full subcategories
$R\modl_{u\ctra}$ and $\R\contra\subset R\modl$ actually coincide.
 In Section~\ref{faithful-perfect-secn}, we extend this result to left
flat ring epimorphisms of uncountable type, showing that, under certain
assumption, there is an equivalence of abelian categories
$\R\contra\simeq R\modl_{u\ctra}$ for an injective ring epimorphism
$u\:R\rarrow U$ such that $U$ is a flat left $R$\+module of
projective dimension at most~$1$.
 The additional assumption here is that of the condition~(T$_\omega$),
which is needed in order to use the results of
Section~\ref{perfect-of-type-lambda-secn}; in particular, it is
satisfied automatically for right linear topology with a base of
two-sided ideals.

 The similar questions for multiplicative subsets and finitely generated
ideals in commutative rings (or central multiplicative
subsets/centrally finitely generated ideals in associative rings)
were discussed at length in the examples in~\cite[Sections~2, 3,
and~5]{Pper}.
 In this paper, we extend this discussion to the realm of left flat
ring epimorphisms and perfect Gabriel topologies.

 In Section~\ref{when-Delta=Lambda-secn} we offer a treatment, in
the generality of left flat ring epimorphisms and the related
right Gabriel topologies, of a classical topic, going back
to Nunke~\cite{Nun} and Matlis~\cite{Mat}.
 This concerns connections between the completion of an $R$\+module
$M$ in what was classically called the ``$R$\+topology'' and
the Ext module $\Ext^1_R(K,M)$, where $K$ was classically defined as
the quotient module $Q/R$ of the field of fractions $Q$ of the ring $R$
by its subring $R\subset Q$.
 A natural morphism from the Ext module to the completion was
constructed and a sufficient condition for it to be an isomorphism
was established in Matlis' memoir~\cite[Proposition~2.4 and
Theorem~6.10]{Mat}.

 An original idea of how to produce, under certain assumption,
a map in the opposite direction (i.~e., from the completion to
the Ext module) was suggested in a recent preprint by Facchini
and Nazemian~\cite[Sections~3\+-4]{FN}.
 We seize on their idea and use it, in combination with Matlis'
classical approach, in order to obtain an isomorphism of left
$R$\+modules between the completion $\Lambda_\bG(M)$ of
a $u$\+torsion-free left $R$\+module $M$ with respect to a perfect
right Gabriel topology $\bG$ on $R$, and an Ext (or rather,
derived category Hom) module $\Ext^1_R(K^\bu_{R,U},M)$.
 Here $u\:R\rarrow U$ is the left flat ring epimorphism related
to~$\bG$ and $K^\bu_{R,U}$ is the two-term complex $R\rarrow U$.

 Some of the technical assumptions on the Gabriel topology $\bG$
mentioned above in this introduction are required for our
proof of this result (which we first prove for perfect Gabriel
topologies with a countable base, and then extend to
the uncountable case using the results of
Section~\ref{perfect-of-type-lambda-secn}).

\subsection{{}}
 Section~\ref{preliminaries-secn} contains preliminary material on
linear topologies, Gabriel topologies, completions, discrete/torsion
modules, rings of quotients, and contramodules.
 The reader may wish to consult~\cite[Section~1]{BP} for further
preliminaries on contramodules with further references, and
the overview~\cite{Prev} for a more leisurely introduction.

 In Sections~\ref{F-systems-secn}\+-\ref{separated-contramodules-secn},
we spell out and develop the technique of \emph{$\bF$\+systems},
which first appeared, in a slightly disguised form, in
the paper~\cite[Section~6]{PR}.
 This is, in fact, an important technique for working with contramodules
over noncommutative topological rings, without which even very simple
constructions, such that the completion of a left module with respect to
a topology of right ideals in a ring, cannot be confidently performed.
 In fact, we discuss such a completion construction in
Section~\ref{strongly-flat-weakly-cotorsion-secn}, and its (more
complicated) contramodule version in
Section~\ref{separated-contramodules-secn}.
 The observation that separated contramodules, particularly over
topological rings with countably based right linear topologies, can be
described in terms of covariant $\bF$\+systems was instrumental
in~\cite[Section~6]{PR}, and is also important in this paper.

 In Section~\ref{fully-faithful-secn}, we discuss the important
question when the forgetful functor $\R\contra\rarrow R\modl$ is
fully faithful, building upon the argument that first appeared
in~\cite[Theorem~1.1]{Psm} and was subsequently developed
in~\cite[Section~3]{Pper}.
 In this paper, we improve upon these results of
the papers~\cite{Psm,Pper}, providing an equivalent
\emph{characterization} of countably based right linear topologies
$\bF$ on associative rings $R$ for which such full-and-faithfulness
property holds.
 We also show that it always holds for Gabriel topologies $\bG$ with
a countable base of finitely generated right ideals.
 It is worth mentioning that the contramodule Nakayama lemma was
a key technical tool for proving the full-and-faithfulness
theorem in~\cite{Psm} and remains so in the present paper.

 In Section~\ref{extending-subcategories-secn}, we discuss the following
question, which sheds some light on the theorem that every
$u$\+contramodule $R$\+module is an extension of
two $\R$\+contramodules.
 Given an abelian category $\sA$ and a full subcategory $\sC\subset\sA$
closed under kernels, cokernels, and finite direct sums, it is clear
that $\sC$ is an abelian category with an exact embedding functor
$\sC\rarrow\sA$.
 Consider the full subcategory $\sC_\sA^{(2)}\subset\sA$ of all objects
in $\sA$ that can be presented as extensions of two objects from~$\sA$.
 Is it true that $\sC_\sA^{(2)}$ is necessarily also an abelian category
with an exact embedding functor $\sC_\sA^{(2)}\rarrow\sA$\,?
 We show that the answer is ``yes''.

\subsection*{Acknowledgement}
 The conversations and communications with Jan \v St\!'ov\'\i\v cek and
Silvana Bazzoni were very helpful to me when working on this project.
 In particular, the conversations and communications with Silvana
provided a substantial part of my original motivation for learning
the classical material on Gabriel topologies and the rings of quotients.
 The influence of my later communications with her is felt particularly
strongly in Section~\ref{faithful-perfect-secn}, and
most of all in Remark~\ref{two-topological-rings-remark}
(and my work on Sections~\ref{perfect-of-type-lambda-secn} and
and~\ref{when-Delta=Lambda-secn} was largely motivated by the purposes
of Section~\ref{faithful-perfect-secn}).
 The influence of my conversations and communications with
Jan \v St\!'ov\'\i\v cek is felt in
Section~\ref{perfect-of-type-lambda-secn},
and particularly in Lemma~\ref{perfect-topologies-and-direct-limits},
Remark~\ref{left-countably-generated-remark}, and
Corollary~\ref{direct-limit-of-projdim1-cor}.
 I~also wish to thank Jan Trlifaj, Michal Hrbek, and Alexander Sl\'avik
for illuminating discussions.
 I~would like to thank the anonymous referee for reading the manuscript
carefully and suggesting many improvements.
 The author's research is supported by research plan RVO:~67985840
and the Israel Science Foundation grant~\#\,446/15.

\Section{Preliminaries on Topological Rings}  \label{preliminaries-secn}

 The material of Sections~\ref{prelim-linear-topologies}\+-%
\ref{prelim-sheafification} below is fairly standard;
we use the book~\cite{St} as the main reference (see also
the original sources~\cite{Gab,PG} and the Bourbaki
exercises~\cite[Exercices~II.2.16--22]{Bour}).
 The material of Sections~\ref{prelim-contramodules}\+-%
\ref{prelim-contratensor} was developed by the present
author~\cite{Psemi,Pweak,Prev,PR}.

\subsection{Linear topologies} \label{prelim-linear-topologies}
 A topological abelian group $A$ is said to have a \emph{linear topology}
if open subgroups form a base of neighborhoods of zero in~$A$.
 A linear topology on a topological group $A$ is uniquely determined
by the set $\bF$ of all open subgroups of~$A$.
 Conversely, a set $\bF$ of subgroups of an abelian group $A$ is the set
of all open subgroups in some linear topology on $A$ if and only if it
is a filter, i.~e., the following three conditions
are satisfied~\cite[Section~VI.4]{St}:
\begin{enumerate}
\renewcommand{\theenumi}{T\arabic{enumi}}
\setcounter{enumi}{-1}
\item $A\in\bF$;
\item if $V\in\bF$ and $V\subset U\subset A$, then $U\in\bF$;
\item if $U\in\bF$ and $V\in\bF$, then $U\cap V\in\bF$.
\end{enumerate}

 All topologies considered in this paper will be linear.
 Abusing terminology, we will call $\bF$ ``a topology on~$A$''.

 A set of subgroups $\bB$ in an abelian group $A$ is said to be
a \emph{base} of (neighborhoods of zero in) a topology $\bF$ on $A$
if $\bF$ consists precisely of all the subgroups $U\in A$ for
which there exists $V\in\bB$ such that $V\subset U$.
 A set of subgroups $\bB$ in an abelian group $A$ is a base of
a topology if and only if $\bB$ is nonempty and for any two
subgroups $U$, $V\in\bB$ there exists a subgroup $W\in\bB$ such
that $W\subset U\cap V$.

 A topological abelian group $A$ is said to be \emph{separated} if
the intersection of all its open subgroups is the zero subgroup,
that is $\bigcap_{U\in\bF}U=0$.
 In other words, it means that the natural abelian group homomorphism
$\lambda_{A,\bF}\:A\rarrow\varprojlim_{U\in\bF}A/U$ is injective.
 A topological abelian group $A$ is said to be \emph{complete} if
the map $\lambda_{A,\bF}$ is surjective.

 The abelian group $\A=\varprojlim_{U\in\bF}A/U$ is called
the \emph{completion} of the topological abelian group~$A$.
 The abelian group $\A$ is endowed with the projective limit topology
$\bfF$ consisting of all the subgroups $\U_U\subset\A$ of the form
$\U_U=\ker(\A\to A/U)$, where $U\in\bF$ and $\A\rarrow A/U$ is
the natural projection map.

 A topological abelian group $A$ is separated and complete if and only
if the natural homomorphism of topological groups $A\rarrow\A$ from $A$
to its completion $\A$ is an isomorphism of topological groups.
 For any topological abelian group $A$, its completion $\A$ is
separated and complete in the projective limit topology.

\subsection{Right linear topologies} \label{prelim-right-linear}
 A topological ring $R$ is said to have a \emph{right linear topology}
if open right ideals form a base of neighborhoods of zero in~$R$.
 A set $\bF$ of right ideals in an associative ring $R$ is the set of
open right ideals in a right linear topology of $R$ if and only if, in
addition to the conditions (T0\+-T2), it also satisfies the following
condition~\cite[Section~VI.4]{St}:
\begin{enumerate}
\renewcommand{\theenumi}{T\arabic{enumi}}
\setcounter{enumi}{2}
\item if a right ideal $I\subset R$ belongs to $\bF$ and $s\in R$,
then the right ideal
$$
 (I:s)=\{r\in R\mid sr\in I\}\subset R
$$
belongs to~$\bF$.
\end{enumerate}

 For any topological ring $R$ with a right linear topology $\bF$,
the topological abelian group $\R=\varprojlim_{I\in\bF}R/I$ with its
projective limit topology $\bfF$ has a unique topological ring structure
such that the natural map $R\rarrow\R$ is a ring homomorphism.
 Given two elements $\mathfrak s=(s_I\in R/I)_{I\in\bF}$ and
$\mathfrak r=(r_I\in R/I)_{I\in\bF}$ in $\R$, in order to compute
the $I$\+component $t_I$ of their product $\mathfrak t=\mathfrak s
\mathfrak r\in\R$, one chooses preimages $\tilde s_I\in R$ and
$\tilde r_I\in R$ of the elements $s_I$ and $r_I\in R/I$, and
considers the ideal $J=(I:\tilde s_I)\in\bF$.
 Then one puts $t_I=\tilde s_I\tilde r_J+I$.
 The topology $\bfF$ on the topological ring $\R$ is right linear.

\subsection{Discrete modules} \label{prelim-discrete}
 Let $R$ be a topological ring with a right linear topology~$\bF$.
 Then a right $R$\+module $\N$ is said to be \emph{discrete} if for
every element $b\in\N$ its annihilator $I_b=\{r\in R\mid br=0\}
\subset R$ belongs to~$\bF$.
 Equivalently, this means that the right action map $\N\times R\rarrow
\N$ is continuous as a function of two variables with respect to
the given topology on $R$ and the discrete topology on~$\N$.
 The full subcategory of all discrete right $R$\+modules $\discr R$
is closed under subobjects, quotient objects, and infinite direct
sums in the abelian category of right $R$\+modules $\modr R$.
 It follows that $\discr R$, just like $\modr R$, is a Grothendieck
abelian category.

 The $R$\+module structure of any discrete right $R$\+module $\N$ can be
extended in a unique way to a discrete right module structure over
the completion $\R=\varprojlim_{I\in\bF}R/I$ of the ring~$R$.
 Hence the categories of discrete right $R$\+modules and discrete
right $\R$\+modules are naturally equivalent (in fact, isomorphic),
$\discr R\simeq\discr\R$.

\subsection{Gabriel topologies} \label{prelim-gabriel-topologies}
 Let $\sA$ be an abelian category with infinite products and coproducts
in which subobjects of any given object form a set.
 A pair of classes of objects $\sT$, $\sF\subset\sA$ is called
a \emph{torsion pair} if an object $T\in\sA$ belongs to $\sT$ if and
only if $\Hom_\sA(T,F)=0$ for all $F\in\sF$, and an object $F\in\sA$
belongs to $\sF$ if and only if $\Hom_\sA(T,F)=0$ for all $T\in\sT$.
 The class $\sT$ in a torsion pair $(\sT,\sF)$ is called
the \emph{torsion class}, and the class $\sF$ is called
the \emph{torsion-free class}.

 A class of objects $\sT\subset\sA$ is the torsion class of some
torsion pair in $\sA$ if and only if it is closed under quotients,
extensions, and coproducts.
 Similarly, a class of objects $\sF\subset\sA$ is the torsion-free
class of some torsion pair in $\sA$ if and only if it is closed under
subobjects, extensions, and products~\cite[Section~VI.2]{St}.

 Given a torsion pair $(\sT,\sF)$ in $\sA$, for every object $A\in\sA$
there exists a (unique and functorial) short exact sequence $0\rarrow T
\rarrow A\rarrow F\rarrow 0$ in $\sA$ with $T\in\sT$ and $F\in\sF$.
 The object $T$ is the maximal torsion subobject of $A$, and the object
$F$ is the maximal torsion-free quotient of~$A$.

 A class of objects $\sP\subset\sA$ is said to be a \emph{pretorsion
class} if it is closed under quotients and coproducts.
 A (pre)torsion class $\sP$ is said to be \emph{hereditary} if it is
closed under subobjects.
 A torsion pair $(\sT,\sF)$ is said to be hereditary if its
torsion class $\sT$ is hereditary.

 Let $R$ be an associative ring.
 For any right linear topology $\bF$ on $R$, the class of all discrete
right $R$\+modules $\discr R$ with respect to the topology $\bF$ is
a hereditary pretorsion class in the abelian category of right
$R$\+modules $\modr R$.
 Conversely, for any hereditary pretorsion class $\sP\subset\modr R$
there exists a unique right linear topology $\bF$ on $R$ such that
$\sP=\discr R$.
 Given a hereditary pretorsion class $\sP$ in $\modr R$, the topology
$\bF$ can be recovered as the set of all right ideals $I\subset R$
such that $R/I\in\sP$, while given a right linear topology $\bF$ on $R$,
the hereditary pretorsion class $\sP$ is defined as the class of all
right $R$\+modules $\N$ such for every element $b\in\N$ the annihilator
of~$b$ in $R$ belongs to~$\bF$ \cite[Section~IV.4]{St}.

 A right linear topology $\bG$ on $R$ is called a (right) \emph{Gabriel
topology} if the related hereditary pretorsion class $\sT=\discr R$
is a torsion class (i.~e., it is closed under extensions in $\modr R$).
 Thus right Gabriel topologies on $R$ correspond bijectively to
hereditary torsion classes in $\modr R$.
 A right linear topology $\bG$ on $R$ is a Gabriel topology if and only
it satisfies the following condition~\cite[Section~IV.5]{St}:
\begin{enumerate}
\renewcommand{\theenumi}{T\arabic{enumi}}
\setcounter{enumi}{3}
\item if $I\subset R$ is a right ideal and there exists a right ideal
$J\in\bG$ such that $(I:s)\in\bG$ for all $s\in J$, then $I\in\bG$.
\end{enumerate}

 The right $R$\+modules belonging to the hereditary torsion class $\sT$
corresponding to a Gabriel topology $\bG$ on $R$ are said to be
\emph{$\bG$\+torsion}, and the right $R$\+modules belonging to
the related torsion-free class $\sF\subset\modr R$ are called
\emph{$\bG$\+torsion-free}.
 So the words ``$\bG$\+torsion right $R$\+module'' are synonymous with
``discrete right $R$\+module with respect to the $\bG$\+topology
on~$R$''.

 Given a topological ring $R$ with a right linear topology $\bG$,
one can consider its completion $\R=\varprojlim_{I\in\bG}R/I$ and view
it as a topological ring in the projective limit topology~$\bfG$.
 If $\bG$ is a Gabriel topology on $R$, then $\bfG$ is a Gabriel
topology on~$\R$.

\subsection{Localization} \label{prelim-localization}
 Another name for a hereditary torsion class is a \emph{localizing
subcategory}.
 Let $\sA$ be an abelian category with exact coproduct functors in which
subobjects of any given object form a set, and let $\sT\subset\sA$ be
a full subcategory closed under subobjects, quotients, extensions, and
coproducts.
 The \emph{quotient category} $\sA/\sT$ is the category whose objects
are the objects of $\sA$ and morphisms $A\rarrow B$ can be equivalently
described as the equivalence classes of
\begin{enumerate}
\renewcommand{\theenumi}{\alph{enumi}}
\item morphisms $A'\rarrow B'$ in $\sA$, where $A'$ is a subobject in
$A$ with $A/A'\in\sT$ and $B'=B/T$ is a quotient object of $B$ by
a subobject $T\in\sT$; or
\item fractions $A\rarrow B'\larrow B$ in $\sA$, where both the kernel
and cokernel of the morphism $B\rarrow B'$ belong to~$\sT$; or
\item fractions $A\larrow A'\rarrow B$ in $\sA$, where both the kernel
and cokernel of the morphism $A'\rarrow A$ belong to~$\sT$.
\end{enumerate}

 The quotient category $\sA/\sT$ is an abelian category with (exact)
coproducts.
 The natural functor $\sA\rarrow\sA/\sT$ is exact and
preserves coproducts.
 Hence, assuming that the category $\sA$ has a set of generators,
the functor $\sA\rarrow\sA/\sT$ has a right adjoint functor
$\sA/\sT\rarrow\sA$.
 The functor $\sA/\sT\rarrow\sA$ is fully faithful, so it allows to
view the quotient category $\sA/\sT$ as a full subcategory in~$\sA$.

 We denote by $L$ the composition of functors $\sA\rarrow\sA/\sT
\rarrow\sA$.
 The functor $L$ is the reflector onto the full subcategory $\sA/\sT
\subset\sA$. 
 For every object $A\in\sA$, there is a natural adjunction morphism
$A\rarrow L(A)$ inducing an isomorphism of the Hom groups
$\Hom_\sA(L(A),B)\simeq\Hom_\sA(A,B)$ for every object
$B\in\sA/\sT\subset\sA$.

 Let us now consider the particular case when $\sA=\modr R$ is
the category of right modules over an associative ring~$R$.
 Then the hereditary torsion class $\sT$ corresponds to a right Gabriel
topology $\bG$ on~$R$.
 Let $L_\bG\:\modr R\rarrow\modr R$ denote the related
localization functor.

 The right $R$\+modules belonging to the full subcategory
$(\modr R)/\sT\subset\modr R$ (that is, to the essential image of
the functor~$L_\bG$) are said to be \emph{$\bG$\+closed}.
 A right $R$\+module $B$ is $\bG$\+closed if and only if any
right $R$\+module morphism $A'\rarrow A''$ with $\bG$\+torsion kernel
and cokernel induces an isomorphism $\Hom_R(A'',B)\simeq\Hom_R(A',B)$.
 So any $\bG$\+closed right $R$\+module is $\bG$\+torsion-free.

 Consider the free right $R$\+module $R$ with one generator.
 Applying the functor $L_\bG$, we obtain a right $R$\+module $L_\bG(R)$.
 The left action of $R$ by right $R$\+module endomorphisms of
the right $R$\+module $R$ induces a left action of $R$ in $L_\bG(R)$,
making it an $R$\+$R$\+bimodule.
 The adjunction map $R\rarrow L_\bG(R)$ is a morphism of
$R$\+$R$\+bimodules.

 There exists a unique associative ring structure on $L_\bG(R)$ 
making the map $R\rarrow L_\bG(R)$ an associative ring homomorphism
in a way compatible with the $R$\+$R$\+bimodule structure on $L_\bG(R)$.
 Moreover, for any right $R$\+module $N$, the right $R$\+module
structure on $L_\bG(N)$ extends uniquely to a right $L_\bG(R)$\+module
structure.
 The ring $L_\bG(R)$ is the universal ring acting on the right on
(the underlying abelian groups of) all the $\bG$\+closed right
$R$\+modules, and any right $R$\+linear morphism between
$\bG$\+closed right $R$\+modules is right $L_\bG(R)$\+linear.

\subsection{Sheafification construction} \label{prelim-sheafification}
 Let $R$ be an associative ring and $\bG$ be a Gabriel topology of
right ideals in~$R$.
 The localization functor $L_\bG$ has the following explicit
construction, which is essentially a particular case of
the ``additive sheaf theory'' \cite[Section~IX.1]{St}.

 Let $N$ be a right $R$\+module.
 Consider the inductive limit of abelian groups
$$
 N_{(\bG)}=\varinjlim\nolimits_{I\in\bG}\Hom_{R^\rop}(I,N),
$$
where $R^\rop$ is the ring opposite to $R$ and $\Hom_{R^\rop}(K,N)$
denotes the group of morphisms $K\rarrow N$ in the category of
right $R$\+modules $\modr R$.
 The diagram is indexed by the partially ordered set of all right ideals
$I\in\bG$ with respect to the inverse inclusion, and for any two right
ideals $J\subset I$, \ $I$, $J\in\bG$ the map $\Hom_{R^\rop}(I,N)\rarrow
\Hom_{R^\rop}(J,N)$ is induced by the inclusion morphism $J\rarrow I$.
 According to the condition~(T2), this diagram is indexed by
a directed poset~$\bG$.

 By the definition of the inductive limit, the construction above means
that elements of $N_{(\bG)}$ are represented by right $R$\+module
morphisms $I\rarrow N$, where $I\in\bG$.
 Two such morphisms $I\overset\alpha\rarrow N$ and $J\overset\beta
\rarrow N$, where $I$, $J\in\bG$, represent the same element of
$N_{(\bG)}$ if and only if there exists a right ideal $K\in\bG$ such
that $K\subset I\cap J$ and $\alpha|_K=\beta|_K$.

 To endow the abelian group $N_{(\bG)}$ with a natural right $R$\+module
structure, consider an element $\beta\in N_{(\bG)}$ represented by
a right $R$\+module morphism $I\rarrow N$, and an element $s\in R$.
 According to the condition~(T3), we have $(I:s)\in\bG$.
 Let $\beta s\in N_{(\bG)}$ be the element represented by
the composition
$$
 (I:s)\overset s\lrarrow I\overset\beta\lrarrow N,
$$
where $(I:s)\overset s\rarrow I$ is the right $R$\+module morphism
of left multiplication with~$s$.
 There is a natural morphism of right $R$\+modules
$\varkappa_N\:N\rarrow N_{(\bG)}$ assigning to an element $b\in N$
the element of $N_{(\bG)}$ represented by the right $R$\+module morphism
$R\overset b\rarrow N$ (where $R\in\bG$ by the condition~(T0)).

 The kernel of the morphism $N\rarrow N_{(\bG)}$ coincides with
the maximal $\bG$\+torsion submodule $t_\bG(N)$ of the right
$R$\+module~$N$.
 The cokernel $N_{(\bG)}/\im\varkappa_N$ is $\bG$\+torsion, since for
any element $\beta\in N_{(\bG)}$ represented by a right $R$\+module
morphism $I\rarrow N$ one has $\beta s\in\im\varkappa_N$
for all $s\in I$.
 Using the condition~(T4), one shows that the right $R$\+module
$N_{(\bG)}$ is $\bG$\+torsion-free for any right $R$\+module~$N$.

 For any right $R$\+module $N$ the two maps $\varkappa_{N_{(\bG)}}$
and $(\varkappa_N)_{(\bG)}\:N_{(\bG)}\rarrow N_{(\bG)(\bG)}$ coincide,
since so do their compositions with the map $\varkappa_N$, whose
cokernel is $\bG$\+torsion, while $N_{(\bG)(\bG)}$
is $\bG$\+torsion-free.
 The functor $N\longmapsto N_{(\bG)}$ is left exact by construction,
hence the map $N_{(\bG)}\rarrow N_{(\bG)(\bG)}$ is an isomorphism for
any $\bG$\+torsion-free right $R$\+module~$N$ \cite[Lemma~IX.1.6]{St}.
 For any $\bG$\+torsion-free right $R$\+module $N$, the right
$R$\+module $N_{(\bG)}$ is $\bG$\+closed.
 The functor $N\longmapsto L_\bG(N)$ can be computed in two
alternative ways as
$$
 L_\bG(N)=(N/t_\bG(N))_{(\bG)}=N_{(\bG)(\bG)}
$$
for any right $R$\+module~$N$.

 Following the notation in~\cite{St}, when the $R$\+$R$\+bimodule
$L_\bG(R)$ is viewed as a ring, it is denoted by $R_\bG$, and when
the right $R$\+module $L_\bG(N)$ is viewed as a right $R_\bG$\+module,
it is denoted by~$N_\bG$.
 So one has $N_\bG=N_{(\bG)(\bG)}$.
 The ring $R_\bG$ is called the \emph{ring of quotients} of
the ring $R$ with respect to the Gabriel topology~$\bG$.
 The right $R_\bG$\+module $N_\bG$ is called the \emph{module of
quotients} of a right $R$\+module~$N$.

 The direct limit/``sheafification'' construction above can be used
to describe explicitly the multiplication in $R_\bG$ and the right
action of $R_\bG$ in~$N_\bG$.
 Let us discuss the ring structure (the construction of the right
action being similar).

 Let $\sigma\:R\rarrow S$ be a ring homomorphism such that both
the kernel and the cokernel of~$\sigma$ are $\bG$\+torsion right
$R$\+modules.
 Assume that either the map~$\sigma$ is injective, or otherwise $S$ is
a $\bG$\+torsion-free right $R$\+module.
 Our aim is to construct an associative ring structure on the right
$R$\+module (or $S$\+$R$\+bimodule) $S_{(\bG)}$ making
$\varkappa_S\:S\rarrow S_{(\bG)}$ a ring homomorphism.

 Let $\alpha$ and $\beta\in S_{(\bG)}$ be two elements represented by
right $R$\+module morphisms $I\overset\alpha\rarrow S$ and $J\overset
\beta\rarrow S$, where $I$ and $J\subset R$ are two right ideals
belonging to~$\bG$.
 Denote by $\iota=\iota_I\:I\rarrow R$ and $\iota_J\:J\rarrow R$
the identity inclusion maps.
 Consider the composition of the inclusion $\iota_I\:I\rarrow R$ with
the map $\sigma\:R\rarrow S$, and denote by $\widetilde K$ the fibered
product of the right $R$\+modules $I$ and $J$ over $S$ with respect to
the morphisms $I\overset\iota\rarrow R\overset\sigma\rarrow S$
and $J\overset\beta\rarrow S$.
 Let $\tilde\gamma\:\widetilde K\rarrow I$ and $\tilde\eta\:\widetilde K
\rarrow J$ be the two related morphisms of right $R$\+modules.

 The cokernel of the morphism $\tilde\eta\:\widetilde K\rarrow J$ is
a submodule of the cokernel of the composition $I\rarrow R\rarrow S$,
which is an extension of the cokernel of the morphism $\sigma\:R
\rarrow S$ and of a certain quotient module of~$R/I$.
 Since the class of all $\bG$\+torsion right $R$\+modules is closed
under submodules, quotients, and extensions, the cokernel of
the morphism $\tilde\eta\:\widetilde K\rarrow J$ is $\bG$\+torsion.
 The cokernel of the composition $\widetilde K\rarrow J\rarrow R$
of the morphism~$\tilde\eta$ with the inclusion $\iota_J\: J\rarrow R$
is an extension of $R/J$ and the cokernel of~$\tilde\eta$;
hence it is $\bG$\+torsion, too.

 The kernel of the morphism $\tilde\eta\:\widetilde K\rarrow J$ is
a submodule of the kernel of the composition $I\rarrow R\rarrow S$,
which is a submodule of $\ker\sigma$.
 Hence the kernel of~$\tilde\eta$ is also $\bG$\+torsion.
 Consider the composition
$$
 \widetilde K\overset{\tilde\gamma}\lrarrow I\overset\alpha\lrarrow S.
$$
 If the morphism~$\sigma$ is injective, then so is the morphism
$\tilde\eta\:\widetilde K\rarrow J$.
 Therefore, $K$ defined as $K=\widetilde K$ is
a right $R$\+submodule in $J$, hence a right ideal in~$R$.
 If the right $R$\+module $S$ is $\bG$\+torsion-free, then the morphism
$\alpha\widetilde\gamma\:\widetilde K\rarrow S$ annihilates
the kernel of~$\tilde\eta$.
 In this case, we set $K$ to be the image of the morphism
$\tilde\eta\:\widetilde K\rarrow J$; so, once again, $K\subset J$
is a right ideal in~$R$.
 In both cases, the quotient module $R/K$, which coincides with
the cokernel of the morphism~$\tilde\eta$, is $\bG$\+torsion;
hence $K\in\bG$.

 In both cases, we have a right $R$\+module morphism $K\rarrow S$
induced by $\alpha\tilde\gamma$.
 This morphism represents the desired element
$\alpha\beta\in S_{(\bG)}$.

 Applying this construction for the first time to the ring $R=S$
and the identity morphism $\sigma=\id_R$, we construct
a ring structure on the $R$\+$R$\+bimodule~$R_{(\bG)}$.
 Applying the same construction for the second time to the ring
$S=R_{(\bG)}$ and the ring homomorphism $\sigma=\varkappa_R$,
we obtain a ring structure on the $R$\+$R$\+bimodule
$R_\bG=R_{(\bG)(\bG)}$, which was our aim.

 Alternatively, notice that the maximal $\bG$\+torsion submodule
$t_\bG(R)\subset R$ of the right $R$\+module $R$ is a two-sided ideal
in the ring~$R$.
 Applying the construction above to the ring $S=R/t_\bG(R)$ and
the natural surjective ring homomorphism $\sigma\:R\rarrow S$,
one obtains the same ring structure on the $R$\+$R$\+bimodule
$R_\bG=(R/t_\bG(R))_{(\bG)}$.

\subsection{Contramodules} \label{prelim-contramodules}
 The following constructions and definitions go back to
the book~\cite[Remark~A.3]{Psemi},
the memoir~\cite[Section~1.2]{Pweak},
the overview~\cite[Sections~2.1 and~2.3]{Prev},
and the paper~\cite[Sections~1.2 and~5]{PR}.

 For any abelian group $A$ and a set $X$, we will use the notation
$A[X]=A^{(X)}$ for the direct sum of $X$ copies of~$A$.
 Alternatively, one can view $A[X]$ as the group of all finite formal
linear combinations $\sum_{x\in X}a_xx$ of elements of $X$ with
the coefficients in~$A$.
 For any map of sets $X\rarrow Y$, one has the obvious induced
(``pushforward'') map $A[X]\rarrow A[Y]$; so $X\longmapsto A[X]$ is
a covariant functor $A[{-}]\:\Sets\rarrow\Ab$ from the category of
sets to the category of abelian groups.

 In particular, if $R$ is an associative ring, then $R[X]$ is
a notation for the free left $R$\+module with $X$~generators.
 One can also consider the underlying set of the abelian group $R[X]$
and view it as an abstract set.
 Then, for any associative ring $R$ and a set $X$, there are natural
maps of sets
$$
 \epsilon_{R,X}\:X\rarrow R[X] \quad\text{and}\quad
 \phi_{R,X}\:R[R[X]]\rarrow R[X].
$$
 Here $\epsilon_{R,X}$ is the ``point measure'' map assigning to
an element $x_0\in X$ the corresponding generator of the free left
$R$\+module $R[X]$, or in other words, the formal linear combination
$\sum_{x\in X}r_xx$ with $r_{x_0}=1$ and $r_x=0$ for $x\ne x_0$.
 The map $\phi_{R,X}$ is the ``opening of parentheses'' map assigning
to a formal linear combination of formal linear combinations of
elements of $X$ with the coefficients in $R$ a formal linear
combination of elements of $X$ with the coefficients in $R$ (using
both the additive and the multiplicative structures on~$R$).

 The natural transformations $\epsilon_R$ and $\phi_R$ endow
the functor
$$
 \bM_R=R[{-}]\:\Sets\lrarrow\Sets
$$
with the structure of a \emph{monad} on the category of sets.
 The category of left $R$\+modules can be defined as the category
of algebras (or, in the additive language that we prefer as more
suitable to our context, ``modules'') over the monad~$\bM_R$.
 In other words, the datum of a left $R$\+module structure on a set
$M$ is equivalent to the datum of a map of sets $\mu_M\:R[M]\rarrow M$
satisfying the associativity and unity equations of an algebra/module
over the monad~$\bM_R$.

 Now let $\A$ be a complete, separated topological abelian group (with
a linear topology~$\bfF$).
 Then we denote by $\A[[X]]$ the abelian group
$$
 \A[[X]]=\varprojlim\nolimits_{\U\in\bfF}(\A/\U)[X],
$$
where the projective limit is taken over all the open subgroups
$\U\subset\A$.
 Alternatively, $\A[[X]]$ can be defined as the group of all infinite
formal linear combinations $\{\sum_{x\in X}a_xx\}$ with the coefficients
$a_x\in\A$ forming a family of elements of $\A$ converging to zero
in the topology $\bfF$ of~$\A$.
 The latter condition means that for every $\U\in\bfF$ the set of all
$x\in X$ for which $a_x\notin\U$ must be finite.

 For any map of sets $f\:X\rarrow Y$, there is the induced
``push-forward'' homomorphism
$$
 \A[[f]]\:\A[[X]]\lrarrow\A[[Y]]
$$
taking an infinite formal linear combination $\sum_{x\in X}a_xx$
to the infinite formal linear combination $\sum_{y\in Y}b_yy$ with
the coefficients
$$
 b_y=\sum_{x:f(x)=y}a_x\in\A.
$$
 Here the infinite sum of elements of $\A$ is understood as
the limit of finite partial sums in the (complete and separated,
by assumption) topology $\bfF$ of~$\A$.
 Thus we have a functor $\A[[{-}]]\:\Sets\rarrow\Ab$ from the category
of sets to the category of abelian groups corresponding to any complete,
separated topological abelian group~$\A$.

 Let $\R$ be a complete, separated topological associative ring with
a right linear topology.
 For any set $X$, we have a natural ``point measure'' map
$$
 \epsilon_{\R,X}\:X\lrarrow\R[[X]],
$$
which can be constructed as the composition $X\rarrow\R[X]
\rarrow\R[[X]]$.
 Furthermore, an ``opening of parentheses'' map
$$
 \phi_{\R,X}\:\R[[\R[[X]]]]\lrarrow\R[[X]]
$$
can be constructed using the multiplication of pairs of elements in $\R$
and the formation of infinite sums, computed as the limits of finite
partial sums in the topology of~$\R$.
 The assumption that the topology on $\R$ is right linear guarantees
convergence~\cite[Section~2.1]{Prev}, \cite[Section~5]{PR}.

 As in the case of a discrete ring $R$ above, the natural
transformations $\epsilon_\R$ and $\phi_\R$ defined for a complete,
separated topological ring $\R$ with a right linear topology, endow
the functor
$$
 \bM_\R=\R[[{-}]]\:\Sets\lrarrow\Sets
$$
with the structure of a monad on the category of sets.
 By the definition, a \emph{left\/ $\R$\+contramodule} is
an algebra/module over this monad.
 In other words, a left $\R$\+contramodule $\C$ is a set endowed
with a \emph{left contraaction} map
$$
 \pi_\C\:\R[[\C]]\lrarrow\C,
$$
which must be a map of sets satisfying the following equations.
 The two composition
$$
 \R[[\R[[\C]]]]\birarrow\R[[\C]]\rarrow\C
$$
of the maps $\phi_{\R,\C}$ and $\R[[\pi_\C]]\:\R[[\R[[\C]]]]\birarrow
\R[[\C]]$ with the map $\pi_\C\:\R[[\C]]\rarrow\C$ should be equal
to each other; and the composition
$$
 \C\rarrow\R[[\C]]\rarrow\C
$$
of the map $\epsilon_{\R,\C}\:\C\rarrow\R[[\C]]$ with the map
$\pi_\C\:\R[[\C]]\rarrow\C$ should be equal to
the identity map~$\id_\C$.
 A \emph{morphism} of left $\R$\+contramodules $f\:\C\rarrow\D$ is
a morphism of algebras/modules over the monad $\bM_\R$; in other words,
it is a map of sets forming a commutative square diagram with
the contraaction maps $\pi_\C$, $\pi_\D$ and the push-forward
map $\R[[f]]\:\R[[\C]]\rarrow\R[[\D]]$.

 The category of left $\R$\+contramodules $\R\contra$ is abelian.
 For any left $\R$\+contra\-module $\C$, the composition $\R[\C]\rarrow
\R[[\C]]\rarrow\C$ of the natural inclusion $\R[\C]\rarrow\R[[\C]]$
with the contraaction map $\pi_\C\:\R[[\C]]\rarrow\C$ defines
a map $\mu_\C\:\R[\C]\rarrow\C$ endowing $\C$ with a left
$\R$\+module structure.
 This defines a natural forgetful functor $\R\contra\rarrow\R\modl$,
which is exact and preserves infinite
products~\cite[Section~1.2]{Pweak}, \cite[Section~1.1]{PR},
\cite[Lemma~1.1]{Pper}.

 For any set $X$, the set $\R[[X]]$ has a natural left
$\R$\+contramodule structure with the contraaction map
$\pi_{\R[[X]]}=\phi_X$.
 The left $\R$\+contramodule $\R[[X]]$ is called the \emph{free}
left $\R$\+contramodule generated by a set~$X$.
 For any left $\R$\+contramodule $\D$, there is a natural isomorphism
of abelian groups
$$
 \Hom^\R(\R[[X]],\D)\simeq\Hom_{\Sets}(X,\D)=\D^X,
$$
where $\Hom^\R(\C,\D)$ denotes the abelian group of morphisms
$\C\rarrow\D$ in the category $\R\contra$.
 It follows that free left $\R$\+contramodules are projective objects
in $\R\contra$.
 There are also enough of them; so there are enough projectives in
$\R\contra$ and a left $\R$\+contramodule is projective if and only if
it is a direct summand of a free one.
 The free left $\R$\+contramodule with one generator $\R=\R[[*]]$ is
a projective generator of $\R\contra$.
 If $\lambda$~denotes the cardinality of a base $\bfB$ of
the topology $\bfF$ on $\R$, then the category $\R\contra$ is locally
$\lambda^+$\+presentable and its projective generator $\R$ is
$\lambda^+$\+presentable~\cite[Sections~1.2 and~5]{PR}.

\subsection{Contratensor product} \label{prelim-contratensor}
 Let $\R$ be a complete, separated topological associative ring with
a right linear topology.
 Let $A$ be an associative ring, and let $\N$ be an $A$\+$\R$\+bimodule
which is discrete as a right $\R$\+module.
 Let $V$ be a left $A$\+module.
 Then the Hom group $\D=\Hom_A(\N,V)$ has a natural left
$\R$\+contramodule structure (extending its familiar left $\R$\+module
structure).
 Given an element $\mathfrak r=\sum_{d\in\D}r_dd\in\R[[\D]]$, one
defines its image $\pi_\D(\mathfrak r)\in\D$ under the left
contraaction map~$\pi_\D$ by the rule
$$
 \pi_\D(\mathfrak r)(b)=\sum_{d\in\D}d(br_d)\in V
$$
for all $b\in\N$, where the sum in the right-hand side is finite
because the annihilator of $b$ is open in~$\R$ and the family of
elements $(r_d\in\R)_{d\in\D}$ converges to zero in~$\R$.

 The \emph{contratensor product} $\N\ocn_\R\C$ of a discrete right
$\R$\+module $\N$ and a left $\R$\+contramodule $\C$ is an abelian
group constructed as the cokernel of (the difference of)
the following pair of maps
$$
 \N\ot_\boZ\R[[\C]]\birarrow\N\ot_\boZ\C.
$$
 The first map $\N\ot_\boZ\R[[\C]]\rarrow\N\ot_\boZ\C$ is obtained by
applying the functor $\N\ot_\boZ{-}$ to the contraaction map
$\pi_\C\:\R[[\C]]\rarrow\C$.
 The second map $\N\ot_\boZ\R[[\C]]\rarrow\N\ot_\boZ\C$ is defined
by the rule
$$
 b\ot\sum_{c\in\C}r_cc\longmapsto\sum_{c\in\C}br_c\ot c\in\N\ot_\boZ\C
$$
for any $b\in\N$ and $\mathfrak r=\sum_{c\in\C}r_cc\in\R[[\C]]$.
 Here, once again, the sum in the right-hand side is finite because
the annihilator of $b$ is open in~$\R$ and the family of
elements $(r_c\in\R)_{c\in\C}$ converges to zero in~$\R$.

 The conventional tensor product $\N\ot_\R\C$ of the right $\R$\+module
$\N$ and the underlying left $\R$\+module of the left $\R$\+contramodule
$\C$ can be constructed as the cokernel of the pair of maps
$\N\ot_\boZ\R[\C]\birarrow\N\ot_\boZ\C$ obtained as the composition
$\N\ot_\boZ\R[\C]\rarrow\N\ot_\boZ\R[[\C]]\birarrow\N\ot_\boZ\C$.
 Hence there is a natural surjective map of abelian groups
$$
 \N\ot_\R\C\lrarrow\N\ocn_\R\C.
$$

 For any left $A$\+$\R$\+bimodule $\N$ which is discrete as a right
$\R$\+module, left $\R$\+contramodule $\C$, and a left $A$\+module $V$
there is a natural isomorphism of abelian groups
$$
 \Hom_A(\N\ocn_\R\C,\>V)\simeq\Hom^\R(\C,\Hom_A(\N,V)),
$$
where the left $A$\+module structure on the contratensor product
$\N\ocn_\R\C$ is induced by the left $A$\+module structure on~$\N$.
 For any discrete right $\R$\+module $\N$ and any set $X$, there is
a natural isomorphism of abelian groups
$$
 \N\ocn_\R\R[[X]]\simeq\N[X].
$$

 For any associative ring $R$, any subgroup $A\subset R$, and any
left $R$\+module $M$, we denote, as usual, by $AM=A\cdot M$
the subgroup in $M$ spanned by all the elements $am$ with $a\in A$
and $m\in M$.
 Alternatively, the subgroup $AM\subset M$ can be defined as the image
of the composition $A[M]\rarrow R[M]\rarrow M$ of the natural inclusion
$A[M]\rarrow R[M]$ with the (monad) action map $\mu_M\:R[M]\rarrow M$.

 For any complete, separated topological associative ring $\R$ with
a right linear topology, any closed subgroup $\A\subset\R$, and any
left $\R$\+contramodule $\C$, we denote by $\A\tim\C\subset\C$
the image of the composition
$$
 \A[[\C]]\lrarrow\R[[\C]]\lrarrow\C
$$
of the natural inclusion $\A[[\C]]\lrarrow\R[[\C]]$ with
the contraaction map $\pi_\C\:\R[[\C]]\rarrow\C$.
 So $\A\tim\C$ is a subgroup in~$\C$.
 Clearly, one has $\A\C\subset\A\tim\C$.

 Let $\J\subset\R$ be a closed right ideal.
 Then for any set $X$ one has
$$
 \J\tim\R[[X]]=\J[[X]]\subset\R[[X]].
$$

 Let $\I\subset\R$ be an open right ideal.
 Then $\R/\I$ is a discrete right $\R$\+module and, for any left
$\R$\+contramodule $\C$, the quotient group $\C/(\I\tim\C)$ can
be computed as the contratensor product
$$
 \C/(\I\tim\C)\simeq(\R/\I)\ocn_\R\C.
$$

\Section{Perfect Gabriel Topologies of Type~$\lambda$}
\label{perfect-of-type-lambda-secn}

 A homomorphism of associative rings $u\:R\rarrow U$ is called
a \emph{ring epimorphism} if it is an epimorphism in the category of
associative rings, i.~e., for any two homomorphisms of associative
rings $v'$, $v''\:U\rarrow V$ the equation $v'u=v''u$ implies $v'=v''$.
 Equivalently, this means that the multiplication map
$U\otimes_RU\rarrow U$ is an isomorphism of
$R$\+$R$\+bimodules~\cite[Proposition~XI.1.2]{St}.
 If this is the case, then the two $R$\+$R$\+bimodule morphisms
$U\birarrow U\otimes_RU$ induced by~$u$ are also isomorphisms and,
in addition, are equal to each other.
 An associative ring homomorphism $u\:R\rarrow U$ is a ring epimorphism
if and only if the induced functor of restriction of scalars for
left modules $u_*\:U\modl\rarrow R\modl$ is fully faithful, or
equivalently, the similar functor for right modules $u_*\:\modr U\rarrow
\modr R$ is fully faithful.

 A \emph{homological} ring epimorphism is a ring epimorphism~$u$
such that $\Tor^R_i(U,U)=0$ for all $i>0$.
 A ring epimorphism~$u$ is said to be \emph{left flat} if $U$ is a flat
left $R$\+module.
 Obviously, any (left or right) flat ring epimorphism is homological.

 Let $u\:R\rarrow U$ be a left flat ring epimorphism.
 Then the functor of extension of scalars $u^*\:\modr R\rarrow\modr U$
taking a right $R$\+module $N$ to the right $U$\+module $u^*(N)=N\ot_RU$
is an exact functor left adjoint to the fully faithful functor
$u_*\:\modr U\rarrow\modr R$.
 It follows that $\modr U$ is the quotient category of $\modr R$ by
the localizing subcategory (hereditary torsion class) $\sT$ of all right
$R$\+modules annihilated by~$u^*$, that is,
$\modr U\simeq(\modr R)/\sT$.

 The corresponding Gabriel topology $\bG$ on $R$ is the set of all right
ideals $I\subset R$ such that $R/I\ot_RU=0$.
 The localization functor $L_\bG\:\modr R\rarrow\modr R$ is isomorphic
to the functor of tensor product ${-}\ot_RU$.
 Hence the ring $U$ together with the ring homomorphism~$u$ can be
recovered from the Gabriel topology $\bG$ on $R$ using the construction
of the ring of quotients $R_\bG$ of an associative ring $R$
with respect to its right Gabriel topology~$\bG$ (see
Sections~\ref{prelim-localization}\+-\ref{prelim-sheafification}).
 So there is a natural isomorphism of associative rings $U\simeq R_\bG$
forming a commutative triangle diagram with the natural homomorphisms
$R\rarrow R_\bG$ and $u\:R\rarrow U$  \cite[Theorem~XI.2.1]{St}.

 Following the terminology in~\cite{St}, right Gabriel topologies $\bG$
on $R$ corresponding to left flat ring epimorphisms $u\:R\rarrow U$ in
this way are called \emph{perfect}.
 A left $R$\+module $M$ is said to be \emph{$\bG$\+divisible} if
$T\ot_RM=0$ for all $T\in\sT$, or equivalently, $R/I\ot_RM=0$ for all
$I\in\bG$ (which means, in other words, that $IM=M$ for all $I\in\bG$).
 A right Gabriel topology $\bG$ on $R$ is perfect if and only if
the left $R$\+module $R_\bG$ is
$\bG$\+divisible~\cite[Proposition~XI.3.4]{St}.
 Any perfect Gabriel topology has a base consisting of finitely
generated right ideals.

 Let $\lambda$~be a cardinal.
 We will say that a right linear topology $\bF$ on an associative
ring $R$ is \emph{of type\/~$\lambda$} if it has a base $\bB$ of
cardinality not exceeding~$\lambda$.
 A poset $\Xi$ is said to be \emph{$\lambda^+$\+directed} if
for any its subset $\Upsilon\subset\Xi$ of the cardinality
not exceeding~$\lambda$ there exists an element $\xi\in\Xi$
such that $\xi\ge\upsilon$ for all $\upsilon\in\Upsilon$.
 Our aim in this section is to describe perfect Gabriel topologies
on associative (and particularly, commutative) rings $R$ as
$\lambda^+$\+directed unions of perfect Gabriel topologies of
type~$\lambda$.

 The following lemma is a noncommutative generalization
of~\cite[Lemma~2.3]{Hrb} (see also
the classical~\cite[Lemma~VI.5.3]{St}).

\begin{lem} \label{finitely-generated-gabriel-lemma}
 Let $R$ be an associative ring and\/ $\bG$ be a right linear topology
on $R$ with a base of finitely generated right ideals (i.~e., $\bG$
satisfies the conditions~(T0\+-T3) of
Sections~\ref{prelim-linear-topologies}\+-\ref{prelim-right-linear}).
 Then\/ $\bG$ is a Gabriel topology on $R$ (i.~e., $\bG$ satisfies
the condition~(T4) of Section~\ref{prelim-gabriel-topologies})
if and only if the following condition~(T4$'$) is satisfied:
\begin{enumerate} \upshape
\renewcommand{\theenumi}{T\arabic{enumi}$'$}
\setcounter{enumi}{3}
\item if $J\in\bG$ is a finitely generated right ideal, then for some
(or equivalently, for any) finite set of generators $s_1$,~\dots,~$s_m$
of $J$ and any right ideal $K\in\bG$ the right ideal
$s_1K+\dotsb+s_mK\subset R$ belongs to~$\bG$.
\end{enumerate}
\end{lem}

\begin{proof}
 As one can see from the arguments below, for any fixed right ideal
$J\in\bG$, the condition~(T4$'$) remains unchanged when one varies
its finite set of generators $s_1$,~\dots, $s_m$, provided that one
is allowed to vary simultaneously the ideal $K$ as well.

 ``Only if'': set $I=s_1K+\dotsb+s_mK$.
 Let $s=s_1r_1+\dotsb+s_mr_m$ be an element of~$J$ (where $r_1$,~\dots,
$r_m\in R$).
 By the axioms~(T2) and~(T3), the right ideal $H=(K:r_1)\cap\dotsb
\cap(K:r_m)$ belongs to~$\bG$.
 For any element $h\in H$, we have $sh=s_1r_1h+\dotsb+s_mr_mh\in
s_1K+\dotsb+s_mK=I$, so $H\subset(I:s)$.
 By the axiom~(T1), it follows that the right ideal $(I:s)$ belongs
to~$\bG$.
 It remains to apply~(T4) in order to conclude that $I\in\bG$.

 ``If'': since $\bG$ has a base of finitely generated ideals, without
loss of generality we can assume the right ideal $J\in\bG$ in~(T4) to be
finitely generated.
 By assumption, we have $(I:s)\in\bG$ for all $s\in J$, and
in particular for $s=s_1$, $s=s_2$,~\dots, $s=s_m$.
 Hence, by the axiom~(T2), the right ideal $K=(I:\nobreak s_1)\cap
\dotsb\cap(I:\nobreak s_m)$ belongs to~$\bG$.
 By~(T4$'$), we have $s_1K+\dotsb+s_mK\in\bG$.
 Now $s_1K+\dotsb+s_mK\subset I$ by construction, hence by~(T1)
it follows that $I\in\bG$.
\end{proof}

 Let $\bF$ be a right linear topology on an associative ring $R$ and
$\lambda$~be an infinite cardinal.
 Consider the following condition:
\begin{itemize}
\item[(T$_\lambda$)] for every right ideal $I\in\bF$, there exists
a subset of right ideals $F_I\subset\bF$ of the cardinality not
exceeding~$\lambda$ such that for every element $s\in R$ there exists
a right ideal $J\in F_I$ for which $sJ\subset I$.
\end{itemize}

\begin{exs} \label{T-lambda-examples}
 (1)~Any linear topology of ideals in a commutative ring $R$ satisfies
the condition~(T$_\omega$) for the countable cardinal~$\omega$ (hence
also~(T$_\lambda$) for any infinite cardinal~$\lambda$).
 It suffices to take $F_I=\{I\}$ for every ideal $I\in\bF$.

 (2)~More generally, any right linear topology $\bF$ with a base $\bB$
consisting of two-sided ideals in an associative ring $R$
satisfies~(T$_\omega$).
 For every $I\in\bF$, choose $J\in\bB$ such that $J\subset I$, and set
$F_I=\{J\}$.

 (3)~If the cardinality of (the underlying set of) an associative
ring~$R$ does not exceed~$\lambda$, then any right linear topology
on $R$ satisfies~(T$_\lambda$).

 (4)~More generally, let $Z\subset R$ denote the center of
an associative ring~$R$.
 Let $\bF$ be a right linear topology on~$R$.
 Assume that the $Z$\+module $R/I$ has a set of generators of
the cardinality not exceeding~$\lambda$ for every right ideal $I\in\bF$.
 Then $\bF$ satisfies~(T$_\lambda$).

 Indeed, let $(s_\gamma\in R)_{\gamma\in\Gamma}$ be a set of preimages
in $R$ of a set of generators $\bar s_\gamma\in R/I$ of the
$Z$\+module~$R/I$.
 For every finite subset $\Delta\subset\Gamma$ in the set of indices,
consider the right ideal $J_\Delta=\bigcap_{\delta\in\Delta}
(I:s_\delta)\in\bF$.
 Then the set $F_I=\{J_\Delta\}$ all such right ideals in $R$ has
the desired property, because for any $s\in\sum_{\delta\in\Delta}
z_\delta s_\delta + I\subset R$, where $z_\delta\in Z$ for all
$\delta\in\Delta$, one has $sJ_\Delta\subset I$.
\end{exs}

\begin{rem} \label{unions-of-topologies-remark}
 It is clear from the form of the conditions (T0\+-T3) and~(T4) that
the set of all right linear topologies on an associative ring $R$ is
closed under directed unions (over nonempty sets of indices), and so is
the set of all right Gabriel topologies on $R$.
 In other words, if $\Xi$ is a nonempty set of right linear (resp.,
right Gabriel) topologies on $R$ such that for any two topologies
$\bF_1$, $\bF_2\in\Xi$ there exists a topology $\bF\in\Xi$ with
$\bF_1\cup\bF_2\subset\bF$, then $\bigcup_{\bF\in\Xi}\bF$ is a right
linear (resp., Gabriel) topology on~$R$.
 Moreover, if every $\bF\in\Xi$ has a base of finitely generated right
ideals, then so does $\bigcup_{\bF\in\Xi}\bF$.
 If every $\bF\in\Xi$ satisfies (T$_\lambda$), then so does
$\bigcup_{\bF\in\Xi}\bF$.

 Moreover, if $\Xi$ is a nonempty directed set of perfect Gabriel
topologies on $R$, then the Gabriel topology
$\bG=\bigcup_{\bH\in\Xi}\bH$ is also perfect.
 Indeed, by~\cite[Proposition~XI.3.4]{St}, a Gabriel topology $\bG$
on $R$ is perfect if and only if the ring of quotients $R_\bG$ is
a $\bG$\+divisible left $R$\+module, that is $R/I\ot_R R_\bG=0$ for
every $I\in\bG$.
 For any two embedded Gabriel topologies $\bH\subset\bG$ on $R$,
there is a natural ring homomorphism $R_\bH\rarrow R_\bG$ forming
a commutative triangle diagram with the ring homomorphisms
$R\rarrow R_\bH$ and $R\rarrow R_\bG$.
 Now if $I\subset R$ is a right ideal belonging to $\bG$ and
$\bG=\bigcup_{\bH\in\Xi}\bH$, then there exists $\bH\in\Xi$ such that
$I\in\bH$.
 Hence, if $\bH$ is a perfect Gabriel topology, we have $R/I\ot_RR_\bG=
(R/I\ot_RR_\bH)\ot_{R_\bH}R_\bG=0$.
\end{rem}

\begin{prop} \label{finitely-based-gabriel-lambda-prop}
 Let\/ $R$ be an associative ring, $\lambda$~be an infinite cardinal,
$\bG$ be a right Gabriel topology on $R$ satisfying~(T$_\lambda$) and
having a base consisting of finitely generated right ideals, and\/
$\bI\subset\bG$ be a subset of the cardinality not exceeding~$\lambda$.
 Then there exists a subset\/ $\bI\subset\bH\subset\bG$ such that\/
$\bH$ is a Gabriel topology on~$R$ having a base of the cardinality
not exceeding~$\lambda$ consisting of finitely generated right ideals.
\end{prop}

\begin{proof}
 It suffices to iterate the following procedure over the set of all
nonnegative integers $\omega=\boZ_{\ge0}$.
 For every right ideal $I\in\bI$, choose a finitely generated right
ideal $I'\subset I$, \ $I'\in\bG$, and denote by $\bI_0\subset\bG$
the set of all right ideals $I'$ so obtained.
 To make sure that $\bI_0$ is nonempty, one can also include the unit
right ideal $R$ into~$\bI_0$.
 Given a subset $\bI_n\subset\bG$ consisting of finitely generated
right ideals, for every $I\in\bI_n$ consider a subset $F_I\subset\bG$
provided by the condition~(T$_\lambda$), and choose a finitely generated
right ideal $J'\subset J$, \ $J'\in\bG$, for every $J\in F_I$.
 Denote by $\bI_n'\subset\bG$ the union of $\bI_n$ with the set of all
right ideals $J'$ so obtained.
 For every pair of right ideals $J$ and $K\in\bI_n'$, choose a finite
set of generators $s_1$,~\dots,~$s_m$ of the right ideal $J$, and
consider the right ideal $I=s_1K+\dotsb+s_mK\in\bG$.
 Denote by $\bI_n''\subset\bG$ the union of $\bI_n'$ with the set of all
right ideals $I$ so obtained.
 For every pair of right ideals $I$ and $J\in\bI_n''$, choose a finitely
generated right ideal $K\subset I\cap J$, \ $K\in\bG$.
 Denote by $\bI_{n+1}$ the union of $\bI_n''$ with the set of all
right ideals $K$ so obtained.
 Set $\bB=\bigcup_{n<\omega}\bI_n$.
 Then $\bB\subset\bG$ is a set of finitely generated right ideals,
the cardinality of $\bB$ does not exceed~$\lambda$, and $\bB$ is
a base of a Gabriel topology $\bH$ on $R$ such that
$\bI\subset\bH\subset\bG$.
\end{proof}

\begin{cor} \label{finitely-based-gabriel-lambda-directed-cor}
 Let $R$ be an associative ring, $\lambda$~be an infinite cardinal, and
$\bG$ be a right Gabriel topology on $R$ satisfying~(T$_\lambda$) and
having a base of consisting of finitely generated right ideals.
 Let\/ $\Xi$ denote the set of all Gabriel topologies\/ $\bH\subset\bG$
having a base of the cardinality not exceeding~$\lambda$ consisting of
finitely generated right ideals.
 Then the set\/ $\Xi$ is $\lambda^+$\+directed by inclusion and
one has\/ $\bG=\bigcup_{\bH\in\Xi}\bH$. 
\end{cor}

\begin{proof}
 Follows immediately from
Proposition~\ref{finitely-based-gabriel-lambda-prop}.
\end{proof}

 We recall from Section~\ref{prelim-sheafification} that for any right
Gabriel topology $\bG$ on an associative ring $R$ the maximal
$\bG$\+torsion submodule $t_\bG(R)\subset R$ of the right $R$\+module
$R$ is a two-sided ideal, which can be alternatively constructed as
the kernel of the ring homomorphism $R\rarrow R_\bG$.
 A Gabriel topology $\bG$ on $R$ is said to be \emph{faithful} if
$t_\bG(R)=0$.
 More generally, we will say that a right $R$\+module $N$ has
\emph{$\lambda$\+bounded\/ $\bG$\+torsion} if there exists a set of
right ideals $\bJ_0\subset\bG$ of the cardinality
not exceeding~$\lambda$ such that for every element $b\in t_\bG(N)$
there is a right ideal $J\in\bJ_0$ for which $bJ=0$.

 Clearly, if $\Xi$ is a nonempty directed set of faithful Gabriel
topologies on $R$, then the Gabriel topology $\bigcup_{\bH\in\Xi}\bH$
is also faithful.
 If $\Xi$ is a nonempty directed set of faithful Gabriel topologies
on $R$ such that the cardinality of $\Xi$ does not exceed~$\lambda$
and the right $R$\+module $R$ has $\lambda$\+bounded $\bH$\+torsion
for every $\bH\in\Xi$, then $R$ has $\lambda$\+bounded $\bG$\+torsion
for the Gabriel topology $\bG=\bigcup_{\bH\in\Xi}\bH$ as well.

\begin{lem} \label{sheafification-and-direct-limits}
 Let $R$ be an associative ring, $\Xi$ be a directed set of right
Gabriel topologies on $R$, and\/ $\bG=\bigcup_{\bH\in\Xi}\bH$ be
their union.
 Let $N$ be a right $R$\+module such that $t_\bH(N)=t_\bG(N)$ for
all\/ $\bH\in\Xi$.
 Then there is a natural isomorphism of right $R$\+modules\/
$N_\bG\simeq\varinjlim_{\bH\in\Xi}N_\bH$.
\end{lem}

\begin{proof}
 For any directed set of right Gabriel topologies $\Xi$ on $R$ and
any right $R$\+module $N$, one has a natural diagram of right
$R$\+modules $(N_\bH)_{\bH\in\Xi}$ indexed by the poset $\Xi$, which
is commutative together with the natural right $R$\+module morphisms
$N\rarrow N_\bH\rarrow N_\bG$.
 These right $R$\+modules can be described in terms of the
sheafification construction of Section~\ref{prelim-sheafification} as
$$
 N_\bH=(N/t_\bH(N))_{(\bH)} \quad\text{and}\quad
 N_\bG=(N/t_\bG(N))_{(\bG)}.
$$
 Now when $t_\bH(N)=t_\bG(N)$ for all $\bH\in\Xi$, it is clear from
this construction that $N_\bG$ is the direct limit of $N_\bH$ over
$\bH\in\Xi$.
\end{proof}

\begin{rem}
 When the ring $R$ is right $\lambda^+$\+Noetherian (i.~e., every right
ideal in $R$ has a set of generators of the cardinality~$\le\lambda$),
or the Gabriel topology $\bG=\bigcup_{\bH\in\Xi}\bH$ has a base
consisting of right ideals with at most $\lambda$ generators and
the ring $R$ is right $\lambda^+$\+coherent (i.~e., every right ideal
in $R$ with at most $\lambda$ generators can be defined, as a right
$R$\+module, by a set of relations of the cardinality~$\le\lambda$),
one can drop the assumption $t_\bH(N)=t_\bG(N)$ in
Lemma~\ref{sheafification-and-direct-limits}
for a $\lambda^+$\+directed set of Gabriel topologies~$\Xi$.
 Indeed, the functor $N\longmapsto N_{(\bG)}$ commutes with direct
limits over $\lambda^+$\+directed posets of indices in this case, so
one has
\begin{multline*}
 N_\bG=(N/t_\bG(N))_{(\bG)}=
 \bigl(\varinjlim\nolimits_{\bH'\in\Xi}N/t_{\bH'}(N)\bigr)_{(\bG)} =
 \varinjlim\nolimits_{\bH'\in\Xi}(N/t_{\bH'}(N))_{(\bG)} \\ =
 \varinjlim\nolimits_{\bH''\in\Xi}
 \varinjlim\nolimits_{\bH'\in\Xi}(N/t_{\bH'}(N))_{(\bH'')} 
 = \varinjlim\nolimits_{\bH\in\Xi}(N/t_\bH(N))_{(\bH)}
 = \varinjlim\nolimits_{\bH\in\Xi} N_\bH.
\end{multline*}
 Hence the $\lambda$\+bounded $\bG$\+torsion assumption can be dropped
in Proposition~\ref{perfect-gabriel-lambda-prop} and
Corollary~\ref{perfect-gabriel-lambda-directed-cor} below when
the ring $R$ is right $\lambda^+$\+coherent.
\end{rem}

 The next lemma shows that one does not need the condition on torsion
when dealing with directed unions of perfect Gabriel topologies.

\begin{lem} \label{perfect-topologies-and-direct-limits}
 Let $R$ be an associative ring, $\Xi$ be a directed set of perfect
right Gabriel topologies on $R$, and\/ $\bG=\bigcup_{\bH\in\Xi}\bH$
be their union.
 Then there is a natural isomorphism of associative rings
$R_\bG\simeq\varinjlim_{\bH\in\Xi}R_\bH$.
\end{lem}

\begin{proof}
 As in Lemma~\ref{sheafification-and-direct-limits}, for any directed
set of right Gabriel topologies $\Xi$ on $R$ one has a diagram
of associative rings $(R_\bH)_{\bH\in\Xi}$ indexed by the poset $\Xi$,
which is commutative together with the natural ring homomorphisms
$R\rarrow R_\bH\rarrow R_\bG$.
 Now, if the Gabriel topology $\bH$ is perfect for every $\bH\in\Xi$,
then so is the Gabriel topology $\bG$ (see
Remark~\ref{unions-of-topologies-remark}).
 Set $U=\varinjlim_{\bH\in\Xi}R_\bH$; then there is a commutative
diagram of ring homomorphisms $R\rarrow U\rarrow R_\bG$.
 The map $R\rarrow U$ is a ring epimorphism, since so are the maps
$R\rarrow R_\bH$; and the left $R$\+module $U$ is flat, since
the left $R$\+modules $R_\bH$~are.
 Denote by $\bG'$ the perfect right Gabriel topology on $R$
corresponding to the left flat ring epimorphism $R\rarrow U$;
so $\bG'$ consists of all the right ideals $I\subset R$ such that
$R/I\ot_RU=0$, and one has $U=R_{\bG'}$.

 Let us show that $\bG'=\bG$.
 As we have already seen, $\bH\subset\bG'$ for all $\bH\in\Xi$,
since $R/I\ot_RU=(R/I\ot_RR_\bH)\ot_{R_\bH}U=0$ for any $I\in\bH$.
 Conversely, for any right ideal $I\subset R$ the right $U$\+module
$U/IU=R/I\ot_RU=\varinjlim_{\bH\in\Xi}R/I\ot_RR_\bH
=\varinjlim_{\bH\in\Xi}R/IR_\bH$ is the inductive limit of
the right $R_\bH$\+modules $R_\bH/IR_\bH$.
 These are cyclic right modules with given natural generators
$\bar 1\in U/IU$ and $\bar 1\in R_\bH/IR_\bH$, and the maps in
the diagram take the generator to the generator.
 Hence one has $U/IU=0$ if and only if the generator $\bar 1\in
U/IU$ is a zero element, which holds if and only if there exists
$\bH\in\Xi$ such that the generator $\bar 1\in R_\bH/IR_\bH$
is a zero element, which means that $R_\bH/IR_\bH=0$.
 Thus $I\in\bG'$ implies existence of $\bH\in\Xi$ for which $I\in\bH$.
 We can conclude that $\bG'=\bG$ and $U=R_{\bG'}=R_\bG$.
\end{proof}

\begin{prop} \label{perfect-gabriel-lambda-prop}
 Let $R$ be an associative ring, $\lambda$~be an infinite cardinal,
and\/ $\bG$~be a perfect right Gabriel topology on $R$
satisfying~(T$_\lambda$) and such that the right $R$\+module $R$ has
$\lambda$\+bounded\/ $\bG$\+torsion.
 Let\/ $\bI\subset\bG$ be a subset of the cardinality not
exceeding~$\lambda$.
 Then there exists a subset\/ $\bI\subset\bP\subset\bG$ such that\/
$\bP$ is a perfect Gabriel topology on $R$ with a base of
the cardinality not exceeding~$\lambda$.
\end{prop}

\begin{proof}
 As in the proof of
Proposition~\ref{finitely-based-gabriel-lambda-prop}, one iterates
a certain procedure over the smallest countable ordinal~$\omega$.
 By~\cite[Proposition~XI.3.4]{St}, the right Gabriel topology $\bG$
has a base consisting of finitely generated right ideals.
 Let $\bJ_0$ be a set of right ideals in $R$ of the cardinality not
exceeding~$\lambda$ such that for every $r\in t_\bG(R)$ there exists
$J\in\bJ_0$ for which $rJ=0$.
 By Proposition~\ref{finitely-based-gabriel-lambda-prop}, one can choose
a right Gabriel topology $\bH_0$ on $R$ with a base of the cardinality
not exceeding~$\lambda$ such that $\bI\cup\bJ_0\subset\bH_0\subset\bG$.

 Let $\bH_n$ be another right Gabriel topology with a base of $\bB_n$
the cardinality not exceeding~$\lambda$ such that
$\bH_0\subset\bH_n\subset\bG$.
 By Corollary~\ref{finitely-based-gabriel-lambda-directed-cor},
the set $\Xi$ of all right Gabriel topologies $\bH$ on $R$ having a base
of the cardinality not exceeding~$\lambda$ consisting of finitely
generated right ideals and such that $\bH_n\subset\bH\subset\bG$ is
$\lambda^+$\+directed by inclusion, and one has
$\bG=\bigcup_{\bH\in\Xi}\bH$.

 As in the proof of Lemma~\ref{perfect-topologies-and-direct-limits},
we have a diagram of associative rings and ring homomorphisms
$(R_\bH)_{\bH\in\Xi}$ indexed by the poset~$\Xi$, which is commutative
together with the ring homomorphisms $R\rarrow R_\bH\rarrow R_\bG$.
 Notice that $t_\bG(R)=t_\bH(R)$ for all $\bH\in\Xi$, as
$\bJ_0\subset\bH$.
 By Lemma~\ref{sheafification-and-direct-limits}, we have
$R_\bG=\varinjlim_{\bH\in\Xi}R_\bH$.

 Since $\bG$ is a perfect Gabriel topology, for any right ideal
$I\in\bG$ one has $R/I\ot_RR_\bG=0$.
 In particular, this holds for any right ideal $I$ belonging to our
base $\bB_n$ of the Gabriel topology~$\bH_n$.
 Since the ring $R_\bG$ is a direct limit of the rings $R_\bH$, there
exists a Gabriel topology $\bH(I)\in\Xi$ such that
$R/I\ot_RR_{\bH(I)}=0$.
 Choose a Gabriel topology $\bH_{n+1}$ with a base of the cardinality
not exceeding~$\lambda$ such that $\bH_n\subset\bH_{n+1}\subset\bG$ and
$\bH(I)\subset\bH_{n+1}$ for all $I\in\bB_n$.
 Then $\bP=\bigcup_{n<\omega}\bH_n$ is a perfect Gabriel topology on $R$
with a base of the cardinality not exceeding~$\lambda$, and one has
$\bI\subset\bP\subset\bG$.
\end{proof}

\begin{cor} \label{perfect-gabriel-lambda-directed-cor}
 Let $R$ be an associative ring, $\lambda$~be an infinite cardinal,
and\/ $\bG$~be a perfect right Gabriel topology on $R$
satisfying~(T$_\lambda$) and such that the right $R$\+module $R$ has
$\lambda$\+bounded $\bG$\+torsion.
 Let\/ $\Upsilon$ denote the set of all perfect Gabriel topologies\/
$\bP\subset\bG$ having a base of the cardinality not exceeding~$\lambda$.
 Then the set\/ $\Upsilon$ is $\lambda^+$\+directed by inclusion
and one has\/ $\bG=\bigcup_{\bP\in\Upsilon}\bP$.
 Moreover, the ring of quotients of $R$ with respect to\/ $\bG$ is
the direct limit over\/ $\Upsilon$ of the rings of quotients of $R$ with
respect to\/~$\bP$,
$$
 R_\bG=\varinjlim\nolimits_{\bP\in\Upsilon}R_\bP.
$$
\end{cor}

\begin{proof}
 Follows from Proposition~\ref{perfect-gabriel-lambda-prop}
and Lemma~\ref{perfect-topologies-and-direct-limits}.
\end{proof}

\Section{$\bF$-Systems and Pseudo-$\bF$-Systems}  \label{F-systems-secn}

 In this section we discuss some techniques of working with
contramodules over topological associative rings with right linear
topology, which were developed originally in~\cite[Section~6]{PR}.
 Here we rephrase them in a more intuitively accessible language.
 These techniques become particularly powerful when the topology on
the ring has a countable base, as we will see in the next
Section~\ref{separated-contramodules-secn}.

 Let $R$ be an associative ring, $\bF$ be a right linear topology on
$R$, and $\R=\varprojlim_{I\in\bF}R/I$ be the completion of $R$ with
respect to $\bF$, viewed as a topological ring in the projective
limit topology~$\bfF$.
 There is a natural continuous (completion) homomorphism of topological
rings $\rho\:R\rarrow\R$.

 Notice that there is a natural bijection between the set $\bF$ of
all open right ideals $I\subset R$ and the set $\bfF$ of all open right
ideals $\I\subset\R$, given by the rules $I=\rho^{-1}(\I)$ and
$\I=\varprojlim_{J\in\bF,\,J\subset I}I/J=\ker(\R\to R/I)$.
 The ring homomorphism $\rho\:R\rarrow\R$ induces an isomorphism
of quotient groups/modules $R/I\simeq\R/\I$.

 Denote the full subcategory of cyclic discrete right modules $R/I$
in $\discr R$ by~$\sQ_\bF$.
 The objects of $\sQ_\bF$ are indexed by the open right
ideals $I\in\bF$.

 According to~(T2), the poset of open right ideals $\bF$ is
downwards directed by inclusion.
 When the need arises to distinguish this poset from the topology $\bF$
on the ring $R$, we will denote the poset $\bF$, and the category
associated with it, by~$\Pi_\bF$.
 For any pair of open right ideals $J\subset I\subset R$, one of which
is contained in the other one, there is the natural surjective morphism
of right $R$\+modules $R/J\rarrow R/I$.
 Hence $\Pi_\bF$ is naturally a subcategory in~$\sQ_\bF$ (with
the same set of objects).
 The morphisms in $\Pi_\bF$ are exactly those morphisms of right
$R$\+modules $R/J\rarrow R/I$ which form a commutative triangle
diagram with the projections $R\rarrow R/I$ and $R\rarrow R/J$.

 According to Section~\ref{prelim-discrete}, the categories of discrete
right modules over the topological rings $R$ and $\R$ are naturally
equivalent (in fact, isomorphic), $\discr R\simeq\discr\R$.
 The restriction of this equivalence of additive categories to the full
subcategories $\sQ_\bF\subset\discr R$ and $\sQ_\bfF\subset\discr\R$
provides an equivalence of preadditive categories
$\sQ_\bF\simeq\sQ_\bfF$, which further restricts to an equivalence
of categories $\Pi_\bF\simeq\Pi_\bfF$.

 Furthermore, given an additive category $\sA$, let $\Com(\sA)$ denote
the additive category of ($\boZ$\+graded) complexes in~$\sA$.
 Given an associative ring $R$ with a right linear topology $\bF$,
let us consider the full subcategory in the category of
complexes of right $R$\+modules $\Com(\modr R)$ formed by the two-term
complexes $(I\to R)$, where $I\in\bF$ ranges over the open right ideals
in $R$ and the map $I\rarrow R$ is the identity inclusion.
 Here the term $I$ of the complex $(I\to R)$ sits in the cohomological
degree~$0$ and the term~$R$ in the cohomological degree~$1$.
 We denote this full subcategory by $\sP_\bF\subset\Com(\modr R)$.

 The abelian groups of morphisms in the preadditive category $\sP_\bF$
can be easily described.
 For any two open right ideals $I$ and $J\in\bF$, a morphism of
complexes $(J\to R)\rarrow (I\to R)$ is uniquely determined by
its action on the terms~$R$ in degree~$1$ (because the map $I\rarrow R$
is injective).
 The group of all such morphisms is naturally isomorphic to the subgroup
in $R$ consisting of all the elements $s\in R$ such that $sJ\subset I$.
 An element~$s$ satisfying this condition corresponds to the morphism
$(J\to R)\overset s\rarrow (I\to R)$ of left multiplication with~$s$.

 There is a natural additive functor $H^1\:\sP_\bF\rarrow\sQ_\bF$
induced by the passage to the discrete right $R$\+modules of cohomology
$H^1(I\to R)=R/I$ of the two-term complexes of right $R$\+modules
$(I\to R)$.
 This functor is bijective on objects and surjective on morphisms.
 The kernel of the map $\Hom_{\sP_\bF}(J\to R,\>I\to R)\rarrow
\Hom_{\sQ_\bF}(R/J,R/I)$ consists of all the morphisms $(J\to R)
\overset{s}\rarrow(I\to R)$ for which the element~$s$ belongs to
the right ideal~$I$.

 The inclusion $\Pi_\bF\rarrow\sQ_\bF$ of the category corresponding
to the poset $\bF$ into the category $\sQ_\bF$ lifts naturally to
an inclusion $\Pi_\bF\rarrow\sP_\bF$.
 For any two open right ideals $J\subset I\subset R$, the only morphism
$J\rarrow I$ in $\Pi_\bF$ corresponds to the morphism
$(J\to R)\overset1\rarrow(I\to R)$ in $\sP_\bF$ acting by
the identity inclusion $J\rarrow I$ in degree~$0$ and by
the identity isomorphism~$\id_R$ in degree~$1$.

 Notice that, unlike the categories $\Pi_\bF$ and $\sQ_\bF$,
the category $\sP_\bF$ \emph{does} change when one passes from
a topological ring $R$ with a right linear topology $\bF$ to
its completion $\R$ with its projective limit topology~$\bfF$.
 The natural functors $\Pi_\bF\rarrow\Pi_\bfF$ and $\sQ_\bF\rarrow
\sQ_\bfF$ are isomorphisms of categories, but the natural additive
functor $\sP_\bF\rarrow\sP_\bfF$ is \emph{not} an equivalence (even
though it is bijective on objects).

 More specifically, let $I\in\bF$ be an open right ideal in $R$ and
$\I\in\bfF$ be the corresponding open right ideal in~$\R$.
 Then the functor $\sP_\bF\rarrow\sP_\bfF$ assigns the two-term
complex $(\I\to\R)$ to the two-term complex $(I\to R)$.
 To construct the action of this functor on morphisms, one can use
the explicit description of morphisms in the category~$\sP_\bF$
given above.
 For any morphism $(J\to R)\overset s\rarrow(I\to R)$ in $\sP_\bF$,
the image~$\rho(s)\in\R$ of the element $s\in R$ defines
the corresponding morphism $(\J\to\R)\overset{\rho(s)}\lrarrow
(\I\to\R)$.

 In particular, one can observe that the functor $\sP_\bF\rarrow
\sP_\bfF$ takes the two-term complex of right $R$\+modules
$(R\to R)$ to the two-term complex of right $\R$\+modules $(\R\to\R)$.
 One has $\Hom_{\sP_\bF}(R\to R,\>R\to R)=R\ne
\R=\Hom_{\sP_\bfF}(\R\to\R,\>\R\to\R)$.
 For comparison, the images of these two-term complexes in
the categories $\sQ_\bF$ and $\sQ_\bfF$ are, of course, zero objects.

 Let $\sA$ be an additive category.
 By a \emph{covariant pseudo-$\bF$-system} in $\sA$ we will mean
a covariant additive functor $\sP_\bF\rarrow\sA$.
 Similarly, a \emph{contravariant pseudo-$\bF$-system} in $\sA$ is
a contravariant additive functor $\sP_\bF\rarrow\sA$.
 We will denote the category of covariant pseudo-$\bF$-systems
in $\sA$ by ${}_{\{\bF\}}\sA=\Add(\sP_\bF,\sA)$ and the category of
contravariant pseudo-$\bF$-systems in $\sA$ by $\sA_{\{\bF\}}=
\Add(\sP_\bF^\sop,\sA)$ (where $\Add(\sC,\sA)$ denotes the category of
additive functors $\sC\rarrow\sA$ and $\sC^\sop$ is the opposite
category to~$\sC$).
 When the category $\sA$ is abelian, so are the categories
${}_{\{\bF\}}\sA$ and~$\sA_{\{\bF\}}$.

 Composing an additive functor $\sP_\bfF^\sop\rarrow\sA$ or
$\sP_\bfF\rarrow\sA$ with the natural additive functor
$\sP_\bF\rarrow\sP_\bfF$ discussed above, one can construct
the underlying pseudo-$\bF$-system of a (contravariant or covariant)
pseudo-$\bfF$-system.

 A \emph{left $R$\+module object} $M$ in $\sA$ is an object endowed
with a left action of $R$, i.~e., with a ring homomorphism
$R\rarrow\Hom_\sA(M,M)$.
 Similarly, a \emph{right $R$\+module object} $N$ in $\sA$ is an object
endowed with a right action of $R$, i.~e., with a ring homomorphism
$R^\rop\rarrow\Hom_\sA(N,N)$.
 Denote the category of left $R$\+module objects in $\sA$ by ${}_R\sA$
and the category of right $R$\+module objects in $\sA$ by~$\sA_R$.

\begin{lem} \label{modules-from-pseudo-systems-lem}
\textup{(a)} Let\/ $\sA$ be an additive category with direct limits
and $M\:\sP_\bF^\sop\rarrow\sA$ be a contravariant pseudo-$\bF$-system
in\/~$\sA$.
 Then the object
$$
 \il(M)=\varinjlim\nolimits_{I\in\bF}M(I\to R)\in\sA,
$$
where the direct limit is taken over the poset\/ $\Pi_\bF^\sop$, carries
a natural right action of the ring~$R$. \par
\textup{(b)} Let\/ $\sA$ be an additive category with inverse limits
and $D\:\sP_\bF\rarrow\sA$ be a covariant pseudo-$\bF$-system
in\/~$\sA$.
 Then the object
$$
 \pl(D)=\varprojlim\nolimits_{I\in\bF}D(I\to R)\in\sA,
$$
where the inverse limit is taken over the poset\/ $\Pi_\bF$, carries
a natural left action of the ring~$R$.
\end{lem}

\begin{proof}
 The abbreviation ``$\il$'' stands for ``inductive limit'', while
the abbreviation ``$\pl$'' means ``projective limit''.
 Let us construct the left action of $R$ in the object $\pl(D)$
in part~(b) (part~(a) being dual).
 Given an element $r\in R$, for every open right ideal
$I\in\bF$ we have the morphism $((I:r)\to R)\overset r\rarrow
(I\to R)$ in the category~$\sP_\bF$.
 Consider the natural projection $\pl(D)\rarrow D((I:r)\to R)$ from
the inverse limit to one of the objects in the diagram, and its
composition
$$
 \pl(D)\lrarrow D((I:r)\to R)\overset{D(r)}\lrarrow D(I\to R)
$$
with the morphism~$D(r)$.
 The collection of such morphisms $\pl(D)\rarrow D(I\to\nobreak R)$,
defined for all the open right ideals $I\in\bF$, forms a compatible
cone (i.~e., a commutative diagram) with all the morphisms in
the diagram $D|_{\Pi_\bF}\:\Pi_\bF\rarrow\sA$.
 Hence we obtain the induced morphism
$$
 \pl(D)\rarrow\varprojlim\nolimits_{I\in\bF}D(I\to R)=\pl(D),
$$
providing the desired left action of the element $r\in R$ in the object
$\pl(D)\in\sA$.
 It is straightforward to check that the map $R\rarrow
\Hom_\sA(\pl(D),\pl(D))$ constructed in this way is a ring homomorphism,
using the assumption that $D\:\sP_\bF\rarrow\sA$ is an additive functor.
\end{proof}

 The functor $\il\:\sA_{\{\bF\}}\rarrow\sA_R$ is left adjoint to
the ``constant contravariant pseudo-$\bF$-system'' functor
$\sA_R\rarrow\sA_{\{\bF\}}$.
 The latter functors assigns to every right $R$\+module
object $N$ in $\sA$ the contravariant pseudo-$\bF$-system taking
an object $(I\to R)\in\sP_\bF$ to the object $N\in\sA$ for every
$I\in\bF$ and a morphism $(J\to R)\overset s\rarrow(I\to R)$ in
$\sP_\bF$ to the right multiplication morphism $s\:N\rarrow N$.
 Similarly, the functor $\pl\:{}_{\{\bF\}}\sA\rarrow{}_R\sA$ is right
adjoint to the ``constant covariant pseudo-$\bF$-system'' functor
${}_R\sA\rarrow{}_{\{\bF\}}\sA$, which assigns to every left $R$\+module
object $C$ in $\sA$ the covariant pseudo-$\bF$-system taking an object
$(I\to R)\in\sP_\bF$ to the object $C\in\sA$ for every $I\in\bF$ and
a morphism $(J\to R)\overset s\rarrow(I\to R)$ in $\sP_\bF$ to
the left multiplication morphism $s\:C\rarrow C$.

 Let $\sA$ be an additive category.
 By the definition, a \emph{contravariant\/ $\bF$\+system} in $\sA$ is
a contravariant additive functor $\sQ_\bF\rarrow\sA$.
 A \emph{covariant\/ $\bF$\+system} in $\sA$ is a covariant additive
functor $\sQ_\bF\rarrow\sA$.
 We will denote the category of contravariant $\bF$\+systems in
$\sA$ by $\sA_\bF$ and the category of covariant $\bF$\+systems in
$\sA$ by~${}_\bF\sA$.
 When the category $\sA$ is abelian, so are the categories $\sA_\bF$
and~${}_\bF\sA$.
 Composing an additive functor $\sQ_\bF^\sop\rarrow\sA$ or
$\sQ_\bF\rarrow\sA$ with the additive functor $H^1\:\sP_\bF\rarrow
\sQ_\bF$, one can view any (covariant or contravariant)
$\bF$\+system in $\sA$ as a (respectively, covariant or contravariant)
pseudo-$\bF$-system.
 This makes the category of $\bF$\+systems a full subcategory
in the category of pseudo-$\bF$-systems,
${}_{\bF}\sA\subset{}_{\{\bF\}}\sA$ and
$\sA_\bF\subset\sA_{\{\bF\}}$.

 We recall the notation $\Ab$ for the category of abelian groups.
 We also recall that every left $\R$\+contramodule $\C$ has
the underlying left $\R$\+module structure.
 The restriction of scalars with respect to the ring homomorphism
$\rho\:R\rarrow\R$ makes $\C$ a left $R$\+module.
 In this context, we will speak of ``the underlying left $R$\+module
structure'' of a left $\R$\+contramodule~$\C$.

\begin{prop} \label{discrete-modules-contramodules-from-systems}
\textup{(a)} The functor\/ $\il\:\Ab_{\{\bF\}}\rarrow\modr R$
constructed in Lemma~\ref{modules-from-pseudo-systems-lem}(a) takes
contravariant\/ $\bF$\+systems of abelian groups to discrete right
$R$\+modules.
 The resulting functor\/ $\IL\:\Ab_\bF\rarrow\discr R$\/ has a right
adjoint functor\/ $\Dh\:\discr R\rarrow\Ab_\bF$.
 The functor\/ $\Dh$ assigns to every discrete right $R$\+module\/ $\N$
the restriction of the contravariant functor of discrete right
$R$\+module homomorphisms\/ $\Hom_{R^\rop}({-},\N)\:\discr R\rarrow\Ab$
to the full subcategory\/ $\sQ_\bF\subset\discr R$. {\hbadness=1550\par}
\textup{(b)} The left $R$\+module\/ $\pl(D)$ assigned by the functor\/
$\pl\:{}_{\{\bF\}}\Ab\rarrow R\modl$ of
Lemma~\ref{modules-from-pseudo-systems-lem}(b) to a covariant\/
$\bF$\+system of abelian groups $D\:\sQ_\bF\rarrow\Ab$ is the underlying
left $R$\+module of a naturally defined left\/ $\R$\+contramodule\/
$\PL(D)$.
 The resulting functor\/ $\PL\:{}_\bF\Ab\rarrow\R\contra$ has a left
adjoint functor\/ $\CT\:\R\contra\rarrow{}_\bF\Ab$.
 The functor\/ $\CT$ assigns to every left\/ $\R$\+contramodule\/ $\C$
the restriction of the covariant functor of contratensor product\/
${-}\ocn_\R\C\:\discr\R\rarrow\Ab$ to the full subcategory\/
$\sQ_\bF=\sQ_\bfF\subset\discr\R$.
\end{prop}

\begin{proof}
 The abbreviation $\Dh$ stands for ``discrete module Hom'', while
the abbreviation $\CT$ means ``contratensor product''.
 Part~(a): let $M\:\sQ_\bF^\sop\rarrow\Ab$ be a contravariant
$\bF$\+system of abelian groups.
 Let $b\in\il(M)$ be an element represented by an element
$\tilde b \in M(R/I)$ for some open right ideal $I\subset R$, and let
$r\in R$ be an element.
 Consider the discrete right $R$\+module morphism $R/(I:r)\overset r
\rarrow R/I$ of left multiplication with~$r$.
 Applying the contravariant functor $M$ to this morphism, we obtain
an abelian group homomorphism $M(r)\:M(R/I)\lrarrow M(R/(I:r))$.
 By the definition, the element $br\in\il(N)$ is represented by
the element $M(r)(b)\in M(R/(I:r))$.

 To check that $\IL(M)=\il(M)$ is a discrete right $R$\+module, one
observes that $bI=0$ in $\il(M)$ whenever an element $b\in\il(M)$
is represented by an element $\tilde b\in M(R/I)$.
 Indeed, for any $r\in I$ one has $(I:r)=R$, so $R/(I:r)=0$ is a zero
object of the preadditive category $\sQ_\bF$, and $M(0)=0$
for an additive functor~$M$.

 The adjunction isomorphism $\Hom_{\discr R}(\IL(M),\N)\simeq
\Hom_{\Ab_\bF}(M,\Dh(\N))$ holds for any contravariant $\bF$\+system of
abelian groups $M$ and any discrete right $R$\+module $\N$, because
the datum of a right $R$\+module morphism
$$
 \varinjlim\nolimits_{I\in\bF}M(R/I)\lrarrow\N
$$
is equivalent to the datum of an $\bF$\+indexed family of abelian
group homomorphisms
$$
 M(R/I)\lrarrow\Hom_{R^\rop}(R/I,\N)\,\subset\,\N
$$
satisfying the compatibility equation for all the morphisms in
the category~$\sQ_\bF$.
 It is helpful to keep in mind that every morphism in $\sQ_\bF$ has
the form of a composition $R/J\rarrow R/(I:\nobreak s)\overset s
\rarrow R/I$, where $I$ and $J\subset(I:s)$ are open right ideals,
$s\in R$ is an element, $R/J\rarrow R/(I:s)$ is the morphism in
$\Pi_\bF$, and $R/(I:s)\overset s\rarrow R/I$ is the morphism of
left multiplication with~$s$.

 Part~(b): this is~\cite[Lemma~6.2(a,c)]{PR} (notice that the existence
of a countable base of the topology on $\R$, which is the running
assumption in~\cite[Section~6]{PR}, is not yet used
in~\cite[Lemma~6.2]{PR}).
 It will be convenient for us to work with covariant $\bfF$\+systems
of abelian groups instead of the covariant $\bF$\+systems here.
 Then one has $\CT(\C)(\R/\I)=\R/\I\ocn_\R\C=\C/(\I\tim\C)$ for any left
$\R$\+contramodule $\C$ and open right ideal $\I\subset\R$.
 Given an element $s\in\R$ and a pair of open right ideals $\I$ and
$\J\subset\R$ such that $s\J\subset\I$, applying the functor
$\CT(\C)\:\sQ_\bfF\rarrow\Ab$ to the morphism $\R/\J\overset s\rarrow
\R/\I$ in the category $\sQ_\bfF$ produces the map
$\C/(\J\tim\C)\rarrow\C/(\I\tim\C)$ induced by the abelian group
homomorphism $s\:\C\rarrow\C$ of left multiplication with~$s$.

 The functor $\PL\:{}_\bfF\Ab\rarrow\R\contra$ assigns to a covariant
$\bfF$\+system of abelian groups $D\:\sQ_\bfF\rarrow\Ab$ the abelian
group
$$
 \PL(D)=\varprojlim\nolimits_{\I\in\bfF}D(\R/\I),
$$
where the projective limit is taken over the directed poset $\Pi_\bfF$,
endowed with the following left $\R$\+contramodule structure.
 Denote by $\psi_\I$ the natural projection map
$\PL(D)\rarrow D(\R/\I)$.
 For any open right ideal $\I\subset\R$ and an element $r\in\R$,
we have the related morphism $\R/(\I:r)\overset{r}\rarrow\R/\I$
in the category~$\sQ_\bfF$.
 Set $\E=\PL(D)$; and suppose that we are given an element
$\mathfrak r=\sum_{e\in\E}r_ee\in\R[[\E]]$.
 Then the element $\pi_\E(\mathfrak r)\in\E=\PL(D)$ is defined
by the rule
$$
 \psi_\I(\pi_\E(\mathfrak r))=
 \sum\nolimits_{e\in\E}D(r_e)(\psi_{(\I:r_e)}(e)),
$$
where $\psi_{(\I:r_e)}(e)\in D(\R/(\I:r_e))$ and
$D(r_e)\:D(\R/(\I:r_e))\rarrow D(\R/\I)$.
 The sum in the right-hand side is finite, because one has
$r_e\in\I$ for all but a finite subset of elements $e\in\E$,
and $r_e\in\I$ implies $(\I:r_e)=\R$, so $\R/(\I:r_e)=0$ is a zero
object in $\sQ_\bfF$, and $D(0)=0$ for an additive functor
$D\:\sQ_\bfF\rarrow\Ab$.

 The adjunction isomorphism $\Hom^\R(\C,\PL(D))\simeq
\Hom_{{}_\bfF\Ab}(\CT(\C),D)$ holds for any left $\R$\+contramodule
$\C$ and any covariant $\bfF$\+system of abelian groups $D$,
because the datum of a left $\R$\+contramodule morphism
$$
 \C\lrarrow\varprojlim\nolimits_{\I\in\bfF}D(\R/\I)
$$
is equivalent to the datum of an $\bfF$\+indexed family of abelian
group homomorphisms
$$
 \C/(\I\tim\C)\lrarrow D(\R/\I)
$$
satisfying the compatibility equations for all the morphisms in
the category~$\sQ_\bfF$.
 The argument is similar to the one in part~(a), and it is helpful
to observe, from the construction of the left $\R$\+contramodule
structure on $\E=\PL(D)$ above, that one has $\psi_\I(e)=0$
for any $e\in\I\tim\E$.
\end{proof}

\Section{Separated Contramodules}  \label{separated-contramodules-secn}

 We keep the notation of the previous Section~\ref{F-systems-secn}.
 So $R$ is a topological associative ring with a right linear topology
$\bF$ and $\R$ is the completion of $R$ with respect to $\bF$, viewed
as a topological ring in the projective limit topology~$\bfF$.

 Let $J\subset I\subset R$ be a pair of open right
ideals in $R$, one of which is contained in the other one.
 The following exact sequence of discrete right $R$\+modules
will be important for us:
\begin{equation} \label{I-J-sequence}
 \bigoplus\nolimits_{s\in I} R/(J:s)\overset{(s)}\lrarrow R/J
 \lrarrow R/I\lrarrow 0,
\end{equation}
where the $s$\+indexed component of the map
$(s)_{s\in I}\:\bigoplus_{s\in I}R/(J:s)\rarrow R/J$ is the right
$R$\+module morphism $R/(J:s)\rarrow R/J$ acting by the left
multiplication with~$s$.

 Whenever $\sA$ is a complete abelian category, we will say that
a contravariant $\bF$\+system $M\:\sQ_\bF^\sop\rarrow\sA$ is
\emph{left exact} if for any two open right ideals
$I\subset J\subset R$ the short sequence of objects of~$\sA$
$$
 0\lrarrow M(R/I)\lrarrow M(R/J)\lrarrow
 \prod\nolimits_{s\in I}M(R/(J:s))
$$
obtained by applying the contravariant functor $M$ to
the sequence~\eqref{I-J-sequence} (and replacing formally
the undefined image of the direct sum with the product of the images)
is left exact in~$\sA$.

 Similarly, whenever $\sA$ is a cocomplete abelian category, we will
say that a covariant $\bF$\+system $D\:\sQ_\bF\rarrow\sA$ is
\emph{right exact} if for any two open right ideals
$I\subset J\subset R$ the short sequence of objects of~$\sA$
$$
 \coprod\nolimits_{s\in I} D(R/(J:s))\lrarrow D(R/J)\lrarrow
 D(R/I)\lrarrow 0
$$
obtained by applying the covariant functor $D$ to
the sequence~\eqref{I-J-sequence} (and replacing formally
the undefined image of the direct sum with the coproduct of
the images) is right exact in~$\sA$.

\begin{prop} \label{discrete-modules-contravariant-systems}
 In the context of
Proposition~\ref{discrete-modules-contramodules-from-systems}(a),
the composition\/ $\IL\circ\Dh$ of the two adjoint functors\/
$\Dh\:\discr R\rarrow\Ab_\bF$ and\/ $\IL\:\Ab_\bF\rarrow\discr R$ is
the identity functor\/ $\discr R\rarrow\discr R$.
 A contravariant\/ $\bF$\+system of abelian groups\/
$\sQ_\bF^\sop\rarrow\Ab$ belongs to the image of the functor\/ $\Dh$
if and only if it is left exact.
 So the functor\/ $\Dh$ is fully faithful, and it provides
an equivalence between the category of discrete right $R$\+modules\/
$\discr R$ and the full subcategory in\/ $\Ab_\bF$ formed by all
the left exact contravariant\/ $\bF$\+systems of abelian groups.
\end{prop}

\begin{proof}
 For every discrete right $R$\+module $\N$, the contravariant
$\bF$\+system of abelian groups $\Dh(\N)$ is left exact, because
the functor $\Hom_R({-},\N)\:\discr R\rarrow\Ab$ is left exact.
 The adjunction morphism $\IL(\Dh(\N))\rarrow\N$ is an isomorphism
representing $\N$ as the direct limit/union
$\N=\varinjlim_{I\in\bF}\N_I$ of its subgroups $\Hom_R(R/I,\N)=\N_I
\subset\N$ of all elements annihilated by a given open right ideal
$I\subset R$.

 To prove that all the left exact contravariant $\bF$\+systems
of abelian groups belong to the essential image of the functor
$\Dh$, one can argue as follows.
 By the Special Adjoint Functor Theorem~\cite[Corollary~5.57]{Fa},
every contravariant functor $\discr R\rarrow\Ab$ taking colimits
to limits is representable by an object of $\discr R$, since
$\discr R$ is a cocomplete abelian category with a set of generators
(formed, e.~g., by the cyclic discrete right $R$\+modules~$R/I$).
 It remains to show that any left exact contravariant $\bF$\+system
of abelian groups $M$ can be extended to such a functor.
 For this purpose, one can use the construction of the Kan extension,
setting
$$
 G_M(\N)=\varprojlim\nolimits_{R/I\to\N}M(R/I)\qquad
 \text{for every \,$\N\in\discr R$},
$$
where the projective limit is taken over the diagram formed by
all the morphisms of discrete right $R$\+modules $R/I\rarrow\N$
(indexing the vertices of the diagram) and all the commutative
triangles $R/I\rarrow R/J\rarrow\N$ (indexing the arrows), where
$R/I\rarrow R/J$ ranges over the morphisms in~$\sQ_\bF$.
 For any contravariant $\bF$\+system $M$, the functor $G_M$ takes
coproducts to products (since the right $R$\+modules $R/I$ are
finitely generated); and it is straightforward to check that
the contravariant functor $G_M$ is left exact whenever
the contravariant $\bF$\+system $M$~is.

 Alternatively, one can prove explicitly from the construction that
the adjunction morphism $M\rarrow\Dh(\IL(M))$ is an isomorphism for
every left exact contravariant $\bF$\+system of abelian groups~$M$.
 For this purpose, one needs to check that the subgroup of all
elements annihilated by an open right ideal $I\subset R$ in
the discrete right $R$\+module $\IL(M)=\varinjlim_{J\in\bF}M(R/J)$
coincides with the image of the (injective) map
$M(R/I)\rarrow\varinjlim_{J\in\bF}M(R/J)$.
 This is but a restatement of the definition of left exactness
of a contravariant $\bF$\+system.
\end{proof}

 Let $\C$ be a left $\R$\+contramodule.
 Denote by $\bB_\C$ the set of all subgroups $\I\tim\C\subset\C$ of
the underlying abelian group of $\C$, where $\I$ ranges over
the open right ideals in~$\R$.
 Then $\bB_\C$ is a linear topology base on the underlying abelian
group of $\C$, as $\bB_\C$ is nonempty and
$$
 (\I\tim\C)\cap(\J\tim\C)\supset(\I\cap\J)\tim\C
$$
for any two open right ideals $\I$ and $\J\subset\R$.
 Notice a base of neighborhoods of zero in this topology on $\C$
is formed by subgroups and \emph{not} submodules; so this is not
a linear topology on a module in the sense of~\cite[Section~VI.4]{St}
(it is instructive to observe that topologies on \emph{right}
modules over a right linearly topological ring are considered
in~\cite[Section~VI.4]{St}, while our $\C$ is a \emph{left}
$\R$\+module).

 A left $\R$\+contramodule $\C$ is said to be \emph{separated}
(respectively, \emph{complete}) if it is a separated (resp., complete)
abelian group in the topology with a base~$\bB_\C$.
 In other words, $\C$ is called \emph{separated} if the natural map
of abelian groups $\lambda_{\R,\C}\:\C\rarrow\varprojlim_{\I\in\bfF}
\C/(\I\tim\C)$ is injective, and $\C$ is \emph{complete} if this map
is surjective.
 Clearly, $\C$ is separated if and only if
$\bigcap_{\I\in\bfF}\I\tim\C=0$.
 Any $\R$\+subcontramodule of a separated left $\R$\+contramodule
is separated.

\begin{prop} \label{contramodules-covariant-systems-prop}
 In the context of
Proposition~\ref{discrete-modules-contramodules-from-systems}(b),
for any covariant\/ $\bfF$\+system of abelian groups $D$
the left\/ $\R$\+contramodule\/ $\PL(D)$ is separated.
 For any left\/ $\R$\+contramodule\/ $\C$, the covariant\/
$\bfF$\+system of abelian groups\/ $\CT(\C)$ is right exact.
 For any left\/ $\R$\+contramodule\/ $\C$, the abelian group\/
$\varprojlim_{\I\in\bfF}\C/(\I\tim\C)$ is the underlying abelian group
of the left\/ $\R$\+contramodule\/ $\PL(\CT(\C))$, and
$\lambda_{\R,\C}$ is the adjunction morphism.
 In particular, $\lambda_{\R,\C}$ is a morphism of left\/
$\R$\+contramodules, and its kernel\/ $\bigcap_{\I\in\bfF}\I\tim\C$ is
an\/ $\R$\+subcontramodule in\/~$\C$.
\end{prop}

\begin{proof}
 The first assertion is~\cite[Lemma~6.2(b)]{PR}, while the second one
is explained in~\cite[paragraphs preceding Lemma~6.2]{PR}.
 The covariant $\bfF$\+system $\CT(\C)$ is right exact for every
left $\R$\+contramodule $\C$, because the functor ${-}\ocn_\R\C\:
\discr\R\rarrow\Ab$ is right exact.
 The left $\R$\+contramodule $\E=\PL(D)$ is separated for any covariant
$\bfF$\+system $D$, because, in the notation of the proof of
Proposition~\ref{discrete-modules-contramodules-from-systems}(b),
one has $\psi_\I(e)=0$ for any $e\in\I\tim\E$, so
$e\in\bigcap_{\I\in\bfF}\I\tim\E$ implies $e=0$.
 The remaining assertions follow from the constructions of the functors
$\CT$ and $\PL$ and the construction of the adjunction between them
in Proposition~\ref{discrete-modules-contramodules-from-systems}(b).
\end{proof}

\begin{thm} \label{separated-vs-complete-contramodules-thm}
 Assume that a complete, separated right linear topology\/ $\bfF$ on
a ring\/ $\R$ has a countable base.
 Then all left\/ $\R$\+contramodules are complete, so
the adjunction/completion morphism $\lambda_{\R,\C}\:\C\rarrow
\PL(\CT(\C))$ is surjective for any left\/ $\R$\+contramodule\/~$\C$.
 A left\/ $\R$\+contramodule  is separated if and only if it
belongs to the essential image of the functor\/~$\PL$.
 A covariant\/ $\bfF$\+system of abelian groups is right exact
if and only if it belongs to the essential image of the functor\/~$\CT$.
 The restrictions of the adjoint functors\/ $\CT$ and\/ $\PL$ are
mutually inverse equivalences between the full subcategory of
separated left\/ $\R$\+contramodules in\/ $\R\contra$ and the full
subcategory of right exact covariant\/ $\bfF$\+systems of abelian
groups in~${}_\bfF\Ab$.
\end{thm}

\begin{proof}
 This is~\cite[Lemma~6.3]{PR}.
 Notice that the exposition in~\cite[Section~6]{PR} uses
direct-limit preserving covariant additive functors $\discr\R\rarrow\Ab$
in lieu of our covariant $\bfF$\+systems $\sQ_\bfF\rarrow\Ab$.
 One readily observes that the two approaches are essentially
the same, as no need to apply such functors to any but cyclic
discrete right $\R$\+modules ever arises in~\cite{PR}.
 Alternatively, one can notice that, as in the proof of
Proposition~\ref{discrete-modules-contravariant-systems},
the construction of the Kan extension
$$
 F_C(\N)=\varinjlim\nolimits_{\R/\I\to\N}C(\R/\I)\qquad
 \text{for every \,$\N\in\discr\R$}
$$
can be used to extend a covariant additive functor
$C\:\sQ_\bfF\rarrow\Ab$ to a coproduct-preserving covariant additive
functor $F_C\:\discr\R\rarrow\Ab$.
 The functor $F_C$ is right exact whenever the covariant
$\bfF$\+system $C$~is.
\end{proof}

 The following lemma is quite basic.
 It will be useful in the proofs of Proposition~\ref{orthogonality-prop}
and Theorem~\ref{weakly-cotorsion-obtainable}.

\begin{lem} \label{quotient-of-separated-lemma}
 Let\/ $\R$ be a complete, separated topological ring in a right
linear topology\/~$\bfF$.
 Then any left $\R$\+contramodule\/ $\C$ is the cokernel of
an injective morphism of separated left\/ $\R$\+contramodules\/
$\E\rarrow\D$.
\end{lem}

\begin{proof}
 For any left $\R$\+contramodule $\C$, there exists a set $X$ and
a surjective left $\R$\+contramodule morphism $\R[[X]]\rarrow\C$ onto
$\C$ from the free left $\R$\+contramodule $\R[[X]]$ (e.~g., one can
take $X=\C$ and the contraaction map $\pi_\C\:\R[[\C]]\rarrow\C$).
 Any free left $\R$\+contramodule $\D=\R[[X]]$ is separated, because
$\I\tim\R[[X]]=\I[[X]]$ for all $\I\in\bfF$, hence
$\bigcap_{\I\in\bfF}\I\tim\D=\bigcap_{\I\in\bfF}\I[[X]]=0$ in $\R[[X]]$
(since the topology $\bfF$ on $\R$ is separated by assumption, so
$\bigcap_{\I\in\bfF}\I=0$ in~$\R$).
 The kernel $\E$ of the surjective morphism $\D\rarrow\C$ is
an $\R$\+subcontramodule in a separated left $\R$\+contramodule $\D$,
hence $\E$ is also a separated left $\R$\+contramodule.
\end{proof}

 For any discrete right $\R$\+module $\N$ and any open right ideal
$\I\subset\R$, we denote by $\N_\I\subset\N$ the subgroup in
(the underlying abelian group of) $\N$ consisting of all the elements
annihilated by~$\I$.
 So one has $\N=\varinjlim_{\I\in\bfF}\N_\I$.

 The next lemma will be useful in the proof of
Theorem~\ref{fully-faithful-thm}.

\begin{lem} \label{embedding-into-hom-lemma}
 Let\/ $\R$ be a complete, separated topological ring in a right
linear topology\/~$\bfF$.
 Then a left $\R$\+contramodule\/ $\C$ is separated if and only if
it can be embedded as a subcontramodule into a left\/
$\R$\+contramodule of the form\/ $\Hom_\boZ(\N,\boQ/\boZ)$, where\/
$\N$ is a discrete right\/ $\R$\+module.
\end{lem}

\begin{proof}
 ``If'': for any associative ring $A$, an $\R$\+discrete
$A$\+$\R$\+bimodule $\N$, and a left $A$\+module $V$, the left
$\R$\+contramodule $\D=\Hom_A(\N,V)$ is separated.
 Indeed, for any open right ideal $\I\subset\R$, the subgroup
$\I\tim\D\subset\D$ consists of (some) $A$\+linear maps
$\N\rarrow V$ annihilating~$\N_\I$.
 Hence $\bigcap_{\I\in\bfF}\I\tim\D=0$.

 ``Only if'' (cf.~\cite[first proof of Corollary~7.8]{PR}): notice that
the class of all left $\R$\+contramodules of the form
$\Hom_\boZ(\N,\boQ/\boZ)$, where $\N\in\discr\R$, is closed under
products, as for any family of discrete right $R$\+modules $\N_\alpha$
one has $\prod_\alpha\Hom_\boZ(\N_\alpha,\boQ/\boZ)\simeq
\Hom_\boZ(\bigoplus_\alpha\N_\alpha,\>\boQ/\boZ)$.
 Thus, given a separated left $\R$\+contramodule $\C$, it suffices
to show that for every element $c\in\C$ there exists a discrete
right $\R$\+module $\N_c$ and a left $\R$\+contramodule morphism
$g_c\:\C\rarrow\Hom_\boZ(\N_c,\boQ/\boZ)$ such that $g_c(c)\ne0$.
{\emergencystretch=1.5em\par}

 Choose an open right ideal $\I\subset\R$ such that $c\notin\I\tim\C$
and an abelian group homomorphism $f_c\:\C/(\I\tim\C)\rarrow\boQ/\boZ$
such that $f_c(\bar c)\ne0$ (where $\bar c$~is the image of~$c$ in
$\C/(\I\tim\C)$).
 Set $\N_c=\R/\I$.
 The natural isomorphism $\Hom^\R(\C,\Hom_\boZ(\N_c,\boQ/\boZ))\simeq
\Hom_\boZ(\N_c\ocn_\R\C,\>\boQ/\boZ)$ from
Section~\ref{prelim-contratensor} allows to assign to the abelian
group homomorphism $f_c\:\N_c\ocn_\R\C\simeq\C/(\I\tim\C)\rarrow
\boQ/\boZ$ a left $\R$\+contramodule morphism $g_c\:\C\rarrow
\Hom_\boZ(\N_c,\boQ/\boZ)$.
 By construction, the element $g_c(c)$ is an abelian group homomorphism
$\N_c\rarrow\boQ/\boZ$ taking the element $\bar 1=1+\I\in\N_c$ to
the element $f_c(\bar c)\in\boQ/\boZ$.
 Hence $g_c(c)\ne0$ in $\Hom_\boZ(\N_c,\boQ/\boZ)$.
\end{proof}

 The following proposition is a stronger version of
Lemma~\ref{embedding-into-hom-lemma} (holding in the narrower
generality of a countable topology base).
 We will use it in the proof of Proposition~\ref{orthogonality-prop}.

\begin{prop} \label{separated-kernel-of-homs}
 Let\/ $\R$ be a complete, separated topological ring in a right
linear topology\/ $\bfF$ with a countable base.
 Then any separated left\/ $\R$\+contramodule can be presented as
the kernel of an\/ $\R$\+contramodule morphism between two left\/
$\R$\+contramodules of the form\/ $\Hom_\boZ(\N',\boQ/\boZ)$ and\/
$\Hom_\boZ(\N'',\boQ/\boZ)$, where\/ $\N'$ and $\N''$ are discrete
right\/ $\R$\+modules.
\end{prop}

\begin{proof}[First proof]
 It suffices to show that any separated left $\R$\+contramodule $\C$
can be embedded into a left $\R$\+contramodule $\D$ of the form
$\D=\Hom_\boZ(\N,\boQ/\boZ)$, where $\N\in\discr\R$, in such a way
that the quotient $\R$\+contramodule $\D/\C$ is also separated.
 (Then it remains to apply Lemma~\ref{embedding-into-hom-lemma} to
the left $\R$\+contramodule~$\D/\C$.)

 The argument is based on the following simple lemma.

\begin{lem}
 Let\/ $\R$ be a complete, separated topological ring in a right
linear topology\/ $\bfF$ with a countable base\/~$\bfB$.
 Let\/ $\D$ be a separated left\/ $\R$\+contramodule and\/
$\C\subset\D$ be an\/ $\R$\+subcontramodule.
 Suppose that for every open right ideal $\I\in\bfB$ the intersection\/
$\C\cap(\I\tim\D)\subset\D$ is equal to\/ $\I\tim\C$.
 Then the quotient\/ $\R$\+contramodule\/ $\D/\C$ is separated.
\end{lem}

\begin{proof}
 For any complete, separated topological ring $\R$ with a right linear
topology, any short exact sequence of left $\R$\+contramodules
$0\rarrow\C\rarrow\D\rarrow\E\rarrow0$, and any open right ideal
$\I\subset\R$, the short sequence of abelian groups
$$
 \C/(\I\tim\C)\lrarrow\D/(\I\tim\D)\lrarrow\E/(\I\tim\E)\lrarrow0
$$
is right exact.
 When one has $\I\tim\C=\C\cap(\I\tim\D)\subset\D$, this short sequence
is exact at is leftmost term, too.

 In the situation at hand, set $\E=\D/\C$ and consider the projective
system of short exact sequences
$$
 0\lrarrow\C/(\I\tim\C)\lrarrow\D/(\I\tim\D)
 \lrarrow\E/(\I\tim\E)\lrarrow0
$$
indexed by the poset $\bfB$ of open right ideals~$\I$.
 These are projective systems of abelian groups and surjective
morphisms between them, indexed by a countable directed poset.
 Hence the passage to the projective limits preserves exactness,
and we have a short exact sequence of abelian groups
$$
 0\lrarrow\varprojlim_{\I\in\bfB}\C/(\I\tim\C)\lrarrow
 \varprojlim_{\I\in\bfB}\D/(\I\tim\D)\lrarrow
 \varprojlim_{\I\in\bfB}\E/(\I\tim\D)\lrarrow0.
$$
 The natural morphism to the latter short exact sequence from
the short exact sequence $0\rarrow\C\rarrow\D\rarrow\E\rarrow0$
is an isomorphism at the middle term (since $\D$ is separated by
assumption and all left $\R$\+contramodules are complete by
Theorem~\ref{separated-vs-complete-contramodules-thm}) and
at the leftmost term (since $\C$ is separated as a subcontramodule
of a separated left $\R$\+contramodule~$\D$).
 It follows that this map of short exact sequences is also
an isomorphism at the rightmost terms, that is $\E$ is a separated
left $\R$\+contramodule.
\end{proof}

 Now we can finish the first proof of
Proposition~\ref{separated-kernel-of-homs}.
 It suffices to embed our separated left $\R$\+contramodule $\C$ into
a left $\R$\+contramodule $\D$ of the form $\D=\Hom_\boZ(\N,\boQ/\boZ)$
in such a way that $\C\cap(\I\tim\D)=\I\tim\C$ for every open right
ideal $\I\subset\R$.
 Recall from the proof of Lemma~\ref{embedding-into-hom-lemma} that
the subgroup $\I\tim\D$ is contained in the subgroup all abelian
group homomorphisms $\N\rarrow\boQ/\boZ$ annihilating the subgroup
$\N_\I\subset\N$.
 Thus it suffices to construct a discrete right $\R$\+module $\N$
and a left $\R$\+contramodule morphism $g\:\C\rarrow
\Hom_\boZ(\N,\boQ/\boZ)$ in such a way that for every open right ideal
$\I\subset\R$ and every element $c\in\C$, \ $c\notin\I\tim\C$
the abelian group homomorphism $g(c)\:\N\rarrow\boQ/\boZ$ does
\emph{not} annihilate $\N_\I$, i.~e., there exists an element $b\in\N$
for which $b\I=0$ in $\N$ and $g(c)(b)\ne0$ in $\boQ/\boZ$.

 For this purpose, for every open right ideal $\I\subset\R$ and
every element $c\in\C\setminus\I\tim\C$ we choose an abelian group
homomorphism $f_{\I,c}\:\C/(\I\tim\C)\rarrow\boQ/\boZ$ such that
$f_{\I,c}(c+\I\tim\C)\ne0$, and consider the related left
$\R$\+contramodule morphism $g_{\I,c}\:\C\rarrow
\Hom_\boZ(\N_{\I,c},\boQ/\boZ)$, where $\N_{\I,c}=\R/\I$ (as in
the proof of Lemma~\ref{embedding-into-hom-lemma}).
 Then the element $1+\I\in\R/\I=\N_{\I,c}$ is annihilated by~$\I$ and
one has $g_{\I,c}(c)(1+\I)=f_{\I,c}(c+\I\tim\C)\ne0$ in $\boQ/\boZ$.
 It remains to take $\N$ to be the direct sum of the discrete right
$\R$\+modules $\N_{\I,c}$ over all the pairs $(\I,c)$ as above, and
$g\:\C\rarrow\Hom_\boZ(\N,\boQ/\boZ)$ to be the morphism with
the components~$g_{\I,c}$.
\end{proof}

\begin{proof}[Second proof of
Proposition~\ref{separated-kernel-of-homs}]
 We will prove the following more general result, not depending on
the countability assumption on the topology of~$\R$: for any
complete, separated topological ring $\R$ in a right linear
topology~$\bfF$, and any covariant $\bfF$\+system of abelian groups $D$,
the left $\R$\+contramodule $\PL(D)$ can be presented as the kernel
of an $\R$\+contramodule morphism between two left $\R$\+contramodules
of the form $\Hom_\boZ(\N',\boQ/\boZ)$ and $\Hom_\boZ(\N'',\boQ/\boZ)$
in a certain functorial way.
 In particular, if $\C$ is a separated and complete left
$\R$\+contramodule, then the left $\R$\+contramodule $\C$ is isomorphic
to $\PL(\CT(\C))$ by
Proposition~\ref{contramodules-covariant-systems-prop}, so it will
follow that $\C$ can be presented as the kernel of a morphism between
two left $\R$\+contramodules of the desired form.
 Then it will remain to recall that, when the topology $\bfF$ on $\R$
has a countable base, all left $\R$\+contramodules are complete
by Theorem~\ref{separated-vs-complete-contramodules-thm}.

 Indeed, for any abelian group $V$ let us denote by $V^+$ the abelian
group $V^+=\Hom_\boZ(V,\boQ/\boZ)$.
 In particular, if $D\:\sQ_\bfF\rarrow\Ab$ is a covariant $\bfF$\+system
of abelian groups, then $D^+\:\sQ_\bfF^\sop\rarrow\Ab$ is
a contravariant $\bfF$\+system, and vice versa.
 If $M\:\sQ_\bfF^\sop\rarrow\Ab$ is a contravariant $\bfF$\+system of
abelian groups and $M^+\:\sQ_\bfF\rarrow\Ab$ is the dual covariant
$\bfF$\+system, then the left $\R$\+contramodule $\PL(M^+)$ can be
obtained by applying the functor $\Hom_\boZ({-},\boQ/\boZ)$ to
the discrete right $\R$\+module $\IL(M)$, i.~e., $\PL(M^+)=\IL(M)^+$
(since the functor $V\longmapsto V^+$ takes direct limits of abelian
groups to inverse limits of abelian groups).

 Now for any covariant $\bfF$\+system of abelian groups $D$ 
we have a natural short, left exact sequence of covariant
$\bfF$\+systems of abelian groups
$$
 0\lrarrow D\lrarrow D^{++}\lrarrow(D^{++}/D)^{++}.
$$
 The functor $\PL$ is left exact (because it is a right adjoint, or
since the functor of inverse limit is left exact); hence we obtain
a left exact sequence of left $\R$\+contramodules
$$
 0\lrarrow\PL(D)\lrarrow\PL(D^{++})\lrarrow\PL((D^{++}/D)^{++}).
$$
 It remains to recall that, according to the above, we have
$\PL(D^{++})=\IL(D^+)^+$ and $\PL((D^{++}/D)^{++})=
\IL((D^{++}/D)^+)^+$.
\end{proof}

\Section{Fully Faithful Contramodule Forgetful Functors}
\label{fully-faithful-secn}

 Let $R$ be an associative ring, $\bF$ be a right linear topology on
$R$, and $\R=\varprojlim_{I\in\bF}R/I$ be the completion of $R$ with
respect to $\bF$, viewed as a topological ring in the projective
limit topology~$\bfF$.
 Then we have the abelian category of left $\R$\+contramodules
$\R\contra$ (see Section~\ref{prelim-contramodules}), which is endowed
with an exact forgetful functor $\R\contra\rarrow\R\modl$.

 We also have the functor of restriction of scalars
$\rho_*\:\R\modl\rarrow R\modl$ induced by the natural morphism of
(topological) rings $\rho\:R\rarrow\R$.
 In this section we discuss conditions under which the composition
of these two forgetful functors
$$
 \R\contra\lrarrow\R\modl\lrarrow R\modl
$$
is a fully faithful functor $\R\contra\rarrow R\modl$.
 We follow the approach of~\cite[Theorem~1.1]{Psm}
and~\cite[Section~3]{Pper}, improving upon the results obtained there.

 As it was discussed in Section~\ref{F-systems-secn}, there is a natural
bijection between the set $\bF$ of all open right ideals $I\subset R$
and the set $\bfF$ of all open right ideals $\I\subset\R$.
 The ring homomorphism $\rho\:R\rarrow\R$ induces an isomorphism
of right $R$\+modules $R/I\simeq\R/\I$.

 We will say that an open right ideal $I\subset R$ is \emph{strongly
generated} (or, in a different language, the corresponding open right
ideal $\I\subset\R$ is \emph{strongly generated by elements coming
from~$R$}) if for every set $X$ the subgroups $\I\tim\R[[X]]=\I[[X]]
\subset\R[[X]]$ and $I\R[[X]]=I\cdot\R[[X]]\subset\R[[X]]$ coincide
in the group $\R[[X]]$, that is
$$
 I\R[[X]]=\I[[X]]\subset\R[[X]]
$$
(see Section~\ref{prelim-contratensor} for the notation).
 The inclusion $I\R[[X]]\subset\I[[X]]$ always holds for any open
right ideal $I\subset R$; the right ideal $I$ is said to be strongly
generated if the inverse inclusion holds as well.

\begin{lem} \label{strongly-generated-ideal-lemma1}
 Let $R$ be a topological ring with a right linear topology and
$I\subset R$ be an open right ideal, and $\I\subset\R$ be the related
open right ideal in the completion\/ $\R$ of the topological ring~$R$.
 Then the following conditions are equivalent:
\begin{enumerate}
\renewcommand{\theenumi}{\roman{enumi}}
\item the ideal $I\subset R$ is strongly generated;
\item for every left\/ $\R$\+contramodule\/ $\C$, one has
$I\C=\I\tim\C\subset\C$.
\end{enumerate}
\end{lem}

\begin{proof}
 (ii)\,$\Longrightarrow$\,(i) Given a set $X$, take $\C=\R[[X]]$.

 (i)\,$\Longrightarrow$\,(ii) For any left $\R$\+contramodule $\C$,
there exists a set $X$ such that there is a surjective left
$\R$\+contramodule morphism $\R[[X]]\rarrow\C$ (e.~g., one can take
$X=\C$ and the contraaction morphism $\pi_\C\:\R[[\C]]\rarrow\C$).
 It remains to observe that for any surjective left $\R$\+contramodule
morphism $f\:\D\rarrow\C$ one has $I\C=f(I\D)$ and $\I\tim\C=
f(\I\tim\D)$, so the equation $I\D=\I\tim\D$ implies $I\C=\I\tim\C$.
\end{proof}

\begin{thm} \label{fully-faithful-thm}
 Let $R$ be a topological ring with a right linear topology\/ $\bF$
having a countable base\/~$\bB$.
 Then the following conditions are equivalent:
\begin{enumerate}
\renewcommand{\theenumi}{\roman{enumi}}
\item all the open right ideals $I\subset R$ are strongly generated;
\item all the open right ideals $J\in\bB$ are strongly generated in~$R$;
\item the forgetful functor\/ $\R\contra\rarrow R\modl$ is fully
faithful.
\end{enumerate}
\end{thm}

\begin{proof}
 The implications (i)\,$\Longleftrightarrow$\,(ii)%
\,$\Longleftarrow$\,(iii) do not depend on the assumption of
countability of~$\bB$; the proof of the implication
(ii)\,$\Longrightarrow$\,(iii) below does.

 (i)\,$\Longrightarrow$\,(ii) Obvious.

 (ii)\,$\Longrightarrow$\,(i) Let $J\subset I$ be two embedded open
right ideals in $R$, and let $\J\subset\I$ be the two related open
right ideals in~$\R$.
 Then any $X$\+indexed family of elements in $\I$ converging to zero
in the topology of $\R$ has all but a finite subset of its elements
belonging to~$\J$.
 Furthermore, $\I=\rho(I)+\J\subset\R$, hence
$\I[[X]]=\rho(I)[X]+\J[[X]]\subset\R[[X]]$.
 Therefore, the equation $J\R[[X]]=\J[[X]]$ implies $I\R[[X]]=\I[[X]]$.

 (iii)\,$\Longrightarrow$\,(i)
 As in Section~\ref{prelim-contramodules}, we denote by $\Hom^\R(\C,\D)$
the group of all morphisms $\C\rarrow\D$ in the category of left
$\R$\+contramodules $\R\contra$; while, as usual, $\Hom_R(C,D)$
denotes the group of all morphisms $C\rarrow D$ in $R\modl$.
 Then for any left $R$\+module $C$ and right $R$\+module $N$ one has
$$
 \Hom_R(C,\Hom_\boZ(N,\boQ/\boZ))=\Hom_\boZ(N\ot_RC,\>\boQ/\boZ),
$$
while for any left $\R$\+contramodule $\C$ and discrete right
$\R$\+module $\N$ we have
$$
 \Hom^\R(\C,\Hom_\boZ(\N,\boQ/\boZ))=\Hom_\boZ(\N\ocn_\R\C,\>\boQ/\boZ)
$$
(see Section~\ref{prelim-contratensor}).
 Hence the natural surjective map of abelian groups $\N\ot_R\C\rarrow
\N\ocn_\R\C$ is an isomorphism for any left $\R$\+contramodule $\C$
and any discrete right $\R$\+module $\N$ whenever the forgetful
functor $\R\contra\rarrow R\modl$ is fully faithful.
 In particular, one has
$$
 \C/I\C=R/I\ot_R\C=\R/\I\ot_R\C \simeq \R/\I\ocn_\R\C=\C/(\I\tim\C),
$$
and it remains to take $\C=\R[[X]]$.

 (ii)\,$\Longrightarrow$\,(iii)
 Let $\C$ and $\D$ be two left $\R$\+contramodules, and let
$f\:\C\rarrow\D$ be a left $R$\+module morphism between them.
 Choose a set $X$ together with a surjective left $\R$\+contramodule
morphism $\R[[X]]\rarrow\C$.
 In order to show that the map~$f$ is a left $\R$\+contramodule
morphism, it suffices to check that so is the composition
$\R[[X]]\rarrow\C\rarrow\D$.
 Let us denote this composition by $f'\:\R[[X]]\rarrow\D$.

 Restricting~$f'$ to the subset $X\subset\R[[X]]$, we obtain a map
of sets $\bar f'\:X\rarrow\D$.
 The group of all left $\R$\+contramodule morphisms $\R[[X]]\rarrow\D$
is bijective to the group of all maps of sets $X\rarrow\D$; hence
there exists a unique left $\R$\+contramodule morphism
$f''\:\R[[X]]\rarrow\D$ such that $f''|_X=\bar f'$.
 Consider the difference $g=f'-f''$; then $g\:\R[[X]]\rarrow\D$
is a left $R$\+module morphism such that $g|_X=0$, and it remains
to check that $g=0$.

 Let $\E\subset\D$ denote the $\R$\+subcontramodule in $\D$ generated
by $\im g$ (that is, the intersection of all $\R$\+subcontramodules
in $\D$ containing~$\im g$).
 Then $\E$ is a left $\R$\+contramodule, $g\:\R[[X]]\rarrow\E$ is
a left $R$\+module morphism such that $g|_X=0$, there are no proper
$\R$\+subcontramodules in $\E$ containing $\im g$, and it remains
to prove that $\E=0$.

 According to Proposition~\ref{contramodules-covariant-systems-prop},
the kernel $\E'=\ker\lambda_{\R,\E}=\bigcap_{\I\in\bfF}\I\tim\E$ of
the completion morphism $\lambda_{\R,\E}\:\E\rarrow
\varprojlim_{\I\in\bfF}\E/(\I\tim\E)$ is an $\R$\+subcontramodule in
$\E$, and the quotient $\R$\+contramodule $\E''=\E/\E'$ is separated.
 Moreover, by the contramodule Nakayama lemma~\cite[Lemma~6.14]{PR}
the left $\R$\+contramodule $\E$ vanishes whenever its (maximal
separated) quotient contramodule~$\E''$ does.
 (This is where the assumption of a countable base of the topology
of $\R$ is needed.)

 Hence if $\E\ne0$, then $\E''\ne0$, \ $\E'\ne\E$, and by construction
the image of the left $R$\+module morphism $g\:\R[[X]]\rarrow\E$
is not contained in~$\E'$.
 The composition $g''\:\R[[X]]\rarrow\E''$ of the map~$g$ with
the natural surjection $\E\rarrow\E''$ is a left $R$\+module morphism
such that $g''|_X=0$ and $\im g''$ is not contained in any proper
$\R$\+subcontramodule of~$\E''$.
 It remains to show that $\E''=0$, and for this purpose it suffices
to check that $g''=0$, or even that $g''$~is actually
an $\R$\+contramodule morphism (as any $\R$\+contramodule morphism
from a free left $\R$\+contramodule taking the generators to zero
vanishes as a whole).

 By Lemma~\ref{embedding-into-hom-lemma}, the left $\R$\+contramodule
$\E''$ can be embedded into a left $\R$\+contramodule of the form
$\Hom_\boZ(\N,\boQ/\boZ)$, where $\N$ is a discrete right
$\R$\+module.
 In fact, following the proof of Lemma~\ref{embedding-into-hom-lemma},
one take $\N$ to be a direct sum of (sufficiently many) copies of
the discrete right $\R$\+modules $\R/\I$, where $\I$ ranges over any
topology base $\bfB$ of~$\R$.
 We take $\bfB$ to be the topology base of $\R$ corresponding to
our given topology base $\bB$ of~$R$.

 It remains to show that the composition $\R[[X]]\rarrow
\Hom_\boZ(\N,\boQ/\boZ)$ of the morphism $g''\:\R[[X]]\rarrow\E''$
with the embedding $\E''\rarrow\Hom_\boZ(\N,\boQ/\boZ)$ is a left
$\R$\+contramodule morphism, and for this purpose it suffices to
check that every left $R$\+module map $\R[[X]]\rarrow
\Hom_\boZ(\N,\boQ/\boZ)$ is actually an $\R$\+contramodule morphism.
 Indeed, for any left $\R$\+contramodule $\C$ we have
$$
 \Hom_R(\C,\Hom_\boZ(\R/\I,\boQ/\boZ))=
 \Hom_R(\C,\Hom_\boZ(R/I,\boQ/\boZ))\simeq
 \Hom_\boZ(\C/I\C,\boQ/\boZ)
$$
and
$$
 \Hom^\R(\C,\Hom_\boZ(\R/\I,\boQ/\boZ))\simeq
 \Hom_\boZ(\R/\I\ocn_\R\C,\>\boQ/\boZ)=
 \Hom_\boZ(\C/(\I\tim\C),\boQ/\boZ),
$$
so the equation $I\C=\I\tim\C$ for all $I\in\bB$ implies
$\Hom_R(\C,\Hom_\boZ(\N,\boQ/\boZ))=\Hom^\R(\C,\Hom_\boZ(\N,\boQ/\boZ))$.
\end{proof}

\begin{rem}
 We do \emph{not} know whether the implication
(ii)\,$\Longrightarrow$\,(iii) in Theorem~\ref{fully-faithful-thm}
holds true without the assumption of countability of~$\bB$.

 Following the above proof, the countability assumption was only used
in order to invoke the contramodule Nakayama lemma, claiming that any
nonzero $\R$\+contramodule has a nonzero separated quotient
$\R$\+contramodule.
 This form of the contramodule Nakayama lemma is \emph{not} true
without a countable topology base, generally speaking.
 Indeed, let $\R=R$ be the ring of (commutative) polynomials in
an uncountable set of variables~$x_i$ over a field~$k$, and let
$S\subset R$ be the multiplicative subset generated by
the elements~$x_i$.
 Endow $\R$ with the \emph{$S$\+topology}, in which the ideals
$s\R$, \,$s\in S$, form a base of neighborhoods of zero.
 By~\cite[Proposition~1.16]{GT}, \,$\R$ is a complete, separated
topological ring.
 One easily observes that no infinite family of nonzero elements in
$\R$ converges to zero in the $S$\+topology; so an $\R$\+contramodule
is the same thing as an $\R$\+module and the forgetful functor
$\R\contra\rarrow\R\modl$ is an equivalence of categories.
 In particular, the localization $\C=\R[S^{-1}]$ is
an $\R$\+contramodule for which $\bigcap_{s\in S}s\C=\C$; so $\C$
has no nonzero separated quotient $\R$\+contramodules.
 Still, the implication (ii)\,$\Longrightarrow$\,(iii) remains true
in this case (both (ii) and~(iii) are true).

 In some special situations, the contramodule Nakayama lemma is provable
without the countability assumption; notably for the topological
algebras dual to coalgebras over fields~\cite[Corollary~A.2]{Psemi}.
 Hence Theorem~\ref{fully-faithful-thm} holds for all such topological
algebras (cf.~\cite[Theorem~1.1]{Psm}).
 Another such special situation is that of contramodules over
the completion $\R$ of an associative ring $R$ with respect to
the $S$\+topology for some central multiplicative subsets $S\subset R$.
 In this context, a version of Theorem~\ref{fully-faithful-thm} holds
under a certain assumption not unrelated to countability but weaker than
that of a countable topology base; see~\cite[Example~3.7\,(1)]{Pper}.

 An example of topological ring/algebra $\R$ with a countable topology
base (dual to a certain coalgebra of countable dimension over a field)
for which the forgetful functor $\R\contra \rarrow\R\modl$ is \emph{not} fully  faithful (so none of the conditions (i\+-iii) is true) can be
found in~\cite[Section~A.1.2]{Psemi}.
\end{rem}

 The following lemma explains the intuition behind the ``strongly
generated ideal'' terminology.
 Extending the terminology of Section~\ref{prelim-contramodules},
even for a possibly nonseparated and noncomplete topological group $A$
we will say that an $X$\+indexed family of elements $a_x\in A$
converges to zero in the topology of $A$ if for every open subgroup
$U\subset A$ the set of all $x\in X$ such that $a_x\notin U$ is finite.

\begin{lem} \label{strongly-generated-ideal-lemma2}
 Let $R$ be a topological ring with a right linear topology with
a countable base, and let $I\subset R$ be an open right ideal.
 Then the following two conditions are equivalent:
\begin{enumerate}
\renewcommand{\theenumi}{\roman{enumi}}
\item the ideal $I\subset R$ is strongly generated;
\item for any set $X$ and any $X$\+indexed family of elements
$r_x\in I$ converging to zero in the topology of $R$ there exists
a finite set of elements $s_1$,~\dots, $s_m\in I$ and
a $(\{1,\dotsc,m\}\times X)$\+indexed family of elements
$t_{j,x}\in\R$ converging to zero in the topology of\/ $\R$ such that
$$
 \rho(r_x)=\sum_{j=1}^m \rho(s_j)t_{j,x}\in\R
 \qquad\text{for every \,$x\in X$.}
$$
\end{enumerate}
\end{lem}

\begin{proof}
 Let us start from discussing the particular case when the topological
ring $R$ is separated and complete, so $\R=R$, \ $\I=I$, and $\rho$~is
the identity map.
 In this case, the assertion of the lemma is completely tautological
and no countability assumption on the topology of $\R$ is needed.
 The datum of an $X$\+indexed family of elements $r_x\in\I$ converging
to zero in the topology of $\R$ is equivalent to the datum of
an element $\mathfrak r=\sum_{x\in X}r_xx\in\I[[X]]\subset\R[[X]]$.
 By the definition, the element~$\mathfrak r$ belongs to $I\R[[X]]
\subset\I[[X]]$ if and only if there exists a finite set of elements
$s_1$,~\dots, $s_m\in I$ and a finite set of elements
$\mathfrak t_1$,~\dots, $\mathfrak t_m\in\R[[X]]$ such that
$\mathfrak r=\sum_{j=1}^ms_j\mathfrak t_j$ in $\R[[X]]$.
 The datum of an element~$\mathfrak t_j=\sum_{x\in X} t_{j,x}x
\in\R[[X]]$ is equivalent to the datum of an $X$\+indexed family of
elements $t_{j,x}\in\R$ converging to zero in the topology of $\R[[X]]$;
and the equation $\mathfrak r=\sum_{j=1}^ms_j\mathfrak t_j$ in
$\R[[X]]$ is equivalent to the $X$\+indexed family of equations
$r_x=\sum_{j=1}^ms_jt_{j,x}$ for all $x\in X$.

 In the general case, the implication (i)\,$\Longrightarrow$\,(ii)
still does not require the countability assumption on the topology.
 Given an $X$\+indexed family of elements $r_x\in I$ converging to
zero in $R$, consider the element $\mathfrak r=\sum_{x\in X}\rho(r_x)x
\in\R[[X]]$.
 According to~(i), there exist finite sets of elements
$s_1$,~\dots, $s_m\in I$ and $\mathfrak t_1$,~\dots, $\mathfrak t_m
\in\R[[X]]$ such that $\mathfrak r=\sum_{j=1}^m\rho(s_j)\mathfrak t_j$
in $\R[[X]]$, which is a restatement of~(ii).

 To prove (ii)\,$\Longrightarrow$\,(i), one observes that, whenever
the topology of $R$ has a countable base, any element $p\in\R$ can be
presented as a sum $p=\sum_{i=1}^\infty\rho(t_i)$ for some elements
$t_1$, $t_2$, $t_3$,~\dots~$\in R$ (where the infinite sum is understood
as the limit of finite partial sums in the topology of~$\R$).
 Similarly, any element $q\in\I=\varprojlim_{J\in\bF,\,J\subset I}I/J$
can be presented as a countably infinite sum
$q=\sum_{i=1}^\infty\rho(r_i)$ for some sequence of elements
$r_1$, $r_2$, $r_3$,~\dots~$\in I$.
 Moreover, if $\mathfrak q=\sum_{x\in X}q_xx$ is an element of
$\I[[X]]$ (so the $X$\+indexed family of elements $q_x\in\I$ converges
to zero in the topology of~$\R$), then one can choose for every
$x\in X$ a sequence of elements $(r_{x,i}\in I)_{i=1}^\infty$ such that
$q_x=\sum_{i=1}^\infty\rho(r_{x,i})\in\I$ for every $x\in X$ and
the whole doubly indexed family of elements $r_{x,i}\in I$ converges
to zero in the topology of~$R$.

 By~(ii), there exists a finite set of elements $s_1$,~\dots, $s_m\in I$
and a family of elements $t_{j,x,i}\in\R$, indexed by the triples of
indices $1\le j\le m$, \ $x\in X$, and $i\in\boZ_{\ge1}$ and converging
to zero in the topology of $\R$, such that
$$
 \rho(r_{x,i})=\sum_{j=1}^m\rho(s_j)t_{j,x,i}\qquad
 \text{for all $x\in X$ and $i\in\boZ_{\ge1}$.}
$$
 Passing to the sum over~$i$ (understood as the limit of finite partial
sums in the topology of~$\R$), we obtain the equations
$$
 q_x=\sum_{j=1}^m\rho(s_j)\sum_{i=1}^\infty t_{j,x,i}
 \qquad\text{for all \,$x\in X$},
$$
that is, setting $p_{j,x}=\sum_{i=1}^\infty t_{j,x,i}\in\R$,
$$
 q_x=\sum_{j=1}^m\rho(s_j)p_{j,x}
 \qquad\text{for all \,$x\in X$.}
$$
 It remains to observe that the family of elements~$(p_{j,x})_{j,x}$
converges to zero in the topology of $\R$ (because the family of
elements~$(t_{j,x,i})_{j,x,i}$ does) and conclude that
$$
 \mathfrak q=\sum_{x\in X}q_xx=\sum_{j=1}^\infty\sum_{x\in X}
 \rho(s_j)p_{j,x}x=\sum_{j=1}^m\rho(s_j)\mathfrak p_j,
$$
where $\mathfrak p_j=\sum_{x\in X}p_{j,x}x\in\R[[X]]$.
 Thus $\mathfrak q\in I\R[[X]]$.
\end{proof}

\begin{rem}
 One can say that an open right ideal $I$ in a topological ring $R$
with a right linear topology is \emph{strongly finitely generated} if
there exists a finite set of elements $s_1$,~\dots, $s_m\in I$ that
can be used in the condition of
Lemma~\ref{strongly-generated-ideal-lemma2}(ii) for all
zero-convergent families of elements $r_x\in I$.
 Notice that a strongly generated open right ideal does not need to be
strongly finitely generated (e.~g., a right ideal in a discrete
associative ring does not need to be finitely generated).
 On the other hand, one easily observes that a strongly generated
open right ideal in a topological ring $R$ is strongly finitely
generated whenever it is finitely generated as a right ideal
in an abstract associative ring~$R$.

 Moreover, let us say that a subset of elements $Z\subset I$ of an open
right ideal $I\subset R$ is a \emph{set of strong generators} of $I$
if for any zero-convergent family of elements $(r_x\in I)_{x\in X}$
there exists a finite set of elements $s_1$,~\dots, $s_m\in Z$ that
can be used in the condition of
Lemma~\ref{strongly-generated-ideal-lemma2}(ii) for the family of
elements $r_x\in I$.
 Then any set of generators of a strongly generated right ideal is
a set of strong generators.
 Indeed, the equality $I=\sum_{z\in Z}zR$ implies $I\R[[X]]=
\sum_{z\in Z}z\R[[X]]$.
\end{rem}

\begin{thm} \label{gabriel-strongly-finitely-generated-thm}
 Let $R$ be an associative ring and\/ $\bG$ be a right Gabriel
topology on $R$ having a countable base\/ $\bB$ consisting of
finitely generated right ideals.
 Then all the right ideals $I\in\bB$ are strongly (finitely) generated
in the topological ring~$R$.
\end{thm}

\begin{proof}
 Let $s_1$,~\dots,~$s_m$ be a finite set of generators of an open
right ideal $I\subset R$, and let $r_x\in I$ be an $X$\+indexed
family of elements converging to zero in the Gabriel topology $\bG$
of~$R$.
 One can assume the set $X$ to be countable (as any family of elements
converging to zero in $R$ vanishes outside of a countable subset of
indices) and presume an identification $X=\boZ_{\ge1}$ to be chosen.

 One can also choose a (nonstrictly) decreasing sequence of open right
ideals $J_1\supset J_2\supset J_3\supset\dotsb$ in $R$ such that
$J_i\in\bB$ for all $i\ge1$ and the set of open right ideals
$\bJ=\{J_1,J_2,J_2,\dotsc\}$ is a base of the topology~$\bG$.
 By Lemma~\ref{finitely-generated-gabriel-lemma}, for every $i\ge1$
the right ideal $H_i=s_1J_i+\dotsb+s_mJ_i\subset R$ belongs to~$\bG$.

 Choose an integer $x_1\ge1$ such that $r_x\in H_1$ for all $x>x_1$.
 Proceeding by induction, for every $i\ge 2$ choose an integer
$x_i>x_{i-1}$ such that for every $x>x_i$ one has $r_x\in H_i$.
 For every $x\le x_1$, choose elements $u_{1,x}$,~\dots, $u_{m,x}\in R$
such that $r_x=\sum_{j=1}^m s_ju_{j,x}$.
 For every $x_i<x\le x_{i+1}$, choose elements $u_{1,x}$,~\dots,
$u_{m,x}\in J_i$ such that $r_x=\sum_{j=1}^m s_ju_{j,x}$.
 Set $t_{j,x}=\rho(u_{j,x})\in\R$.
 Now the family of elements $(t_{j,x})_{j,x}$ converges to zero in
the projective limit topology $\bfG$ of $\R$ (since the family of
elements $(u_{j,x})_{j,x}$ converges to zero in the topology $\bG$ of
$R$, because $\bJ$ is a base of the topology~$\bG$), and one has
$\rho(r_x)=\sum_{j=1}^\infty\rho(s_j)t_{j,x}$ for every $x\in X$, as
desired in Lemma~\ref{strongly-generated-ideal-lemma2}(ii).
\end{proof}

\begin{cor} \label{countable-gabriel-fully-faithful-cor}
 Let $R$ be an associative ring, $\bG$ be a right Gabriel topology
on $R$ having a countable base consisting of finitely generated
right ideals, and let\/ $\R=\varprojlim_{I\in\bG}R/I$ be the completion
of $R$ with respect to\/ $\bG$, viewed as a topological ring in
the projective limit topology\/~$\bfG$.
 Then the forgetful functor
$$
 \R\contra\lrarrow R\modl
$$
is fully faithful.
 Consequently, the natural surjective map of abelian groups
$$
 \N\ot_R\C\lrarrow\N\ocn_\R\C
$$
is an isomorphism for any left\/ $\R$\+contramodule\/ $\C$ and
any discrete right\/ $\R$\+module\/~$\N$.
 In particular, one has
$$
 I\C=\I\tim\C\subset\C
$$
for any left\/ $\R$\+contramodule\/ $\C$, any open right ideal $I\in\bG$
in $R$, and the related open right ideal\/ $\I\in\bfG$ in\/~$\R$.
\end{cor}

\begin{proof}
 The first assertion follows from
Theorem~\ref{fully-faithful-thm}\,(ii)\,$\Longrightarrow$\,(iii)
and Theorem~\ref{gabriel-strongly-finitely-generated-thm}.
 The second assertion follows from the first one (see the proof of
Theorem~\ref{fully-faithful-thm}\,(iii)\,$\Longrightarrow$\,(i)).
 The third assertion is provided by
Lemma~\ref{strongly-generated-ideal-lemma1} (it is also a particular
case of the second one corresponding to the choice of the discrete
right $\R$\+module $\N=\R/\I=R/I$).
\end{proof}

\Section{Extending Full Abelian Subcategories}
\label{extending-subcategories-secn}

 Given a topological ring $R$ with a right linear topology $\bF$
having a countable base consisting of strongly generated open right
ideals, and denoting by $\R$ the completion of $R$ with respect
to $\bF$, Theorem~\ref{fully-faithful-thm} claims the forgetful functor
$\R\contra\rarrow R\modl$ is fully faithful.
 This functor is also always exact (and it preserves infinite products).
 It follows that the essential image of the forgetful functor
$\R\contra\rarrow R\modl$ is an exactly embedded abelian full
subcategory in $R\modl$.

 In other words, the full subcategory of left $\R$\+contramodules in
the abelian category of left $R$\+modules is closed under the kernels
and cokernels of morphisms (and also under infinite products).
 Generally speaking, it is not closed under extensions, though
(see~\cite[Examples~5.2\,(6\+-8)]{Pper}).
 The aim of this section is to develop a simple technique for
producing a sequence of bigger kernel- and cokernel-closed full
subcategories in a fixed abelian category out of a smaller one,
by adding extensions.

 Let $\sA$ be an abelian category and $\sC\subset\sA$ be a full
subcategory containing the zero object.
 Denote by $\sC_\sA^{(2)}\subset\sA$ the full subcategory formed by all
the objects $X\in\sA$ for which there exists a short exact sequence
$0\rarrow K\rarrow X\rarrow C\rarrow0$ in $\sA$ with the objects $K$
and $C$ belonging to~$\sC$.
 Obviously, one has $\sC\subset\sC^{(2)}_\sA\subset\sA$.

 More generally, for every integer $m\ge0$ let us denote by
$\sC_\sA^{(m)}\subset\sA$ the full subcategory consisting of all
the objects $X\in\sA$ admitting a filtration by subobjects
$0=X_0\subset X_1\subset\dotsb\subset X_{m-1}\subset X_m=X$ such
that $X_j/X_{j-1}\in\sC$ for all $1\le j\le m$.
 One has $\sC_\sA^{(0)}=\{0\}$, \ $\sC_\sA^{(1)}=\sC$, \
$\sC_\sA^{(m)}\subset\sC_\sA^{(m+1)}$, and
$\bigl(\sC_\sA^{(m)}\bigr)_\sA^{(n)}=\sC_\sA^{(mn)}$
for all $m$, $n\ge0$.

\begin{lem} \label{extension-ker-coker-lemma}
 Let\/ $\sC$ be a full subcategory closed under kernels and cokernels
in an abelian category\/ $\sA$, and let $m\ge0$ be an integer.
 Then the cokernel of any morphism in\/ $\sA$ from an object of\/
$\sC^{(m)}_\sA$ to an object of\/ $\sC$ belongs to\/~$\sC$.
 The kernel of any morphism in\/ $\sA$ from an object of\/
$\sC$ to an object of\/ $\sC^{(m)}_\sA$ belongs to\/~$\sC$.
\end{lem}

\begin{proof}
 For any short exact sequence $0\rarrow X\rarrow Y\rarrow Z\rarrow 0$
in $\sA$, an object $A$, and a morphism $f\:Y\rarrow A$, one can
compute the cokernel $C=\coker f$ in two steps: consider the cokernel
$D$ of the composition $X\rarrow Y\rarrow A$; then $C$ is
the cokernel of the induced morphism $Z\rarrow D$.
\end{proof}

\begin{prop} \label{extended-subcategory-prop}
 Let\/ $\sA$ be an abelian category and\/ $\sC\subset\sA$ be a full
subcategory closed under finite direct sums of objects and the kernels
and cokernels of morphisms.
 Let $m\ge0$ be an integer.
 Then the full subcategory\/ $\sC^{(m)}_\sA\subset\sA$ is also closed
under finite direct sums, kernels, and cokernels.
 So, from any two abelian categories\/ $\sC$ and\/ $\sA$ with a fully
faithful exact functor\/ $\sC\rarrow\sA$ one can produce a sequence
of abelian categories\/ $\sC^{(m)}_\sA$ with fully faithful exact
functors\/ $\sC^{(m)}_\sA\rarrow\sA$ in this way.
\end{prop}

\begin{proof}
 We proceed by induction on $m\ge1$ (the cases $m=0$ and~$1$ being
obvious).
 Since the class of all short exact sequences in $\sA$ is closed under
finite direct sums, and so is the full subcategory $\sC\subset\sA$, it
follows that the full subcategory $\sC^{(m)}_\sA\subset\sA$ is also
closed under finite direct sums.

 To check that the full subcategory $\sC^{(m)}_\sA\subset\sA$ is closed
under kernels and cokernels, we consider a morphism $f\:X\rarrow Y$
between two objects from~$\sC^{(m)}_\sA$, \ $m\ge2$.
 By construction, there exist two short exact sequences $0\rarrow K
\rarrow X\rarrow C\rarrow0$ and $0\rarrow L\rarrow Y\rarrow D\rarrow0$
in $\sA$ with $K$, $L\in\sC^{(m-1)}_\sA$ and $C$, $D\in\sC$.
$$
\begin{diagram}
\node{0}\arrow{e}\node{K}\arrow{e}\node{X}\arrow{e}\arrow{s,r}{f}
\node{C}\arrow{e}\node{0}\\
\node{0}\arrow{e}\node{L}\arrow{e}\node{Y}\arrow{e}\node{D}
\arrow{e}\node{0}
\end{diagram}
$$

 Denote by $g\:K\rarrow D$ the composition of morphisms
$K\rarrow X\rarrow Y\rarrow D$, and denote by $L'$ the sum of
the subobject $L$ in $Y$ with the image of the composition $K\rarrow X
\rarrow Y$.
 Then there is a short exact sequence
$$
 0\lrarrow\ker g\lrarrow L\oplus K\lrarrow L'\lrarrow0
$$
in~$\sA$.
 By the induction assumption, we have $\ker g\in\sC^{(m-1)}_\sA$ 
(since $D\in\sC\subset\sC_\sA^{(m-1)}$) and
$L\oplus K\in\sC^{(m-1)}_\sA$, hence $L'\in\sC^{(m-1)}_\sA$.

 Set $D'=\coker g\in\sA$; by Lemma~\ref{extension-ker-coker-lemma},
we have $D'\in\sC$.
 Then there is a short exact sequence $0\rarrow L'\rarrow Y\rarrow D'
\rarrow 0$ in $\sA$ and a morphism of short exact sequences
$$
\begin{diagram}
\node{0}\arrow{e}\node{K}\arrow{e}\arrow{s,r}{k}\node{X}\arrow{e}
\arrow{s,r}{f}\node{C}\arrow{e}\arrow{s,r}{c}\node{0}\\
\node{0}\arrow{e}\node{L'}\arrow{e}\node{Y}\arrow{e}\node{D'}
\arrow{e}\node{0}
\end{diagram}
$$

 Now the Snake lemma provides a six-term exact sequence
$$
 0\rarrow\ker k\rarrow\ker f\rarrow\ker c\rarrow\coker k
 \rarrow\coker f\rarrow\coker c\rarrow0,
$$
where the kernel and cokernel of the morphism~$k$ belong to
$\sC^{(m-1)}_\sA$, while the kernel and cokernel of the morphism~$c$
belong to~$\sC$.
 It follows that for the boundary morphism
$\partial\:\ker c\rarrow\coker k$ we have $\ker\partial\in\sC$ and
$\coker\partial\in\sC^{(m-1)}_\sA$.

 Finally, from the short exact sequences $0\rarrow\ker k\rarrow
\ker f\rarrow\ker\partial\rarrow0$ and $0\rarrow\coker\partial
\rarrow\coker f\rarrow\coker c\rarrow0$ we conclude that the objects
$\ker f$ and $\coker f$ belong to $\sC^{(m)}_\sA$, as desired.
\end{proof}

\begin{cor}
 In the assumptions of Proposition~\ref{extended-subcategory-prop},
for any two objects $X\in\sC^{(m)}_\sA$ and $Y\in\sC^{(n)}_\sA$, where
$m$, $n\ge0$, and any morphism $f\:X\rarrow Y$ in\/ $\sA$, one has\/
$\ker f\in\sC^{(m)}_\sA$ and\/ $\coker f\in\sC^{(n)}_\sA$.
\end{cor}

\begin{proof}
 By Proposition~\ref{extended-subcategory-prop}, the assertion of
Lemma~\ref{extension-ker-coker-lemma} is applicable to the full
subcategories ${}'\mskip-.5\thinmuskip\sC=\sC^{(m)}_\sA\subset\sA$ and
${}''\mskip-.5\thinmuskip\sC=\sC^{(n)}_\sA\subset\sA$.
 So it remains to observe that both $\sC^{(m)}_\sA$ and $\sC^{(n)}_\sA$
are contained in $'\mskip-.5\thinmuskip\sC^{(n)}_\sA=\sC^{(mn)}_\sA=
{}''\mskip-.5\thinmuskip\sC^{(m)}_\sA$ (assuming $m$, $n\ge1$).
\end{proof}

 Obviously, if the infinite product functors are (everywhere defined
and) exact in $\sA$ and the full subcategory $\sC\subset\sA$ is
closed under infinite products, then the full subcategories
$\sC^{(m)}_\sA\subset\sA$ are closed under infinite products, too.

\Section{Projective Dimension at Most~$1$}  \label{projdim-1-secn}

 In this section we prove the result announced in the abstract,
namely, that any left flat ring epimorphism $u\:R\rarrow U$ of
countable type has left projective dimension at most~$1$.
 We also describe the Geigle--Lenzing perpendicular subcategory
$U^{\perp_{0,1}}\subset R\modl$ of the left $R$\+module $U$ in terms
of $\R$\+contramodules.

 We start in the somewhat greater generality of a right linear
topology $\bF$ on an associative ring~$R$.
 Let $\R=\varprojlim_{I\in\bF}R/I$ be the completion of $R$ with
respect to $\bF$, viewed as a complete, separated topological ring
in the projective limit topology~$\bfF$.

\begin{lem} \label{ext-from-discrete}
 Let $A$ be an associative ring, $\N$ be an $A$\+$R$\+bimodule
that is discrete as a right $R$\+module, and $V$ be a left $A$\+module.
 Then for every $n\ge0$ the abelian group\/ $\Ext^n_A(\N,V)$ carries
a natural left\/ $\R$\+contramodule structure whose underlying left
$R$\+module structure is induced by the right $R$\+module structure
on\/~$\N$.
\end{lem}

\begin{proof}
 Let us emphasize that the abelian groups $\Ext^n_A(N,V)$ have natural
left $R$\+module structures for any $A$\+$R$\+bimodule~$N$.
 It is claimed that this $R$\+module structure on $\Ext^n_A(\N,V)$
underlies a naturally defined left $\R$\+contramodule structure whenever
the $A$\+$R$\+bimodule $\N$ is discrete as a right $R$\+module.

 Indeed, according to Section~\ref{prelim-contratensor}, there is
a natural left $\R$\+contramodule structure on the abelian group
$\Hom_A(\N,V)$.
 Now let $0\rarrow V\rarrow E^0\rarrow E^1\rarrow E^2\rarrow\dotsb$
be an injective resolution of the left $A$\+module~$V$.
 Then $\Hom_A(\N,E^\bu)$ is naturally a complex of left
$\R$\+contramodules, hence its cohomology modules
$\Ext^n_A(\N,V)=H^n\Hom_A(\N,E^\bu)$ are left $\R$\+contramodules, too.
\end{proof}

 Given a complex of left $A$\+modules $M^\bu$ and a left $A$\+module
$B$, we will use the simplified notation $\Ext^n_A(M^\bu,B)$ for
the abelian groups of morphisms in the derived category of left
$A$\+modules
$$
 \Ext^n_A(M^\bu,B)=\Hom_{\sD(A\modl)}(M^\bu,B[n]).
$$
 When $M^\bu$ is a complex of $A$\+$R$\+bimodules, the right action of
$R$ in the derived category object $M^\bu\in\sD(A\modl)$ induces
left $R$\+module structures on the groups $\Ext^n_A(M^\bu,B)$.

 Assume that the complex $M^\bu$ is concentrated in the nonpositive
cohomological degrees, that is $M^i=0$ for $i>0$.
 Then $\Ext^n_A(M^\bu,B)=0$ for $n<0$.
 Moreover, denoting by $L$ the cokernel of the map $M^{-1}\rarrow M^0$,
one has $\Ext^0_A(M^\bu,B)=\Hom_A(L,B)$.
 However, one may well have $\Ext^n_A(M^\bu,E)\ne0$ for an injective
left $A$\+module $E$ and some $n>0$, if $H^i(M^\bu)\ne0$ for some $i<0$.
 (In fact, for $E$ injective, one has $\Ext^n_A(M^\bu,E)=
\Hom_A(H^{-n}(M^\bu),E)$ for all $n\in\boZ$, so $\Ext^n_A(M^\bu,E)\ne0$
whenever $E$ is an injective cogenerator of $A\modl$ and
$H^{-n}(M^\bu)\ne0$.)

 Now let us assume that the forgetful functor $\R\contra\rarrow
R\modl$ is fully faithful, so the abelian category $\sC=\R\contra$
can be viewed as a full subcategory in the abelian category
$\sA=R\modl$.
 Then the construction of Section~\ref{extending-subcategories-secn}
provides a sequence of abelian full subcategories $\sC=\sC^{(1)}_\sA
\subset\sC^{(2)}_\sA\subset\sC^{(3)}_\sA\subset\sC^{(4)}_\sA\subset
\sC^{(5)}_\sA\subset\dotsb\subset\sA$, with exact embedding functors
$\sC^{(m)}_\sA\rarrow\sA$, indexed by the integers $m\ge0$.

\begin{lem} \label{ext-from-two-term-complex-lemma}
 Let $M^\bu=(M^{-1}\to M^0)$ be a two-term complex of
$A$\+$R$\+bimodules whose cohomology bimodules $H^{-1}(M^\bu)=
\ker(M^{-1}\to M^0)$ and $H^0(M^\bu)=\coker(M^{-1}\to M^0)$ are discrete
as right $R$\+modules.
 Assume that the forgetful functor $\R\contra\rarrow R\modl$ is fully
faithful.
 Let $B$ be a left $A$\+module.
 Then for every $n\ge0$ the left $R$\+module\/ $\Ext^n_A(M^\bu,B)$
belongs to the full subcategory
$$
 \R\contra^{(2)}_{R\modl}=\sC^{(2)}_\sA\,\subset\,\sA=R\modl.
$$
\end{lem}

\begin{proof}
 Applying the contravariant cohomological functor
$\Hom_{\sD(A\modl)}({-},B[*])$ to the distinguished triangle
$$
 H^{-1}(M^\bu)[1]\lrarrow M^\bu\lrarrow H^0(M^\bu)\lrarrow
 H^{-1}(M^\bu)[2]
$$
in the derived category of $A$\+$R$\+bimodules $\sD(A\bimod R)$,
we obtain a long exact sequence of left $R$\+modules
\begin{multline*}
 \dotsb\lrarrow\Ext_A^{n-2}(H^{-1}(M^\bu),B)\lrarrow
 \Ext_A^n(H^0(M^\bu),B) \lrarrow\Ext_A^n(M^\bu,B) \\ \lrarrow
 \Ext_A^{n-1}(H^{-1}(M^\bu),B)\lrarrow\Ext_A^{n+1}(H^0(M^\bu),B)
 \lrarrow\dotsb
\end{multline*}
 By Lemma~\ref{ext-from-discrete}, the left $R$\+modules
$\Ext^n_A(H^i(M^\bu),B)$ are the underlying left $R$\+modules of
certain left $\R$\+contramodules.
 Since the forgetful functor $\R\contra\rarrow R\modl$ is fully
faithful, the left $R$\+module morphisms
$\Ext_A^{n-2}(H^{-1}(M^\bu),B)\rarrow\Ext_A^n(H^0(M^\bu),B)$
are, in fact, left $\R$\+contramodule morphisms.
 So the kernels and cokernels of these morphisms belong to the full
subcategory $\R\contra\subset R\modl$, and it follows that the left
$R$\+modules $\Ext^n_A(M^\bu,B)$ belong to the full subcategory
$\R\contra^{(2)}_{R\modl}\subset R\modl$.
\end{proof}

 Let $\bG$ be a right Gabriel topology on an associative ring~$R$,
and let $U=R_\bG$ be the ring of quotients of $R$ with respect to~$\bG$.
 The ring $U$ or, which is easier, the $R$\+$R$\+bimodule $U$ can be
obtained from $R$ by applying the sheafification functor
$N\longmapsto N_{(\bG)}$ twice: one has $R_\bG=R_{(\bG)(\bG)}$
(see Section~\ref{prelim-sheafification}).
 Alternatively, one can first pass to the quotient ring $R/t_\bG(R)$
of the ring $R$ by its maximal $\bG$\+torsion right $R$\+submodule
(which is a two-sided ideal), and then set $R_\bG=(R/t_\bG(R))_{(\bG)}$.

 There is a natural morphism of $R$\+$R$\+bimodules (in fact, even of
associative rings) $R\rarrow R_{\bG}=U$.
 It is important for us that both the kernel and the cokernel of
the map $R\rarrow U$ are $\bG$\+torsion right $R$\+modules.

 Consider the two-term complex of $R$\+$R$\+bimodules
$$
 K^\bu_{R,U}=(R\to U)
$$
with the term $R$ sitting in the cohomological degree~$-1$ and
the term $U$ sitting in the cohomological degree~$0$.
 There is a natural distinguished triangle in $\sD(R\bimod R)$
\begin{equation} \label{main-triangle}
 R\lrarrow U\lrarrow K^\bu_{R,U}\lrarrow R[1].
\end{equation}

 Set $A=R$, and suppose that we are given a left $R$\+module~$B$.
 Applying the contravariant cohomological functor
$\Hom_{\sD(R\modl)}({-},B)$ to the distinguished
triangle~\eqref{main-triangle}, we obtain a natural five-term
exact sequence of left $R$\+modules
\begin{multline} \label{main-five-term-sequence}
 0\lrarrow\Ext_R^0(K^\bu_{R,U},B)\lrarrow\Hom_R(U,B)\lrarrow
 B \\ \lrarrow\Ext_R^1(K^\bu_{R,U},B)\lrarrow\Ext_R^1(U,B)\lrarrow0,
\end{multline}
as well as natural isomorphisms of left $R$\+modules
\begin{equation} \label{higher-ext-isomorphism}
 \Ext_R^n(K^\bu_{R,U},B)\simeq\Ext_R^n(U,B), \qquad n\ge2.
\end{equation}
 Notice that the right $R$\+module structure on the ring $U$ is
obtained by restricting scalars from a (free) right $U$\+module
structure.
 So, for every $n\ge0$, the left $R$\+module structure on
$\Ext_R^n(U,B)$ can be obtained by restricting scalars from
a natural left $U$\+module structure.

 Let $\R=\varprojlim_{I\in\bG}R/I$ be the completion of the topological
ring $R$ with respect to its topology $\bG$, viewed as a complete,
separated topological ring in the projective limit topology~$\bfG$.
 Let us assume that a right Gabriel topology $\bG$ on $R$ has
a countable base consisting of finitely generated right ideals.
 Then, by Corollary~\ref{countable-gabriel-fully-faithful-cor},
the forgetful functor $\R\contra\rarrow R\modl$ is fully faithful.

\begin{cor} \label{exts-are-extensions-of-contramodules-cor}
 Let\/ $\bG$ be a right Gabriel topology on an associative ring
$R$ having a countable base consisting of finitely generated
right ideals.
 Then, in the above notation, one has
$$
 \Ext_R^n(K^\bu_{R,U},B)\in\R\contra^{(2)}_{R\modl}
$$
for all left $R$\+modules $B$ and all integers $n\ge0$.
\end{cor}

\begin{proof}
 This is a particular case of
Lemma~\ref{ext-from-two-term-complex-lemma}.
\end{proof}

 At this point we restrict the generality level even further in order
to pass to the situation we are interested in.
 Let $u\:R\rarrow U$ be an epimorphism of associative rings such that
$U$ is a flat left $R$\+module, and let $\bG$ be the related perfect
Gabriel topology on $R$, consisting of all the open right ideals
$I\subset R$ such that $R/I\ot_RU=0$ (see the discussion in
the beginning of Section~\ref{perfect-of-type-lambda-secn}).

 For any left $R$\+module $M$, we denote by $\pd{}_RM$ the projective
dimension of $M$ (as an object of $R\modl$).
 In particular, $\pd{}_RU$ denotes the projective dimension of the left
$R$\+module~$U$.
 As usual, we denote by $\Hom_R$ and $\Ext_R^*$ the Hom and Ext groups
computed in the abelian category of left $R$\+modules $R\modl$.

 Given a left $R$\+module $M$, one can consider the full subcategory
$M^{\perp_{0,1}}\subset R\modl$ of all left $R$\+modules $C$ satisfying
$\Hom_R(M,C)=0=\Ext^1_R(M,C)$.
 According to~\cite[Proposition~1.1]{GL}, the full subcategory
$M^{\perp_{0,1}}$, which we call the \emph{Geigle--Lenzing perpendicular
subcategory}, is closed under kernels and extensions in $R\modl$
(it is also closed under infinite products).
 When $\pd{}_RM\le1$, this full subcategory is closed under
cokernels as well; so the category $M^{\perp_{0,1}}$ is abelian in
this case, and its fully faithful identity inclusion functor
$M^{\perp_{0,1}}\rarrow R\modl$ is exact.

\begin{prop} \label{orthogonality-prop}
 Assume that the perfect Gabriel topology\/ $\bG$ on the ring $R$
associated with a left flat ring epimorphism $u\:R\rarrow U$ has
a countable base or, which is equivalent, the topology\/ $\bfG$ on
the ring\/ $\R$ has a countable base.
 Then \par
\textup{(a)} for any separated left\/ $\R$\+contramodule\/ $\C$,
one has\/ $\Hom_R(U,\C)=0=\Ext^1_R(U,\C)$; \par
\textup{(b)} for any left\/ $\R$\+contramodule\/ $\C$, one has\/
$\Hom_R(U,\C)=0$; \par
\textup{(c)} assuming that\/ $\pd{}_RU\le1$, for any left\/
$\R$\+contramodule\/ $\C$ one has\/ $\Hom_R(U,\C)\allowbreak
=0=\Ext^1_R(U,\C)$.
\end{prop}

\begin{proof}
 By Corollary~\ref{countable-gabriel-fully-faithful-cor}, the forgetful
functor $\R\contra\rarrow R\modl$ is fully faithful in our present
assumptions, but we do not need to use this fact now.
 Moreover, in the next Theorem~\ref{projdim-1-theorem} we will see that
the projective dimension of the left $R$\+module $U$ never exceeds~$1$
in the assumptions of this proposition, so the conclusion of part~(c)
always holds (but we do not know this yet).

 Part~(a): for any discrete right $R$\+module $\N$, we have
$\Hom_\boZ(\N,\boQ/\boZ)\in U^{\perp_{0,1}}$.
 In fact,
$$
 \Hom_R(U,\Hom_\boZ(\N,\boQ/\boZ))=\Hom_\boZ(\N\ot_RU,\>\boQ/\boZ)=0
$$
because $\N\ot_RU=0$ (as $U$ is a $\bG$\+divisible left $R$\+module)
and
$$
 \Ext_R^i(U,\Hom_\boZ(\N,\boQ/\boZ))=
 \Hom_\boZ(\Tor^R_i(\N,U),\,\boQ/\boZ)
 \qquad\text{for all \,$i>0$,}
$$
since $U$ is a flat left $R$\+module.

 By Proposition~\ref{separated-kernel-of-homs}, any separated left
$\R$\+contramodule $\C$ is the kernel of a morphism
$\Hom_\boZ(\N',\boQ/\boZ)\rarrow\Hom_\boZ(\N'',\boQ/\boZ)$ in
$\R\contra$ (hence in $R\modl$) for some $\N'$, $\N''\in\discr\R$.
 By~\cite[Proposition~1.1]{GL}, it follows that
$\C\in U^{\perp_{0,1}}\subset R\modl$.

 Parts~(b) and~(c): by Lemma~\ref{quotient-of-separated-lemma}, there
is a short exact sequence of left $\R$\+contra\-modules $0\rarrow\E
\rarrow\D\rarrow\C\rarrow0$ in which the $\R$\+contramodules
$\E$ and $\D$ are separated.
 Since $\Hom_R(U,\D)=0$ and $\Ext_R^1(U,\E)=0$ by part~(a), it follows
that $\Hom_R(U,\C)=0$.
 Assuming that $\pd{}_RU=1$, the map $\Ext_R^1(U,\D)\rarrow
\Ext_R^1(U,\C)$ is surjective, so $\Ext_R^1(U,\D)=0$ implies
$\Ext_R^1(U,\C)=0$.
\end{proof}

 Finally, we can prove the main three theorems of this section (which
are also among the main results of this paper).

\begin{thm} \label{projdim-1-theorem}
 Let $u\:R\rarrow U$ be an epimorphism of associative rings such that
$U$ is a flat left $R$\+module.
 Let\/ $\bG$ be the related perfect Gabriel topology of right ideals
in~$R$.
 Assume that\/ $\bG$ has a countable base.
 Then the projective dimension of the left $R$\+module $U$ does not
exceed\/~$1$.
\end{thm}

\begin{proof}
 Let $U^{\perp_0}\subset R\modl$ denote the full subcategory in
the category of left $R$\+modules consisting of all the left
$R$\+modules $M$ such that $\Hom_R(U,M)=0$.
 By Proposition~\ref{orthogonality-prop}(b), the full subcategory
of left $\R$\+contramodules $\R\contra\subset R\modl$ is
contained in $U^{\perp_0}$.
 Hence it follows that $\R\contra^{(2)}_{R\modl}\subset U^{\perp_0}$.

 Applying Corollary~\ref{exts-are-extensions-of-contramodules-cor},
we see that $\Ext^n_R(K^\bu_{R,U},B)\in U^{\perp_0}$ for all
left $R$\+modules $B$ and all $n\ge0$.
 On the other hand, by~\eqref{higher-ext-isomorphism} we have
$\Ext^n_R(U,B)\simeq\Ext^n_R(K^\bu_{R,U},B)$ for $n\ge2$.
 Thus the left $R$\+module $\Ext^n_R(U,B)$ belongs to $U^{\perp_0}$
when $n\ge2$.

 Since the left $R$\+module structure on $D=\Ext^n_R(U,B)$ underlies
a left $U$\+module structure, every element $d\in D$
is the image of the unit $1\in U$ under a certain left $R$\+module
morphism (actually, a unique left $U$\+module morphism)
$U\rarrow D$.
 So $\Hom_R(U,D)=0$ implies $D=0$, and we can conclude that
$\Ext^n_R(U,B)=0$ for all left $R$\+modules $B$ and all $n\ge2$.
\end{proof}

\begin{rem} \label{left-countably-generated-remark}
 There is a much easier alternative proof of
Theorem~\ref{projdim-1-theorem} applicable in the case of
a left $\omega^+$\+Noetherian ring $R$ (i.~e., when every left ideal
in $R$ has a countable set of generators).
 The argument is based on the sheafification construction of
Section~\ref{prelim-sheafification}.
 Indeed, if $R$ is left $\omega^+$\+Noetherian, then
any submodule of a countably generated left $R$\+module is
countably generated, and any countably generated left
$R$\+module is countably presented.
 In this case, for any $R$\+$R$\+bimodule $N$ that is
countably generated as a left $R$\+module and every finitely generated
right ideal $I\subset R$, the left $R$\+module $\Hom_{R^\rop}(I,N)$
is countably generated (being a submodule of a finite direct sum of
copies of the left $R$\+module~$N$).
 Since the class of countably generated left $R$\+modules is closed
under countable direct limits, for any right Gabriel topology $\bG$
on $R$ with a countable base consisting of finitely generated right
ideals the $R$\+$R$\+bimodule $N_{(\bG)}$ is countably generated
as a left $R$\+module.
 Hence the $R$\+$R$\+bimodule $N_\bG$ is countably generated
as a left $R$\+module, too.
 In the context of Theorem~\ref{projdim-1-theorem}, we can then
conclude that the left $R$\+module $U$ is countably generated.
 Hence, in our assumptions, it is countably presented.
 By~\cite[Corollary~2.23]{GT}, any countably presented flat module
has projective dimension at most~$1$.
\end{rem}

\begin{cor} \label{direct-limit-of-projdim1-cor}
 Let $u\:R\rarrow U$ be an epimorphism of associative rings such that
$U$ is a flat left $R$\+module.
 Assume that the related perfect Gabriel topology\/ $\bG$ of right
ideals in $R$ satisfies the condition~(T$_\omega$) of
Section~\ref{perfect-of-type-lambda-secn} (e.~g., this holds if $R$ is
commutative; see Examples~\ref{T-lambda-examples} for further cases
when (T$_\omega$)~is satisfied).
 Assume further that the right $R$\+module $R$ has\/
$\omega$\+bounded\/ $\bG$\+torsion.
 Then there exists a diagram of epimorphisms of associative rings
$R\rarrow U_\upsilon$, indexed by an\/ $\omega^+$\+directed poset\/
$\Upsilon$, such that $U_\upsilon$ is a flat left $R$\+module of
projective dimension not exceeding\/~$1$ for every $\upsilon\in\Upsilon$
and the ring homomorphism $R\rarrow U$ is the direct limit of
the ring homomorphisms $R\rarrow U_\upsilon$, that is
$U=\varinjlim_{\upsilon\in\Upsilon}U_\upsilon$.
\end{cor}

\begin{proof}
 Follows from Corollary~\ref{perfect-gabriel-lambda-directed-cor}
and Theorem~\ref{projdim-1-theorem}.
\end{proof}

\begin{rem}
 The following converse assertion to
Corollary~\ref{direct-limit-of-projdim1-cor} holds for commutative
rings~$R$: if $R\rarrow U$ is a ring epimorphism such that
$\Tor_1^R(U,U)=0$ and $\pd{}_RU\le1$, then $U$ is a flat $R$\+module.
 This is~\cite[Remark~16.9]{BP}.
\end{rem}

 The next theorem is a generalization of~\cite[Examples~5.4(2)
and~5.5(2)]{Pper}.

\begin{thm} \label{u-contra=R-contra-2}
 Let $u\:R\rarrow U$ be an epimorphism of associative rings such that
$U$ is a flat left $R$\+module.
 Let\/ $\bG$ be the related perfect Gabriel topology of right ideals
in~$R$.
 Assume that\/ $\bG$ has a countable base.
 Then the Geigle--Lenzing perpendicular subcategory $U^{\perp_{0,1}}
\subset R\modl$ coincides with the full subcategory of two-object
extensions\/ $\R\contra^{(2)}_{R\modl}\subset R\modl$.
 In particular, it follows that the full subcategory\/
$\R\contra^{(2)}_{R\modl}$ is closed under extensions in $R\modl$.
 So one has\/ $\R\contra^{(2)}_{R\modl}=\R\contra^{(m)}_{R\modl}$
for all the integers $m\ge2$.
\end{thm}

\begin{proof}
 Let us emphasize that the full subcategory $U^{\perp_{0,1}}\subset
R\modl$ is abelian and its embedding functor $U^{\perp_{0,1}}\rarrow
R\modl$ is exact by Theorem~\ref{projdim-1-theorem}
and~\cite[Proposition~1.1]{GL}; while the full subcategory
$\R\contra^{(2)}_{R\modl}\subset R\modl$ is abelian and its embedding
functor $\R\contra^{(2)}_{R\modl}\rarrow R\modl$ is exact by
Proposition~\ref{extended-subcategory-prop} (applied to the exact
forgetful functor $\R\contra\rarrow R\modl$, which is fully faithful
by Corollary~\ref{countable-gabriel-fully-faithful-cor}).

 Furthermore, the Geigle--Lenzing perpendicular subcategory
$U^{\perp_{0,1}}$ is obviously closed under extensions in $R\modl$.
 By Proposition~\ref{orthogonality-prop}(c), we know that
$\R\contra\subset U^{\perp_{0,1}}\subset R\modl$; hence
$\R\contra^{(2)}_{R\modl}\subset U^{\perp_{0,1}}$.
 It remains to prove that the inclusion in the opposite direction
holds as well.

 Let $B$ be a left $R$\+module belonging to $U^{\perp_{0,1}}$; so
$\Hom_R(U,B)=0=\Ext^1_R(U,B)$.
 Then from the exact sequence~\eqref{main-five-term-sequence} we
see that the natural left $R$\+module morphism $B\rarrow
\Ext^1_R(K^\bu_{R,U},B)$ is an isomorphism.
 Applying Corollary~\ref{exts-are-extensions-of-contramodules-cor},
we can conclude that $B\in\R\contra^{(2)}_{R\modl}$.
\end{proof}

 A second, alternative proof of the following theorem will be given in
Section~\ref{faithful-perfect-secn}; and an uncountable generalization
will be obtained in Theorem~\ref{faithful-topology-theorem}.

\begin{thm} \label{injective-epi-u-contra=R-contra}
 Let $u\:R\rarrow U$ be an injective epimorphism of associative rings
such that $U$ is a flat left $R$\+module.
 Let\/ $\bG$ be the related faithful perfect Gabriel topology of
right ideals in~$R$.
 Assume that\/ $\bG$ has a countable base.
 Then the Geigle--Lenzing perpendicular subcategory $U^{\perp_{0,1}}
\subset R\modl$ coincides with the full subcategory of\/ left
$\R$\+contramodules\/ $\R\contra\subset R\modl$.
 In particular, it follows that the full subcategory\/
$\R\contra$ is closed under extensions in $R\modl$.
\end{thm}

\begin{proof}[First proof]
 We already know that $\R\contra\subset U^{\perp_{0,1}}$.
 To prove the inverse inclusion, consider a left $R$\+module
$B\in U^{\perp_{0,1}}$.
 As in the previous proof, we have an isomorphism $B\simeq
\Ext_R^1(K^\bu_{R,U},B)$.
 Now the morphism $u\:R\rarrow U$ is injective by assumption, so
the two-term complex $K^\bu_{R,U}$ is quasi-isomorphic to the
quotient bimodule $K_{R,U}=U/R=\coker u$.
 It remains to apply Lemma~\ref{ext-from-discrete} in order to
conclude that $\Ext_R^1(K^\bu_{R,U},B)=\Ext_R^1(K_{R,U},B)\in\R\contra$.
\end{proof}

\Section{$U$-Strongly Flat and $U$-Weakly Cotorsion $R$-Modules}
\label{strongly-flat-weakly-cotorsion-secn}

 Given a left flat ring epimorphism $u\:R\rarrow U$, a left $R$\+module
$C$ is said to be \emph{$U$\+weakly cotorsion} if\/
$\Ext^1_R(U,C)=0$.
 A left $R$\+module $F$ is said to be \emph{$U$\+strongly flat} if
$\Ext^1_R(F,C)=0$ for all $U$\+weakly cotorsion left $R$\+modules~$C$.
 Since $U$ is a flat left $R$\+module, it follows
from~\cite[Lemma~3.4.1]{Xu} that all $U$\+strongly flat left
$R$\+modules are flat.
 Moreover, a left $R$\+module $F$ is $U$\+strongly flat if and only if
it is a direct summand of a left $R$\+module $G$ that can be included
into a short exact sequence of left $R$\+modules $0\rarrow V\rarrow G
\rarrow W\rarrow 0$, where $V$ is a free left $R$\+module and
$W$ is a free left $U$\+module (by~\cite[Corollary~6.13]{GT}
and since $\Ext^1_R(U,W)=0$ for any free left $U$\+module~$W$).
 In this section we apply the techniques developed above in this paper
in order to describe $U$\+strongly flat and $U$\+weakly cotorsion
left $R$\+modules with respect to a left flat ring epimorphism
$u\:R\rarrow U$ of countable type.
 
 We start in the greater generality of a right linear topology $\bF$ on
an associative ring~$R$.
 In Proposition~\ref{discrete-modules-contramodules-from-systems},
we have constructed a pair of adjoint functors
$\IL\:\Ab_{\bF}\rightleftarrows \discr R\,\,\:\!\Dh$ between
the category of contravariant $\bF$\+systems of abelian groups and
the category of discrete right $R$\+modules, and a similar pair of
adjoint functors $\PL\:{}_\bF\Ab\leftrightarrows\R\contra\,\,\:\!\CT$
between the category of covariant $\bF$\+systems of abelian groups
and the category of left $\R$\+contramodules.
 Generalizing to pseudo-$\bF$-systems, in
Lemma~\ref{modules-from-pseudo-systems-lem} we have constructed
an inductive limit functor $\il\:\Ab_{\{\bF\}}\rarrow\modr R$
from the category of contravariant pseudo-$\bF$-systems of abelian
groups to the category of right $R$\+modules, and a projective limit
functor $\pl\:{}_{\{\bF\}}\Ab\rarrow R\modl$ from the category of
covariant pseudo-$\bF$-systems of abelian groups to the category
of left $R$\+modules.

 The following proposition describes the right adjoint functor to
the composition $\Ab_\bF\rarrow\Ab_{\{\bF\}}\rarrow\modr R$ of
the fully faithful functor $\Ab_\bF\rarrow\Ab_{\{\bF\}}$ with
the functor~$\il$, and the left adjoint functor to the composition
${}_\bF\Ab\rarrow{}_{\{\bF\}}\Ab\rarrow R\modl$ of the fully faithful
functor ${}_\bF\Ab\rarrow{}_{\{\bF\}}\Ab$ with the functor~$\pl$.

\begin{prop} \label{systems-modules-adjunction}
\textup{(a)} For any right $R$\+module $N$, there is a contravariant\/
$\bF$\+system of abelian groups\/ $\mh(N)\in\Ab_\bF$ assigning to
every cyclic discrete right $R$\+module $R/I\in\sQ_\bF$ the abelian
group $N_I=\Hom_{R^\rop}(R/I,N)$.
 The functor\/ $\mh\:\modr R\rarrow\Ab_\bF$ is right adjoint to
the composition of functors\/ $\Ab_\bF\rarrow\Ab_{\{\bF\}}
\overset{\il}\rarrow\modr R$. \par
\textup{(b)} For any left $R$\+module $C$, there is a covariant\/
$\bF$\+system of abelian groups\/ $\tp(C)\in{}_\bF\Ab$ assigning to
every cyclic discrete right $R$\+module $R/I\in\sQ_\bF$ the abelian
group $C/IC=R/I\ot_RC$.
 The functor\/ $\tp\:R\modl\rarrow{}_\bF\Ab$ is left adjoint to
the composition of functors\/ ${}_\bF\Ab\rarrow{}_{\{\bF\}}\Ab
\overset{\pl}\rarrow R\modl$.
\end{prop}

\begin{proof}
 The abbreviation ``$\mh$'' means ``module Hom'', while ``$\tp$''
stands for ``tensor product''.
 Given a right $R$\+module $N$, the contravariant $\bF$\+system
$\mh(N)\:\sQ_\bF^\sop\rarrow\Ab$ is constructed by restricting
the contravariant functor $\Hom_{R^\rop}({-},N)$ to the full subcategory
$\sQ_\bF\subset\discr R\subset\modr R$.
 Given a left $R$\+module $C$, the covariant $\bF$\+system
$\tp(C)\:\sQ_\bF\rarrow\Ab$ is constructed by restricting
the covariant functor ${-}\ot_RC$ to the same full subcategory of
cyclic discrete right $R$\+modules in $\modr R$.

 Let us explain why the two functors are adjoint in part~(b).
 This is similar to, though different from (and simper than)
the corresponding argument in the proof of
Proposition~\ref{discrete-modules-contramodules-from-systems}(b).
 The isomorphism of Hom groups $\Hom_R(C,\pl(D))\simeq
\Hom_{{}_\bF\Ab}(\tp(C),D)$ holds for any left $R$\+module $C$ and
any covariant $\bF$\+system of abelian groups $D$, because
the datum of a left $R$\+module morphism
$$
 C\lrarrow\varprojlim\nolimits_{I\in\bF}D(R/I)
$$
is equivalent to the datum of an $\bF$\+indexed family of abelian
group homomorphisms
$$
 C/IC\lrarrow D(R/I),
$$
defined for all the open right ideals $I\in\bF$ and satisfying
the compatibility equations for all the morphisms in
the category~$\sQ_\bF$.
\end{proof}

 Using the construction of
Proposition~\ref{systems-modules-adjunction}(b), we now develop
a (similar, but different and simpler) module version of the theory of
contramodule completion from Section~\ref{separated-contramodules-secn}.
 Let $C$ be a left $R$\+module.
 Denote by $\bB_C$ the set of all subgroups $IC\subset C$ of
the underlying abelian group of $C$, where $I\in\bF$ ranges over
the open right ideals in~$R$.
 Then $\bB_C$ is a linear topology on the underlying abelian group
of $C$, as one has $IC\cap JC\supset (I\cap J)C$ for any two
open right ideals $I$ and $J\subset R$.
 As in Section~\ref{separated-contramodules-secn}, we
notice that the open subgroups $IC\subset C$ are \emph{not}
submodules; so this is not a linear topology on a module in
the sense of~\cite[Section~VI.4]{St}.

 The completion $\varprojlim_{I\in\bF}C/IC$ of a left $R$\+module $C$
in the topology with the base $\bB_C$ can be described in
terms of Proposition~\ref{systems-modules-adjunction}(b) as
the left $R$\+module $\pl(\tp(C))$.
 The natural map $\lambda_{\bF,C}\:C\rarrow\varprojlim_{I\in\bF}C/IC$
is the adjunction morphism for the pair of adjoint functors
in part~(b) of the proposition.
 So the abelian group $\varprojlim_{I\in\bF}C/IC$ is, in fact,
a left $R$\+module (even though the abelian groups $C/IC$ have no such
module structures), and the completion map~$\lambda_{\bF,C}$ is
an $R$\+module morphism.

 Moreover, $\tp(C)$ is a covariant $\bF$\+system of abelian groups,
so one can apply the functor $\PL$ and obtain a left $\R$\+contramodule
$\PL(\tp(C))\in\R\contra$.
 In other words, according to 
Proposition~\ref{discrete-modules-contramodules-from-systems}(b),
\,$\pl(\tp(C))$ is the underlying left $R$\+module of the left
$\R$\+contramodule $\PL(\tp(C))$.
 We can conclude that the left $R$\+module structure of the abelian
group $\varprojlim_{I\in\bF}C/IC$ underlies a naturally
defined left $\R$\+contramodule structure.
 We will sometimes use the notation $\Lambda_\bF(C)=\PL(\tp(C))=
\varprojlim_{I\in\bF}C/IC$ for this left $\R$\+contramodule.

 A left $R$\+module $C$ is said to be $\bF$\+separated (respectively,
$\bF$\+complete) if it is a separated (resp., complete) abelian group
in the topology with a base~$\bB_C$.
 In other words, $C$ is called \emph{$\bF$\+separated} if
the map~$\lambda_{\bF,C}$ is injective and \emph{$\bF$\+complete} if
this map is surjective.
 Clearly, a left $R$\+module $C$ is $\bF$\+separated if and only if
$\bigcap_{I\in\bF}IC=0$.

 One easily observes that the left $R$\+module $\pl(D)$ is
$\bF$\+separated for any covariant $\bF$\+system of abelian groups $D$;
but this is not always true for a covariant pseudo-$\bF$-system $D$
(as, indeed, any left $R$\+module can be obtained by applying
the functor $\pl$ to a constant covariant pseudo-$\bF$-system; see
Section~\ref{F-systems-secn}).
 In particular, for any left $R$\+module $C$, the left $R$\+module
$\varprojlim_{I\in\bF}C/IC$ is $\bF$\+separated.

 The construction of the $\bF$\+completion of a left $R$\+module and
the related notions of $\bF$\+separatedness and $\bF$\+completeness of
a module, as defined in the previous several paragraphs, do \emph{not}
necessarily agree with the similar construction and notions for left
$\R$\+contramodules, as defined in
Section~\ref{separated-contramodules-secn}.
 However, they \emph{do} agree in the case of an associative ring
$R$ endowed with a right Gabriel topology $\bG$ \emph{with a countable
base of finitely generated right ideals}.
 Indeed, by Corollary~\ref{countable-gabriel-fully-faithful-cor},
one has $I\C=\I\tim\C$ for any left $\R$\+contramodule $\C$,
any open right ideal $I\subset R$, and the corresponding open
right ideal $\I\subset\R$ in this case.

 So, in the case of a right Gabriel topology $\bG$ with a countable
base of finitely generated right ideals, a left $\R$\+contramodule
is separated (resp., complete) if and only if it is $\bG$\+separated
(resp., $\bG$\+complete) as a left $R$\+module.
 In fact, an $\R$\+contramodule does not need to be separated,
but it is always complete by
Theorem~\ref{separated-vs-complete-contramodules-thm};
hence all left $\R$\+contramodules are $\bG$\+complete as left
$R$\+modules.

 The following lemma collects some assertions which may help
the reader feel more comfortable.

\begin{lem} \label{completions-unconfused-lem}
\textup{(a)} For any ring $R$ with a right linear topology\/ $\bF$
with a countable base and any left $R$\+module $C$, the left\/
$\R$\+contramodule\/ $\C=\varprojlim_{\I\in\bF}C/IC$ is separated and
complete.
 For any open right ideal $I\subset R$ and the related open right
ideal $\I\subset\R$, the completion map $\lambda_{\bF,C}\:C\rarrow\C$
and the projection map $\C\rarrow C/IC$ induce an isomorphism
$$
 C/IC\simeq\C/\I\tim\C.
$$ \par
\textup{(b)} For any ring $R$ with a right Gabriel topology\/ $\bG$
with a countable base of finitely generated ideals and any left
$R$\+module $C$, the left $R$\+module $\C=\varprojlim_{I\in\bG}C/IC$
is\/ $\bG$\+separated and\/ $\bG$\+complete.
 For any open right ideal $I\subset R$, the completion map
$\lambda_{\bF,C}\:C\rarrow\C$ and the projection map $\C\rarrow C/IC$
induce an isomorphism
$$
 C/IC\simeq\C/I\C.
$$ \par
\textup{(c)} For any ring $R$ with a right Gabriel topology\/ $\bG$
with a countable base of finitely generated ideals, a left
$R$\+module $C$ is\/ $\bG$\+separated and\/ $\bG$\+complete if and
only if it is the underlying left $R$\+module of a separated
left\/ $\R$\+contramodule.
\end{lem}

\begin{proof}
 Part~(a): even without the assumption of a countable topology base,
all the left $\R$\+contramodules belonging to the image of the functor
$\PL$ are separated by
Proposition~\ref{contramodules-covariant-systems-prop};
so the left $\R$\+contramodule $\C=\PL(\tp(C))$ is separated.
 On the other hand, in the assumption of a countable topology base,
any left $\R$\+contra\-module is complete by
Theorem~\ref{separated-vs-complete-contramodules-thm}
(this is~\cite[Lemma~6.3(b)]{PR}).

 The last assertion in part~(a) holds, because the covariant
$\bF$\+system of abelian groups $\tp(C)\:R/I\longmapsto C/IC$
is right exact, and for any right exact covariant $\bF$\+system
of abelian groups $D$ the adjunction map $\CT(\PL(D))\rarrow D$
is an isomorphism.
 This is the result of~\cite[Lemma~6.3(a)]{PR};
see Theorem~\ref{separated-vs-complete-contramodules-thm}.

 Part~(b) follows from part~(a) and
Corollary~\ref{countable-gabriel-fully-faithful-cor}.

 In part~(c), if $C$ is $\bG$\+separated and $\bG$\+complete, then
$\lambda_{\bG,C}\:C\rarrow\varprojlim_{I\in\bG}C/IC$ is an isomorphism;
so $C$ acquires the left $\R$\+contramodule structure of
$\PL(\tp(C))$.
 Conversely, any separated left $\R$\+contramodule is
a $\bG$\+separated $\bG$\+complete left $R$\+module, as it was
explained in the paragraph preceding the lemma.

 Notice also that, in the assumptions of parts~(b\+c), there is always
at most one way to extend a given left $R$\+module structure on $C$
to a left $\R$\+contramodule structure (also by
Corollary~\ref{countable-gabriel-fully-faithful-cor}).
\end{proof}

 Now that we are finished with the preparatory work, let us proceed to
describe $U$\+weakly cotorsion and $U$\+strongly flat left $R$\+modules.
 We refer to~\cite[Section~3]{PSl0} and~\cite[Section~2]{PSl} for
the background discussion of \emph{simply right obtainable} modules.

\begin{thm} \label{weakly-cotorsion-obtainable}
 Let $u\:R\rarrow U$ be a ring epimorphism such that $U$ is a flat
left $R$\+module and the related right Gabriel topology\/ $\bG$ on $R$
has a countable base.
 Then a left $R$\+module is $U$\+weakly cotorsion if and only if
it can be obtained from left $R$\+modules belonging to the union of
the following two classes:
\begin{itemize}
\item the underlying left $R$\+modules of left $U$\+modules
\item $\bG$\+separated $\bG$\+complete left $R$\+modules
\end{itemize}
using the operations of the passage to an extension of two $R$\+modules
and to the cokernel of an injective $R$\+module morphism.
\end{thm}

\begin{proof}
 Notice that all the projective left $U$\+modules are $U$\+strongly
flat left $R$\+mod\-ules; but other left $U$\+modules do not have to be
$U$\+strongly flat as left $R$\+modules.
 On the other hand, any left $U$\+module $D$ is a $U$\+weakly
cotorsion left $R$\+module.
 Indeed, one has $\Ext_R^1(U,D)=\Ext_U^1(U,D)=0$, because $U\ot_RU=U$
and $\Tor_1^R(U,U)=0$.
 Furthermore, for any $\bG$\+separated $\bG$\+complete left
$R$\+module $C$ one has $\Hom_R(U,C)=0=\Ext^1_R(U,C)$ by
Lemma~\ref{completions-unconfused-lem}(c) and
Proposition~\ref{orthogonality-prop}(a).

 To prove the if-part, it remains to show that the class of all
$U$\+weakly cotorsion left $R$\+modules is closed under extensions
and the cokernels of injective morphisms.
 The former is obvious from the definition, and the latter means
that the two classes ($U$\+strongly flat left $R$\+modules,
$U$\+weakly cotorsion left $R$\+modules) form a \emph{hereditary
cotorsion pair} in $R\modl$ (see, e.~g., \cite[Section~5.2]{GT}).
 This is so because one has
$\pd{}_RU\le1$ by Theorem~\ref{projdim-1-theorem}.
 In fact, any cotorsion pair generated by a module of projective
dimension at most~$1$ is hereditary.

 Conversely, let $B$ be a $U$\+weakly cotorsion left $R$\+module,
i.~e., $\Ext^1_R(U,B)=0$.
 Then from the five-term exact sequence~\eqref{main-five-term-sequence}
(which reduces to a four-term exact sequence in this case) we see
that the left $R$\+module $B$ can be obtained using the cokernel of
an injective morphism and an extension from three left $R$\+modules,
namely, $\Ext^0_R(K^\bu_{R,U},B)$, $\Hom_R(U,B)$, and
$\Ext^1_R(K^\bu_{R,U},B)$.

 Now $\Hom_R(U,B)$ is a left $U$\+module.
 As to the left $R$\+modules $\Ext^n_R(K^\bu_{R,U},B)$, by
Corollary~\ref{exts-are-extensions-of-contramodules-cor} each of them
is an extension of at most two left $\R$\+contramodules.
 Finally, by Lemma~\ref{quotient-of-separated-lemma}, each left
$\R$\+contramodule is the cokernel of an injective morphism of
separated left $\R$\+contramodules.
 The latter are $\bG$\+separated $\bG$\+complete left $R$\+modules
by Lemma~\ref{completions-unconfused-lem}(c).
\end{proof}

\begin{rem}
 The above suffices to prove the theorem, but in fact one can say more
about the left $R$\+module $\Ext^0_R(K^\bu_{R,U},B)$, which is
$\bG$\+separated and $\bG$\+complete for any left $R$\+module~$B$.
 Indeed, let $L_{R,U}=H^0(K^\bu_{R,U})$ denote the cokernel of
the $R$\+$R$\+bimodule morphism $R\rarrow U$; then one has
$\Ext^0_R(K^\bu_{R,U},B)=\Hom_R(L_{R,U},B)$.
 Now $L_{R,U}=\coker(R\to R_\bG)$ is a $\bG$\+torsion right $R$\+module;
in other words, $L_{R,U}$ is an $R$\+$R$\+bimodule that is
discrete as a right $R$\+module.
 For any ring $A$, left $A$\+module $V$, and $R$\+discrete
$A$\+$R$\+bimodule $\N$, the left $R$\+module $\Hom_A(\N,V)$ is
a separated left $\R$\+contramodule (see
Section~\ref{prelim-contratensor} and the proof of
Lemma~\ref{embedding-into-hom-lemma}).

 Moreover, let $H_{R,U}=H^{-1}(K^\bu_{R,U})$ be the kernel of
the $R$\+$R$\+bimodule morphism $R\rarrow U$; then $H_{R,U}$ is also
a $\bG$\+torsion right $R$\+module.
 Looking into the proof of Lemma~\ref{ext-from-two-term-complex-lemma}
(on which Corollary~\ref{exts-are-extensions-of-contramodules-cor} is 
based; cf.\ Remark~\ref{ext-1-from-K-remark} below), we see that
the left $R$\+module $\Ext^1_R(K^\bu_{R,U},B)$ is an extension of
the left $R$\+module $\Ext^1_R(L_{R,U},B)$ and the kernel of the left
$R$\+module morphism $\Hom_R(H_{R,U},B)\rarrow\Ext^2_R(L_{R,U},B)$.
 The former one is a left $\R$\+contramodule; while the latter one
is a subcontramodule of a separated left $\R$\+contramodule
$\Hom_R(H_{R,U},B)$, hence a separated left
$\R$\+contramodule itself.
 So the left $R$\+module $\Ext^1_R(K^\bu_{R,U},B)$ is an extension
of one left $\R$\+contramodule and one separated left
$\R$\+contramodule.

 Summarizing the arguments above, one observes that any $U$\+weakly
cotorsion left $R$\+module $B$ can be obtained from one underlying
left $R$\+module of a left $U$\+module and four $\bG$\+separated
$\bG$\+complete left $R$\+modules using two passages to the cokernel
of an injective morphism and two passages to an extension of
$R$\+modules.
 Thus, in total, one needs to apply our operations four times.
\end{rem}

\begin{cor} \label{strongly-flat-characterized}
 Let $u\:R\rarrow U$ be a ring epimorphism such that $U$ is a flat
left $R$\+module and the related right Gabriel topology\/ $\bG$ on $R$
has a countable base.
 Then a left $R$\+module $F$ is $U$\+strongly flat if and only
if it satisfies the following two conditions:
\begin{enumerate}
\renewcommand{\theenumi}{\roman{enumi}}
\item $\Ext^1_R(F,D)=0$ for all left $U$\+modules $D$; and
\item $\Ext^1_R(F,C)=0=\Ext^2_R(F,C)$ for all\/ $\bG$\+separated\/
$\bG$\+complete left $R$\+modules~$C$.  \hbadness=1750
\end{enumerate}
\end{cor}

\begin{proof}
 The if-assertion of Theorem~\ref{weakly-cotorsion-obtainable}
implies the necessity of the condition~(i) and of the $\Ext^1$\+part of
the condition~(ii).
 Furthermore, since $\pd{}_RU\le1$, it follows easily from the Eklof
lemma and (the proof of) the Eklof--Trlifaj theorem~\cite[Lemma~1
and Theorem~10]{ET} (cf.~\cite[Corollary~6.14]{GT}) that $\pd{}_RF\le1$
for every $U$\+strongly flat left $R$\+module~$F$.
 Thus the $\Ext^2$\+part of the condition~(ii) is also necessary, and
in fact one has $\Ext^n_R(F,D)=0=\Ext^n_R(F,C)$ for all
$U$\+strongly flat left $R$\+modules $F$, left $R$\+modules $D$
and $C$ as in~(i) and~(ii), and all $n\ge1$.

 Conversely, for any given left $R$\+module $F$, the class of all
left $R$\+modules $B$ satisfying $\Ext_R^n(F,B)=0$ for $n\ge1$ is
closed under extensions and cokernels of injective morphisms.
 Hence, whenever $\Ext^n_R(F,D)=0=\Ext^n_R(F,C)$ for all $D$ and $C$
as in~(i) and~(ii) and all $n\ge1$, one has $\Ext^n_R(F,B)=0$ for all
left $R$\+modules $B$ that can be obtained from such $R$\+modules
as $D$ and $C$ using extensions and cokernels of injections.
 Since all $U$\+weakly cotorsion left $R$\+modules can be so obtained
by Theorem~\ref{weakly-cotorsion-obtainable}, we can conclude that
$F$ is $U$\+strongly flat.

 A slightly more careful analysis of the specific procedure for
producing $U$\+weakly cotorsion left $R$\+modules out of
the left $U$\+modules and the $\bG$\+separated $\bG$\+complete
$R$\+modules used in the proof of the only if-part of
Theorem~\ref{weakly-cotorsion-obtainable} reveals that
the $\Ext^1$ vanishing in~(i) and the $\Ext^{1,2}$ vanishing
in~(ii) are enough.
\end{proof}

 Let $F$ be a flat left $R$\+module.
 Then for any left $U$\+module $D$ there are natural isomorphisms of
abelian groups $\Ext^n_R(F,D)\simeq\Ext^n_U(U\ot_RF,\>D)$
for all $n\ge0$.
 (Indeed, given a projective resolution $P_\bu$ of the left $R$\+module
$F$, the complex $U\ot_RP_\bu$ is a projective resolution of the left
$U$\+module $U\ot_RF$, and the two complexes of abelian groups
$\Hom_R(P_\bu,D)$ and $\Hom_U(U\ot_RP_\bu,D)$ are naturally isomorphic
by the tensor-Hom adjunction.)
 Therefore, the condition~(i) in
Corollary~\ref{strongly-flat-characterized} holds if and only if
the left $U$\+module $U\ot_RF$ is projective.
 Assuming that the right Gabriel topology $\bG$ \emph{has a countable
base consisting of two-sided ideals}, the condition~(ii) can be
also reformulated in a similar way, as we will now show.

 Given an associative ring $R$ and a sequence of left $R$\+modules
$C_0$, $C_1$, $C_2$,~\dots\ indexed by nonnegative integers, we will
say that a left $R$\+module $C$ is an \emph{infinitely iterated
extension} of the left $R$\+modules $C_i$ \emph{in the sense of
the projective limit} if there exists a decreasing filtration
$C=G^0\supset G^1\supset G^2\supset\dotsb$ of $C$ by its $R$\+submodules
$G^i$ such that the natural $R$\+module morphism $C\rarrow
\varprojlim_iC/G^i$ is an isomorphism and the quotient module
$G^i/G^{i+1}$ is isomorphic to $C_i$ for every $i\ge0$.
 The dual Eklof lemma~\cite[Proposition~18]{ET} tells that if
a left $R$\+module $C$ is an infinitely iterated extension of
left $R$\+modules $C_i$ in the sense of the projective limit and $F$
is a left $R$\+module such that $\Ext^1_R(F,C_i)=0$ for all $i>0$,
then $\Ext_R^1(F,C)=0$.

\begin{thm} \label{two-sided-base-weakly-cotorsion-obtainable}
 Let $u\:R\rarrow U$ be a ring epimorphism such that $U$ is a flat
left $R$\+module and the related Gabriel topology of right ideals\/
$\bG$ on $R$ has a countable base consisting of two-sided ideals.
 Then a left $R$\+module is $U$\+weakly cotorsion if and only if
it can be obtained from left $R$\+modules belonging to the union of
the following classes:
\begin{itemize}
\item left $U$\+modules
\item left modules over quotient rings $R/H$ of the ring $R$ by its
two-sided ideals $H\subset R$ belonging to\/~$\bG$
\end{itemize}
using the operations of the passage to an extension of two $R$\+modules,
to an infinitely iterated extension of a sequence of $R$\+modules,
in the sense of the projective limit, and to the cokernel of
an injective $R$\+module morphism.
\end{thm}

\begin{proof}
 By the dual Eklof lemma~\cite[Proposition~18]{ET}, the class of all
$U$\+weakly cotorsion left $R$\+modules is closed under infinitely
iterated extensions in the sense of the projective limit.
 Since all left $R/H$\+modules are $\bG$\+separated $\bG$\+complete
left $R$\+modules, the if-part follows from
Theorem~\ref{weakly-cotorsion-obtainable}.

 To prove the ``only if'', it remains to observe that every
$\bG$\+separated $\bG$\+complete left $R$\+module is an infinitely
iterated extension of a sequence left $R/H$\+modules.
 Indeed, let $R\supset H_1\supset H_2\supset\dotsb$ be a decreasing
sequence of open two-sided ideals belonging to $\bG$ such that
the collection $\bH$ of all the ideals $H_i$, \ $i\ge1$, is
a base of the topology~$\bG$.
 Given a $\bG$\+separated $\bG$\+complete left $R$\+module $C$, set
$G^i=H_iC\subset C$ for every $i\ge1$.
 This is a decreasing filtration of $C$ by its $R$\+submodules,
the natural map $C\rarrow\varprojlim_i C/G^i$ is an isomorphism,
and the quotient module $G^i/G^{i+1}$ is an $(R/H_{i+1})$\+module
for every $i\ge1$.
\end{proof}

 The following corollary can be thought of as confirming a version
of~\cite[Optimistic Conjecture~1.1]{PSl} for perfect Gabriel
topologies with a countable base of two-sided ideals.
 This corollary is also a generalization of~\cite[Theorem~1.3]{PSl}
(while the previous theorem is a generalization
of~\cite[Proposition~1.6]{PSl}).
 We refer to the introduction to~\cite{PSl}
(see~\cite[Sections~1.1\+-1.3]{PSl}) for a discussion.

\begin{cor} \label{two-sided-base-strongly-flat-characterized}
 Let $u\:R\rarrow U$ be a ring epimorphism such that $U$ is a flat
left $R$\+module and the related Gabriel topology of right ideals\/
$\bG$ on $R$ has a countable base\/ $\bB$ consisting of
two-sided ideals.
 Then a flat left $R$\+module $F$ is $U$\+strongly flat if and only
if it satisfies the following two conditions:
\begin{enumerate}
\renewcommand{\theenumi}{\roman{enumi}}
\item the left $U$\+module $U\ot_RF$ is projective;
\item for every open two-sided ideal $H\subset R$, \ $H\in\bB$,
the left $R/H$\+module $F/HF$ is projective.
\end{enumerate}
\end{cor}

\begin{proof}
 The only if-part is easy to prove, and it does not depend on
the assumption of countability of~$\bB$.
 Indeed, suppose that a left $R$\+module $F$ is a direct summand of
a left $R$\+module $G$ included in a short exact sequence of (flat) left
$R$\+modules $0\rarrow V\rarrow G\rarrow W\rarrow 0$, where
$V$ is a free left $R$\+module and $W$ is a free left $U$\+module.
 Then the left $U$\+module $U\ot_RG$ is free, since the left
$U$\+modules $U\ot_RV$ and $U\ot_RW\simeq W$ are;
and the left $R/H$\+module $G/HG$ is free for any two-sided ideal
$H\in\bG$, since the left $R/H$\+modules $V/HV$ and $W/HW=0$ are.
 Hence the left $U$\+module $U\ot_RF$ and the left $R/H$\+module
$F/HF$ are projective.

 Now we return to the assumption of a countable base of two-sided ideals
$\bB$ in $\bG$ and prove both the if- and only if-parts.
 The same argument as in the proof of
Corollary~\ref{strongly-flat-characterized}, but based on
Theorem~\ref{two-sided-base-weakly-cotorsion-obtainable} instead
of Theorem~\ref{weakly-cotorsion-obtainable} and using also
the dual Eklof lemma, shows that an $R$\+module $F$ is
$U$\+strongly flat if and only if $\Ext^n_R(F,D)=0=\Ext^n_R(F,C)$
for all left $U$\+modules $D$, all left $R/H$\+modules $C$
with $H\in\bB$, and all $n\ge1$.
 Assume that $F$ is flat left $R$\+module.
 We have already seen above that the condition $\Ext^n_R(F,D)=0$
for all $D\in U\modl$ and $n\ge1$ (or just $n=1$) holds if and
only if the left $U$\+module $U\ot_RF$ is projective.
 Similarly, given any fixed $H\in\bB$, one has $\Ext^n_R(F,C)\simeq
\Ext^n_{R/H}(F/HF,\>C)$ for all $C\in R/H\modl$ and all $n\ge0$.
 Thus the condition that $\Ext^n_R(F,C)=0$ for all such $C$ and
$n\ge1$ (or just $n=1$) holds if and only if the left
$R/H$\+module $F/HF$ is projective.
\end{proof}

\Section{When $\Delta$ equals $\Lambda$?}
\label{when-Delta=Lambda-secn}

 Let $u\:R\rarrow U$ be an epimorphism of associative rings such that
$U$ is a left $R$\+module of projective dimension at most~$1$.
 Let us introduce the name \emph{left $u$\+con\-tramodules} (or
\emph{$u$\+contramodule left $R$\+modules}) for objects of
the Geigle--Lenzing perpendicular subcategory
$U^{\perp_{0,1}}\subset R\modl$.
 Let us also introduce the notation $R\modl_{u\ctra}=U^{\perp_{0,1}}$
for the full subcategory of $u$\+contramodules in $R\modl$. 
 We recall that, according to~\cite[Proposition~1.1]{GL},
\,$R\modl_{u\ctra}$ is an abelian category and its embedding
$R\modl_{u\ctra}\rarrow R\modl$ is an exact functor.

 Given an associative ring homomorphism $u\:R\rarrow U$, we consider
the two-term complex of $R$\+$R$\+bimodules $K^\bu_{R,U}=(R\to U)$, as
in Section~\ref{projdim-1-secn}, and for any left $R$\+module $M$ set
$$
 \Delta_u(M)=\Ext^1_R(K^\bu_{R,U},M).
$$
 This defines a functor $\Delta_u\:R\modl\rarrow R\modl$.
 For every left $R$\+module $M$, there is a natural left $R$\+module
morphism $\delta_{u,M}\:M\rarrow\Delta_u(M)$ appearing in the exact
sequence~\eqref{main-five-term-sequence} (cf.~\cite[exact sequence~(9)
in Section~16]{BP}).

 The following result goes back to~\cite[Proposition~2.4]{Mat}
(see also~\cite[Theorem~3.4(b)]{PMat}).

\begin{lem} \label{reflector-delta}
 Let $u\:R\rarrow U$ be an epimorphism of associative rings such that\/
$\pd{}_RU\le1$.
 Then, for every left $R$\+module $M$, the left $R$\+module\/
$\Delta_u(M)$ is a $u$\+contramodule.
 The functor\/ $\Delta_u\:R\modl\rarrow R\modl_{u\ctra}$ is left adjoint
to the embedding\/ $R\modl_{u\ctra}\rarrow R\modl$.
 The natural map $\delta_{u,M}\:M\rarrow\Delta_u(M)$ is the adjunction
morphism.
\end{lem}

\begin{proof}
 This is~\cite[Proposition~17.2(b)]{BP}.
\end{proof}

 The aim of this section is to establish a sufficient condition for
an isomorphism between the $u$\+contramodule left $R$\+module
$\Delta_u(M)$ and (the underlying $R$\+module of) the left
$\R$\+contramodule $\Lambda_\bG(M)=\varprojlim_{I\in\bG}M/IM$ constructed
in the first half of Section~\ref{strongly-flat-weakly-cotorsion-secn}.
 This will provide a generalization of the classical result of
Matlis~\cite[Theorem~6.10]{Mat} (see also~\cite[Theorem~2.5]{PMat})
to our setting.

 In fact, we will even obtain a sufficient condition for an isomorphism
$\Delta_u(M)\simeq\Lambda_\bG(M)$ applicable for any left flat ring
epimorphism $u\:R\rarrow U$ (\emph{without} the assumption on
the projective dimension of the left $R$\+module $U$).
 But we will need some technical assumptions on the right Gabriel
topology $\bG$ related to~$u$, which we inherit from
Section~\ref{perfect-of-type-lambda-secn}.

 The next proposition is a generalization of
Proposition~\ref{orthogonality-prop}(c).
 It is also a generalization of~\cite[Examples~2.4(2) and~2.5(2)]{Pper}.

\begin{prop} \label{uncountable-orthogonality}
 Let $u\:R\rarrow U$ be an epimorphism of associative rings such that
$U$ is a flat left $R$\+module of projective dimension
not exceeding\/~$1$, let\/ $\bG$ be the perfect Gabriel topology of
right ideals in $R$ associated with the left flat epimorphism~$u$, and
let\/ $\R$ be the completion of $R$ with respect to\/ $\bG$, viewed as
a topological ring in the projective limit topology\/~$\bfG$.
 Then the image of the forgetful functor\/ $\R\contra\rarrow R\modl$ is
contained in the full subcategory $R\modl_{u\ctra}\subset R\modl$.
\end{prop}

\begin{proof}
 The full subcategory of $u$\+contramodule left $R$\+modules
$R\modl_{u\ctra}$ is closed under kernels, cokernels, extensions, and
infinite products in $R\modl$.
 The forgetful functor $\R\contra\rarrow R\modl$ preserves the kernels,
cokernels, and infinite products.
 Any left $\R$\+contramodule is the cokernel of a morphism of free
left $\R$\+contramodules; so it suffices to show that the underlying
left $R$\+module of the free left $\R$\+contramodule $\R[[X]]$ belongs
to $R\modl_{u\ctra}$ for every set~$X$.

 For any complete, separated topological ring $\R$ in a right linear
topology $\bfF$, the free left $\R$\+contramodules $\R[[X]]$ are
separated and complete, since
$$
 \R[[X]]=\varprojlim\nolimits_{\I\in\bfF}(\R/\I)[X]=
 \varprojlim\nolimits_{\I\in\bfF}\R[[X]]/\I[[X]]=
 \varprojlim\nolimits_{\I\in\bfF}\R[[X]]/(\I\tim\R[[X]]).
$$
 One can also observe that $\R[[X]]=\PL(F^{(X)})$, where $F^{(X)}\:
\R/\I\longmapsto(\R/\I)^{(X)}$ is the coproduct of $X$ copies of
the identity/forgetful covariant $\bfF$\+system of abelian groups
$F\:\R/\I\longmapsto\R/\I$ assigning to a cyclic discrete right
$\R$\+module $\R/\I$ the abelian group $\R/\I$ for every $\I\in\bfF$.
 Following the second proof of
Proposition~\ref{separated-kernel-of-homs}, for every covariant
$\bfF$\+system $D\:\sQ_\bfF\rarrow\Ab$ the left $\R$\+contramodule
$\PL(D)$ is the kernel of a morphism between two left
$\R$\+contramodules of the form $\Hom_\boZ(\N',\boQ/\boZ)$ and
$\Hom_\boZ(\N'',\boQ/\boZ)$, where $\N'$ and $\N''$ are discrete right
$\R$\+modules.

 Thus, returning to the situation at hand, it remains to show that
$\Hom_\boZ(\N,\boQ/\boZ)\in R\modl_{u\ctra}$ for every discrete right
$\R$\+module~$\N$.
 The latter assertion is provable by the explicit argument from
the first half of the proof of Proposition~\ref{orthogonality-prop}(a).
\end{proof}

 Let $u\:R\rarrow U$ be a ring epimorphism such $U$ is a flat left
$R$\+module of projective dimension at most~$1$, and let $\bG$ be
the related perfect right Gabriel topology on~$R$.
 Let $M$ be a left $R$\+module.
 The left $R$\+module morphism $\lambda_{\bG,M}\:M\rarrow\Lambda_\bG(M)$
has a left $\R$\+contramodule as its target; by
Proposition~\ref{uncountable-orthogonality}, it follows that
$\Lambda_\bG(M)$ is a $u$\+contramodule left $R$\+module.
 By Lemma~\ref{reflector-delta}, the left $R$\+module map
$\delta_{u,M}\:M\rarrow\Delta_u(M)$ is the universal morphism from $M$
to a $u$\+contramodule left $R$\+module.
 Thus there exists a unique left $R$\+module morphism
$$
 \beta_{u,M}\:\Delta_u(M)\lrarrow\Lambda_\bG(M)
$$
making the triangle diagram of $R$\+module morphisms
$M\rarrow\Delta_u(M)\rarrow\Lambda_\bG(M)$ commutative.

\begin{rem} \label{free-module-beta}
 When $M=R[X]$ is a \emph{free} left $R$\+module, our arguments allow
to deduce the existence and uniqueness of a left $R$\+module
morphism
$$
 \beta_{u,X}=\beta_{u,R[X]}\:\Delta_u(R[X])\rarrow\Lambda_\bG(R[X])
$$
forming commutative triangle diagram with the maps~$\delta_{u,R[X]}$
and $\lambda_{\bG,R[X]}$ for \emph{any} left flat ring epimorphism
$u\:R\rarrow U$ (\emph{without} the assumption that $\pd{}_RU\le1$).

 Indeed, even more generally, let $\bF$ be a right linear topology on
an associative ring~$R$, and let $\R$ be the completion of $R$ with
respect to $\bF$, viewed as a topological ring in its projective limit
topology~$\bfF$.
 Then, for any set $X$, the $\bF$\+completion of the free left
$R$\+module $R[X]$ is the free left $\R$\+contramodule $\R[[X]]$,
$$
 \R[[X]]=\varprojlim\nolimits_{\I\in\bfF}(\R/\I)[X]=
 \varprojlim\nolimits_{I\in\bF}R[X]/IR[X]=\Lambda_\bF(R[X]).
$$
 In fact, one has $\R[[X]]=\PL(F^{(X)})$, where $F^{(X)}=\tp(R[X])$
is the coproduct of $X$ copies of the identity/forgetful covariant
$\bF$\+system of abelian groups $F\:R/I\longmapsto R/I$
(as in the proof of Proposition~\ref{uncountable-orthogonality}).
 The completion map~$\lambda_{\bF,R[X]}$ can be described as
the unique left $R$\+module morphism $R[X]\rarrow\R[[X]]$ whose
restriction to the subset of generators $X\subset R[X]$ is
the identity inclusion $X\rarrow\R[[X]]$.

 Following the proof of Proposition~\ref{uncountable-orthogonality},
even without the assumption on the projective dimension of the left
$R$\+module $U$ one has $\R[[X]]\in U^{\perp_{0,1}}\subset R\modl$ for
any set~$X$.
 This makes the existence and uniqueness of the map~$\beta_{u,X}$
provable by the classical Matlis' argument~\cite[Proposition~2.4]{Mat}
(see~\cite[Lemma~2.1(b)]{PMat} for a recent exposition), which can be
easily adopted to the situation at hand.
 (The very same argument is also used in the proof of
Lemma~\ref{reflector-delta}.)
\end{rem}

 Using an idea of Facchini and Nazemian~\cite[Sections~3\+-4]{FN}, we
will now construct a map in the opposite direction, under certain
assumption.
 Namely, let $u\:R\rarrow U$ be a left flat ring epimorphism and
$\bG$ be the related perfect Gabriel topology of right ideals in~$R$.
 Let $M$ be a left $R$\+module.
 Assume that the map $M\rarrow U\ot_RM$ induced by the ring homomorphism
$u\:R\rarrow U$ is injective.
 In the terminology of~\cite[Section~16]{BP}, this means that
the left $R$\+module $M$ is \emph{$u$\+torsion-free}.
 (Let the reader be warned that the class of all $u$\+torsion-free
left $R$\+modules does \emph{not} need to be a torsion-free class in
$R\modl$ in our assumptions; see the discussion in \emph{loc.\ cit}.)

 Then we have a short exact sequence of left $R$\+modules
$0\rarrow M\rarrow U\ot_RM\rarrow U/R\ot_RM\rarrow0$, where
$U/R$ is a shorthand notation for the $R$\+$R$\+bimodule $\coker u$.
 Applying the functor $\Ext^*_R(K^\bu_{R,U},{-})$, we obtain the induced
morphism (the connecting homomorphism in the long exact sequence)
of left $R$\+modules
$$
 \Hom_R(U/R,\>U/R\ot_RM)=\Ext^0_R(K^\bu_{R,U},\>U/R\ot_RM)
 \rarrow\Ext^1_R(K^\bu_{R,U},M)=\Delta_u(M).
$$
 The same morphism can be also constructed as the composition
\begin{equation} \label{hom-to-delta}
 \Hom_R(U/R,\>U/R\ot_RM)\rarrow\Ext^1_R(U/R,M)\rarrow
 \Ext^1_R(K^\bu_{R,U},M)=\Delta_u(M),
\end{equation}
where the left $R$\+module morphism $\Hom_R(U/R,\>U/R\ot_RM)
\rarrow\Ext^1_R(U/R,M)$ is the connecting homomorphism in the long
exact sequence obtained by applying the functor $\Ext^*_R(U/R,{-})$
to the same short exact sequence of left $R$\+modules, and
the morphism $\Ext^1_R(U/R,M)\rarrow\Ext^1_R(K^\bu_{R,U},M)$ is induced
by the natural morphism of complexes of $R$\+$R$\+bimodules
$K^\bu_{R,U}\rarrow U/R$.

 For any associative ring $A$, any $A$\+$R$\+bimodule $\N$ with
a $\bG$\+torsion underlying right $R$\+module, any left
$A$\+module $V$, and any right ideal $I\in\bG$, the subgroup
$I\Hom_A(\N,V)\subset\Hom_A(\N,V)$ is contained in the kernel
of the restriction map $\Hom_A(\N,V)\rarrow\Hom_A(\N_I,V)$, where
$\N_I\subset\N$ denotes the left $A$\+submodule of all elements
annihilated by the right action of $I$ in~$\N$.
 Hence we have a natural morphism of abelian groups
\begin{multline} \label{erasing-lambda}
 \Lambda_\bG(\Hom_A(\N,V))\,=\,
 \varprojlim\nolimits_{I\in\bG}\Hom_A(\N,V)/I\Hom_A(\N,V) \\ \lrarrow
 \varprojlim\nolimits_{I\in\bG}\Hom_A(\N_I,V)\,=\,\Hom_A(\N,V).
\end{multline}
 In fact, this is a left $\R$\+contramodule morphism, which can be
obtained by applying the functor $\PL$ to a natural morphism from
the covariant $\bG$\+system of abelian groups $\tp(\Hom_A(\N,V))$
to the covariant $\bG$\+system of abelian groups
$R/I\longmapsto\Hom_A(\Dh(\N)(R/I),V)$.
 The composition $\Hom_A(\N,V)\rarrow\Lambda_\bG(\Hom_A(\N,V))
\rarrow\Hom_A(\N,V)$ of the map $\lambda_{\bG,\Hom_A(\N,V)}$ with
our map~\eqref{erasing-lambda} is the identity map (so the left
$R$\+module $\Hom_A(\N,V)$ is a direct summand of the left $R$\+module
$\Lambda_\bG(\Hom_A(\N,V))$).

 Now we set $A=R$, \ $\N=U/R$, and $V=U/R\ot_RM$; and consider
the composition
\begin{multline} \label{theta-construction}
 \Lambda_\bG(M)\lrarrow\Lambda_\bG(\Hom_R(U/R,\>U/R\ot_RM)) \\
 \lrarrow\Hom_R(U/R,\>U/R\ot_RM)\lrarrow\Delta_u(M)
\end{multline}
of the $R$\+module morphism obtained by applying $\Lambda_\bG$ to
the natural morphism $M\rarrow\Hom_R(U/R,\>U/R\ot_RM)$,
the $R$\+module morphism~\eqref{erasing-lambda}, and
the $R$\+module morphism~\eqref{hom-to-delta}.
 We will denote the left $R$\+module morphism~\eqref{theta-construction}
so constructed by
$$
 \theta_{\bG,M}\:\Lambda_\bG(M)\lrarrow\Delta_u(M).
$$
 The morphism~$\theta_{\bG,M}$ is well-defined for any $u$\+torsion-free
left $R$\+module~$M$.

\begin{lem} \label{theta-commutative-triangle}
 Let $u\:R\rarrow U$ be an epimorphism of associative rings such that
$U$ is a flat left $R$\+module, and let\/ $\bG$ be the related perfect
Gabriel topology of right ideals in~$R$.
 Then for any $u$\+torsion-free left $R$\+module $M$ the triangle
diagram formed by the left $R$\+module morphisms $\lambda_{\bG,M}\:
M\rarrow\Lambda_\bG(M)$, \ $\delta_{u,M}\:M\rarrow\Delta_u(M)$,
and $\theta_{\bG,M}\:\Lambda_\bG(M)\rarrow\Delta_u(M)$ is commutative.
\end{lem}

\begin{proof}
 Applying the functor $\Lambda_\bG$ and the natural
transformation~$\lambda_\bG$ to the natural morphism $M\rarrow
\Hom_R(U/R,\>U/R\ot_RM)$ produces a commutative square diagram.
 The composition $\Hom_R(U/R,\>U/R\ot_RM)\rarrow
\Lambda_\bG(\Hom_R(U/R,\>U/R\ot_RM))\rarrow\Hom_R(U/R,\>U/R\ot_RM)$
is the identity map, as we mentioned above.

 It remains to check commutativity of the triangle diagram
$M\rarrow\Hom_R(U/R,\allowbreak\>U/R\ot_RM)\rarrow\Delta_u(M)$ formed
by the natural map $M\rarrow\Hom_R(U/R,\>U/R\ot_RM)$, the map
$\Hom_R(U/R,\>U/R\ot_RM)\rarrow\Delta_u(M)$\,~\eqref{hom-to-delta},
and the map $\delta_{u,M}\:M\rarrow\Delta_u(M)$.
 For this purpose, it suffices to show that the triangle diagram
$M\rarrow\Hom_R(U/R,\>U/R\ot_RM)\rarrow\Ext^1_R(U/R,M)$ is
commutative, as the morphism~$\delta_{u,M}$ is the composition
$M\rarrow\Ext^1_R(U/R,M)\rarrow\Ext^1_R(K^\bu_{R,U},M)$.

 Here the map $\Hom_R(U/R,\>U/R\ot_RM)\rarrow\Ext^1_R(U/R,M)$ is
obtained by applying the functor $\Ext^*_R(U/R,{-})$ to the short exact
sequence of left $R$\+modules $0\rarrow M\rarrow U\ot_RM\rarrow
U/R\ot_RM\rarrow0$.
 The map $M\rarrow\Ext^1_R(U/R,M)$ is obtained by applying the functor
$\Ext^*_R({-},M)$ to the short exact sequence of left $R$\+modules
$0\rarrow \overline R\rarrow U\rarrow U/R\rarrow 0$, where
$\overline R$ is the image of the map~$u$.
 Any left $R$\+module morphism $R\rarrow M$ factorizes through
the surjection $R\rarrow\overline R$, since $M\subset U\ot_RM$;
so one has $M=\Hom_R(R,M)=\Hom_R(\overline R,M)$.

 Given an element $m\in M$, one can explicitly check that the two
related Yoneda extension classes in $\Ext^1_R(U/R,M)$ indeed coincide,
by constructing an isomorphism between the two related short exact
sequences of left $R$\+modules.
\end{proof}

\begin{rem} \label{ext-1-from-K-remark}
 For any associative ring homomorphism $u\:R\rarrow U$ and any left
$R$\+module $M$ there is a natural short exact sequence of
left $R$\+modules
\begin{multline*}
 0\lrarrow\Ext^1_R(U/R,M)\lrarrow\Ext^1_R(K^\bu_{R,U},M)\\ \lrarrow
 \Hom_R(H,M)\lrarrow\Ext^2_R(U/R,M)\lrarrow\dotsb,
\end{multline*}
where $H=\ker u\subset R$ (cf.\ the proof of
Lemma~\ref{ext-from-two-term-complex-lemma}).
 In particular, the map $\Ext^1_R(U/R,M)\rarrow\Ext^1_R(K^\bu_{R,U},M)
=\Delta_u(M)$ is always injective.
 Hence it follows from the construction of the map~$\theta_{\bG,M}$
and from the next lemma that, in the assumptions of the latter,
the natural map $\Ext^1_R(U/R,M)\rarrow\Delta_u(M)$ is an isomorphism.
\end{rem}

\begin{lem} \label{xi-is-identity}
 Let $u\:R\rarrow U$ be an epimorphism of associative rings such that
$U$ is a flat left $R$\+module of projective dimension
not exceeding~$1$, and let\/ $\bG$ be the related perfect Gabriel
topology of right ideals in~$R$.
 Let $M$ be a $u$\+torsion-free left $R$\+module.
 Then the composition $\xi=\theta\beta\:\Delta_u(M)\rarrow\Delta_u(M)$
of the left $R$\+module morphisms $\beta_{u,M}\:\Delta_u(M)\rarrow
\Lambda_\bG(M)$ and $\theta_{\bG,M}\:\Lambda_\bG(M)\rarrow\Delta_u(M)$
is the identity map, $\xi=\id$.
\end{lem}

\begin{proof}
 By the definition, the left $R$\+module morphism~$\beta_{u,M}$
forms a commutative triangle diagram with the morphisms~$\delta_{u,M}$
and~$\lambda_{\bG,M}$.
 By Lemma~\ref{theta-commutative-triangle}, the left $R$\+module
morphism~$\theta_{\bG,M}$ forms a commutative triangle diagram with
the morphisms~$\delta_{u,M}$ and~$\lambda_{\bG,M}$ as well.
 Hence it follows that the composition $\xi=\theta\beta$ is
a left $R$\+module morphism forming a commutative triangle diagram
with the morphism $\delta_{u,M}\:M\rarrow\Delta_u(M)$,
that is $\xi\delta=\delta$.
 In view of Lemma~\ref{reflector-delta}, one can conclude that
$\xi=\id$ is the identity morphism.
\end{proof}

 Similarly, it also follows from Lemma~\ref{theta-commutative-triangle}
that the composition $\Lambda_\bG(M)\rarrow\Delta_u(M)\rarrow
\Lambda_\bG(M)$ is a left $R$\+module morphism $\zeta=\beta\theta\:
\Lambda_\bG(M)\rarrow\Lambda_\bG(M)$ forming a commutative triangle
diagram with the morphism $\lambda_{\bG,M}\:M\rarrow\Lambda_\bG(M)$,
that is $\zeta\lambda=\lambda$.
 The next lemma allows to prove that $\zeta$~is the identity morphism,
too, under certain assumptions.

\begin{lem} \label{zeta-is-identity}
 Let $R$ be an associative ring with a right linear topology\/ $\bF$
having a countable base, and let\/ $\R$ be the completion of $R$
with respect to\/ $\bF$, viewed as a topological ring in its projective
limit topology\/~$\bfF$.
 Let $M$ be a left $R$\+module, let $\lambda_{\bF,M}\:R\rarrow
\Lambda_\bF(M)$ be the natural left $R$\+module morphism from $M$ into
its\/ $\bF$\+completion $\Lambda_\bF(M)$, and let $\zeta\:
\Lambda_\bF(M)\rarrow\Lambda_\bF(M)$ be a left\/ $\R$\+contramodule
morphism such that $\zeta\lambda_{\bF,M}=\lambda_{\bF,M}$.
 Then $\zeta$~is the identity map, $\zeta=\id$.
\end{lem}

\begin{proof}
 By the definition, we have $\Lambda_\bF(M)=\PL(\tp(M))$.
 Set $\M=\Lambda_\bF(M)$; then, by
Lemma~\ref{completions-unconfused-lem}(a), we have
$M/IM\simeq\M/\I\tim\M$ for any open right ideal $I\in\bF$ in $R$
and the corresponding open right ideal $\I\in\bfF$ in~$\R$.
 So the map $M\rarrow\M/\I\tim\M$ is surjective.
 Hence for any element $m\in\Lambda_\bF(M)$ there exists an element
$m'\in M$ such that $m\in\lambda(m')+\I\tim\M\subset\M$.
 Applying the map~$\zeta$, we conclude that
$$
 \zeta(m)\,\in\,\zeta\lambda(m')+\zeta(\I\tim\M)\,=\,
 \lambda(m')+\zeta(\I\tim\M)\,\subset\,
 \lambda(m')+\I\tim\M,
$$
as $\zeta$~is a left $\R$\+contramodule morphism by assumption.
 Thus the difference $m-\zeta(m)$ belongs to $\I\tim\M\subset\M$
for every $\I\in\bfF$.
 Since $\M$ is a separated $\R$\+contramodule (again by
Lemma~\ref{completions-unconfused-lem}(a)), it follows that
$m-\zeta(m)=0$.
\end{proof}

\begin{cor} \label{countable-Delta=Lambda}
 Let $u\:R\rarrow U$ be an associative ring epimorphism such that $U$
is a flat left $R$\+module, and let\/ $\bG$ be the related perfect
right Gabriel topology on~$R$.
 Assume that\/ $\bG$ has a countable base.
 Let $M$ be a $u$\+torsion-free left $R$\+module.
 Then\/ $\beta_{u,M}\:\Delta_u(M)\rarrow\Lambda_\bG(M)$ and\/
$\theta_{\bG,M}\:\Lambda_\bG(M)\rarrow\Delta_u(M)$ are mutually inverse
isomorphisms of left $R$\+modules,
$$
 \beta_{u,M}\:\Delta_u(M)\simeq\Lambda_\bG(M):\!\theta_{\bG,M}.
$$
\end{cor}

\begin{proof}
 By Theorem~\ref{projdim-1-theorem}, we have $\pd{}_RU\le1$, so
our construction of the map~$\beta_{u,M}$ is applicable.
 By Lemma~\ref{xi-is-identity}, \,$\xi=\theta\beta$ is the identity map.
 Concerning $\zeta=\beta\theta$, it follows from
Lemma~\ref{theta-commutative-triangle} and from the definition
of~$\beta_{u,M}$ that $\zeta\lambda=\lambda$,
as we have already mentioned.
 By construction, $\zeta$~is a left $R$\+module morphism.
 In our assumptions on $\bG$, it follows by virtue of
Corollary~\ref{countable-gabriel-fully-faithful-cor} that
$\zeta$~is a left $\R$\+contramodule morphism, too.
 Applying Lemma~\ref{zeta-is-identity}, we see that
$\zeta=\id$.
\end{proof}

 In the rest of this section, we use the results of
Section~\ref{perfect-of-type-lambda-secn} in order to extend
the result of  Corollary~\ref{countable-Delta=Lambda} to the case
of uncountable type.

 Assume that $u'\:R\rarrow U'$ and $u''\:R\rarrow U''$ are two
epimorphisms of associative rings such that both $U'$ and $U''$ are
flat left $R$\+modules of projective dimension at most~$1$.
 Let $\bG'$ and $\bG''$ be the related right Gabriel topologies on~$R$.
 Suppose that the morphism~$u''$ factorizes through~$u'$, i.~e.,
there is a ring homomorphism $U'\rarrow U''$ making the triangle
diagram $R\rarrow U'\rarrow U''$ commutative.
 Then one has $\bG'\subset\bG''$.

 Let $\R'$ and $\R''$ be the completions of the ring $R$ with respect
to its right linear (Gabriel) topologies $\bG'$ and~$\bG''$.
 As usual, we view the rings $\R'$ and $\R''$ as complete, separated
topological rings in their respective projective limit topologies
$\bfG'$ and~$\bfG''$.
 Then there is a natural continuous homomorphism of topological rings
$\R''\rarrow\R'$; so any left $\R'$\+contramodule can be also
considered as a left $\R''$\+contramodule.
 For any left $R$\+module $M$, there is a natural morphism of left
$R$\+modules (in fact, of left $\R''$\+contramodules)
$\Lambda_{\bG''}(M)\rarrow\Lambda_{\bG'}(M)$.

 Furthermore, one easily observes that $\Ext^*_R(E,C)=0$ for every
left $U'$\+module $E$ and every $u'$\+contramodule left $R$\+module~$C$.
 In particular, this holds for $E=U''$; hence $R\modl_{u'\ctra}\subset
R\modl_{u''\ctra}$ (cf.~\cite[Lemma~1.2]{PMat}
and~\cite[Lemma~1.1(2)]{BP}).
 Hence, for every left $R$\+module $M$, there is a unique left
$R$\+module morphism $\Delta_{u''}(M)\rarrow\Delta_{u'}(M)$ forming
a commutative triangle diagram with the morphisms
$\delta_{u'',M}\:M\rarrow\Delta_{u''}(M)$ and
$\delta_{u',M}\:M\rarrow\Delta_{u'}(M)$.
 This map $\Delta_{u''}(M)\rarrow\Delta_{u'}(M)$ is induced by
the morphism of complexes of $R$\+$R$\+bimodules $K_{R,U'}^\bu
\rarrow K_{R,U''}^\bu$.

 Finally, for any every left $R$\+module $M$ there is a commutative
diagram of left $R$\+module morphisms
\begin{equation} \label{delta-lambda-diagram}
\begin{diagram}
\node{\Delta_{u''}(M)}\arrow{e,t}{\beta_{u'',M}}\arrow{s}
\node{\Lambda_{\bG''}(M)} \arrow{s} \\
\node{\Delta_{u'}(M)}\arrow{e,t}{\beta_{u',M}}\node{\Lambda_{\bG'}(M)}
\end{diagram}
\end{equation}
 The square diagram is commutative, since $\Lambda_{\bG'}(M)\in
R\modl_{u''\ctra}$, so there exists a unique left $R$\+module
morphism $\Delta_{u''}(M)\rarrow\Lambda_{\bG'}(M)$ forming
a commutative triangle diagram with the maps $\delta_{u'',M}$
and~$\lambda_{\bG',M}$.

 Suppose that we are given a directed set $\Xi$ of right linear
topologies $\bH$ on an associative ring~$R$.
 Then their union $\bF=\bigcup_{\bH\in\Xi}\bH$ is also a right linear
topology (see Section~\ref{perfect-of-type-lambda-secn}).
 Denote by $\R^\bH$ and $\R^\bF$ the completions of the ring $R$
with respect to these topologies (viewed as complete, separated
topological rings in their projective limit topologies $\bfH$
and~$\bfF$).
 For any left $R$\+module $M$, one has a natural isomorphism of
left $R$\+modules (in fact, of left $\R^\bF$\+contramodules)
\begin{equation} \label{lambda-is-projlim-of-lambdas}
 \Lambda_\bF(M)=\varprojlim\nolimits_{\bH\in\Xi}\Lambda_\bH(M).
\end{equation}

\begin{lem} \label{projlim-spectral-sequence}
 Let\/ $(U_\upsilon)_{\upsilon\in\Upsilon}$ be a diagram of associative
rings indexed by a directed poset\/ $\Upsilon$ and commutative together
with associative ring homomorphisms $R\rarrow U_\upsilon$ given for
all\/ $\upsilon\in\Upsilon$.
 Set $U=\varinjlim_{\upsilon\in\Upsilon}U_\upsilon$, and let $M$ be
a left $R$\+module.
 Assume that for every\/ $\upsilon\in\Upsilon$ there exists\/
$\upsilon'\in\Upsilon$, \ $\upsilon\le\upsilon'$ such that all left
$R$\+module morphisms $U_{\upsilon'}/R\rarrow M$ vanish.
 Then the natural morphism of left $R$\+modules
$$
 \Ext^1_R(K^\bu_{R,U},M)\lrarrow
 \varprojlim\nolimits_{\upsilon\in\Upsilon}
 \Ext^1_R(K^\bu_{R,U_\upsilon},M)
$$
is an isomorphism.
\end{lem}

\begin{proof}
 There is a spectral sequence of left $R$\+modules
$$
 E_2^{p,q}=\varprojlim\nolimits_{\upsilon\in\Upsilon}^p
 \Ext_R^q(K^\bu_{R,U_\upsilon},M)
 \Longrightarrow E_\infty^n=\Ext_R^n(K^\bu_{R,U},M)
$$
with the differentials $d_r^{p,q}\:E_r^{p,q}\rarrow E_r^{p+r,q-r+1}$
and the limit term $E_\infty^{p,q}=\mathrm{gr}^pE_\infty^{p+q}$.
 One has $E_2^{p,q}=0$ whenever $p<0$ or $q<0$, so in low degrees
this spectral sequence reduces to an exact sequence
$$
 0\lrarrow E_2^{1,0}\lrarrow E_\infty^1\lrarrow E_2^{0,1}
 \lrarrow E_2^{2,0}\lrarrow E_\infty^2.
$$
 Now the assumption that $\Ext_R^0(K_{R,U_{\upsilon'}}^\bu,M)=
\Hom_R(U_{\upsilon'}/R,M)=0$ for a cofinal subset
$\Upsilon'\subset\Upsilon$ of indices~$\upsilon'$ implies
$E_2^{p,0}=0$ for all $p\in\boZ$.
 Hence the map $E_\infty^1\rarrow E_2^{0,1}$ is an isomorphism,
as desired.
\end{proof}

 The following theorem, generalizing
Corollary~\ref{countable-Delta=Lambda}, is the main result
of this section.

\begin{thm} \label{uncountable-Delta=Lambda-theorem}
 Let $u\:R\rarrow U$ be an epimorphism of associative rings such that
$U$ is a flat left $R$\+module of projective dimension at most\/~$1$.
 Let\/ $\bG$ be the related perfect Gabriel topology on $R$; assume
that\/ $\bG$ satisfies the condition~(T$_\omega$) of
Section~\ref{perfect-of-type-lambda-secn} and the right $R$\+module
$R$ has\/ $\omega$\+bounded\/ $\bG$\+torsion.
 Let $M$ be a $u$\+torsion-free left $R$\+module.
 Then\/ $\beta_{u,M}\:\Delta_u(M)\rarrow\Lambda_\bG(M)$ and\/
$\theta_{\bG,M}\:\Lambda_\bG(M)\rarrow\Delta_u(M)$ are mutually inverse
isomorphisms of left $R$\+modules,
$$
 \beta_{u,M}\:\Delta_u(M)\simeq\Lambda_\bG(M):\!\theta_{\bG,M}.
$$
\end{thm}

\begin{proof}
 We will prove that $\beta_{u,M}$ is an isomorphism.
 Let $\Upsilon$ denote the set of all perfect Gabriel topologies
$\bP\subset\bG$ having a countable base.
 According to Corollary~\ref{perfect-gabriel-lambda-directed-cor},
the set $\Upsilon$ is $\omega^+$\+directed by inclusion and we have
$\bG=\bigcup_{\bP\in\Upsilon}\bP$.
 For every perfect Gabriel topology $\bP\in\Upsilon$, denote by
$u_\bP\:R\rarrow U_\bP$ the related left flat ring epimorphism.
 Then Corollary~\ref{perfect-gabriel-lambda-directed-cor} also claims
that $U=\varinjlim_{\bP\in\Upsilon}U_\bP$.

 By Theorem~\ref{projdim-1-theorem}, for any $\bP\in\Upsilon$
the projective dimension of the left $R$\+module $U_\bP$ does
not exceed~$1$.
 So we have a commutative square diagram of left $R$\+module
morphisms (cf.~\eqref{delta-lambda-diagram})
$$
\begin{diagram}
\node{\Delta_u(M)}\arrow{e,t}{\beta_{u,M}}\arrow{s}
\node{\Lambda_\bG(M)} \arrow{s} \\ \node{\Delta_{u_\bP}(M)}
\arrow{e,t}{\beta_{u_\bP,M}}\node{\Lambda_\bP(M)}
\end{diagram}
$$
and a similar diagram for any two perfect Gabriel topologies
$\bP'\subset\bP''$ belonging to~$\Upsilon$.
 Passing to the projective limit, we obtain a diagram of
left $R$\+module morphisms
\begin{equation} \label{projlimit-comparison-diagram}
\begin{diagram}
\node{\Delta_u(M)}\arrow{e,t}{\beta_{u,M}}\arrow{s}
\node{\Lambda_\bG(M)}\arrow{s} \\
\node{\varprojlim\nolimits_{\bP\in\Upsilon}\Delta_{u_\bP}(M)}
\arrow{e}\node{\varprojlim\nolimits_{\bP\in\Upsilon}\Lambda_\bP(M)}
\end{diagram}
\end{equation}

 The natural map $M\rarrow U\ot_RM$, which is injective by assumption,
decomposes as $M\rarrow U_\bP\ot_RM\rarrow U\ot_RM$.
 So the map $M\rarrow U_\bP\ot_RM$ is injective as well.
 Hence the map $\beta_{u_\bP,M}\:\Delta_{u_\bP}(M)\rarrow\Lambda_\bP(M)$
is an isomorphism by Corollary~\ref{countable-Delta=Lambda}.
 Passing to the projective limit, we conclude that the lower
horizontal arrow $\varprojlim_{\bP\in\Upsilon}\Delta_{u_\bP}(M)
\rarrow\varprojlim_{\bP\in\Upsilon}\Lambda_\bP(M)$ in our
diagram~\eqref{projlimit-comparison-diagram} is an isomorphism.

 According to the discussion above in this section
(see~\eqref{lambda-is-projlim-of-lambdas}), the rightmost vertical map
$\Lambda_\bG(M)\rarrow\varprojlim_{\bP\in\Upsilon}\Lambda_\bP(M)$ is
an isomorphism, since $\bG=\bigcup_{\bP\in\Upsilon}\bP$.
 To prove that the leftmost vertical map $\Delta_u(M)\rarrow
\varprojlim_{\bP\in\Upsilon}\Delta_{u_\bP}(M)$ is an isomorphism,
we will check that the assumptions of
Lemma~\ref{projlim-spectral-sequence} hold for the diagram of
associative ring homomorphisms $R\rarrow U_\bP$ indexed by
the poset~$\Upsilon$.

 Indeed, let us show that all left $R$\+module morphisms
$U_\bP/R\rarrow M$ vanish for all $\bP\in\Upsilon$.
 Since $R\rarrow U_\bP$ is a ring epimorphism, we have
$U\ot_RU_\bP\simeq U\ot_{U_\bP}U_\bP\simeq U$.
 In other words, applying the functor $U\ot_R{-}$ to the morphism
$R\rarrow U_\bP$ produces an isomorphism.
 Hence $U\ot_R(U_\bP/R)=0$ and therefore
$$
 \Hom_R(U_\bP/R,M)\subset\Hom_R(U_\bP/R,\>U\ot_RM)=
 \Hom_U(U\ot_R(U_\bP/R),\>U\ot_RM)=0.
$$
 By Lemma~\ref{projlim-spectral-sequence}, the map $\Delta_u(M)
\rarrow\varprojlim_{\bP\in\Upsilon}\Delta_{u_\bP}(M)$ is an isomorphism.

 It follows that the upper horizontal arrow $\beta_{u,M}\:\Delta_u(M)
\rarrow\Lambda_\bG(M)$ in~\eqref{projlimit-comparison-diagram}
is an isomorphism, too, as desired.
 Concerning the map $\theta_{\bG,M}\:\Lambda_\bG(M)\rarrow\Delta_u(M)$,
we know from Lemma~\ref{xi-is-identity} that the composition
$\xi=\theta\beta$ is the identity map.
 Hence $\theta=\beta^{-1}$ is the inverse isomorphism.
\end{proof}

\begin{rem}
 Dropping the assumption that $\pd{}_RU\le1$ in
Theorem~\ref{uncountable-Delta=Lambda-theorem}, one can still prove
existence of a natural isomorphism of left $R$\+modules
$\Delta_u(M)\simeq\Lambda_\bG(M)$ forming a commutative triangle
diagram with the morphisms $\delta_{u,M}$ and~$\lambda_{\bG,M}$.
 The argument works in the same way as above, except that
the map~$\beta_{u,M}$ is not defined from the outset (but
the maps~$\beta_{u_\bP,M}$ are).
 So the diagram~\eqref{projlimit-comparison-diagram} takes the form
$$ 
\begin{diagram}
\node{\Delta_u(M)}\arrow{s} \node{\Lambda_\bG(M)}\arrow{s} \\
\node{\varprojlim\nolimits_{\bP\in\Upsilon}\Delta_{u_\bP}(M)}
\arrow{e}\node{\varprojlim\nolimits_{\bP\in\Upsilon}\Lambda_\bP(M)}
\end{diagram}
$$
 One shows that the leftmost vertical, righmost vertical, and lower
horizontal arrows are isomorphisms, and deduces the existence of
an upper horizontal isomorphism.
\end{rem}

\Section{Faithful Perfect Gabriel Topologies}
\label{faithful-perfect-secn}

 According to Theorem~\ref{injective-epi-u-contra=R-contra}, if $\bG$
is a faithful perfect Gabriel topology with a countable base on
an associative ring $R$ and $R\rarrow U$ is the related injective
left flat ring epimorphism, then the Geigle--Lenzing abelian
perpendicular subcategory $U^{\perp_{0,1}}\subset R\modl$ coincides with
the abelian full subcategory of left $\R$\+contramodules
$\R\contra\subset R\modl$.
 The aim of this section is to replace the assumption that $\bG$
has a countable base with the weaker assumption that $U$ is a left
$R$\+module of projective dimension~$1$ in this result
(cf.\ Theorem~\ref{projdim-1-theorem}).
 The argument is based on the results of
Section~\ref{when-Delta=Lambda-secn} and~\cite[Proposition~2.1]{Pper}.
 We will also give a second, alternative proof of
Theorem~\ref{injective-epi-u-contra=R-contra}, as
promised in Section~\ref{projdim-1-secn}.

 Now let us return to the context (of the formulation) of
Proposition~\ref{uncountable-orthogonality}.
 Let $u\:R\rarrow U$ be an epimorphism of associative rings such that
$U$ is a flat left $R$\+module of projective dimension at most~$1$,
let $\bG$ be the related Gabriel topology of right ideals in $R$,
and let $\R$ be the completion of $R$ with respect to~$\bG$, viewed
as a topological ring in the projective limit topology~$\bfG$.

 Then, as a particular case of the construction of
the map~$\beta_{u,M}$ in Section~\ref{when-Delta=Lambda-secn}
(cf.\ Remark~\ref{free-module-beta}),
we have a unique left $R$\+module morphism
\begin{equation} \label{delta-free-contra-compare}
 \beta_{u,X}=\beta_{u,R[X]}\:\Delta_u(R[X])\lrarrow\R[[X]]
\end{equation}
forming a commutative triangle diagram with the left
$R$\+module morphisms $\lambda_{\bG,R[X]}\:R[X]\rarrow\R[[X]]$ and
$\delta_{u,R[X]}\:R[X]\rarrow\Delta_u(R[X])$.  {\hbadness=1400\par}

\begin{lem} \label{contramodule-categories-comparison-criterion}
 Let $u\:R\rarrow U$ be an associative ring epimorphism such that
$U$ is a flat left $R$\+module of projective dimension
not exceeding\/~$1$, let\/ $\bG$ be the related perfect Gabriel
topology of right ideals in $R$, and let\/ $\R$ be the completion
of $R$ with respect to\/~$\bG$.
 Then the exact functor\/ $\R\contra\rarrow R\modl_{u\ctra}$, as
defined in Proposition~\ref{uncountable-orthogonality}, is 
an equivalence of categories if and only if
the map~$\beta_{u,X}$~\eqref{delta-free-contra-compare}
is an isomorphism for every set~$X$.
\end{lem}

\begin{proof}
 This is a particular case of~\cite[Proposition~2.1]{Pper}.
\end{proof}

\begin{proof}[Second proof of
Theorem~\ref{injective-epi-u-contra=R-contra}]
 Let $u\:R\rarrow U$ be an injective ring epimorphism such that
$U$ is a flat left $R$\+module and the related faithful perfect
right Gabriel topology $\bG$ on $R$ has a countable base.
 Then, by Theorem~\ref{projdim-1-theorem}, one has $\pd{}_RU\le1$.
 Since $u$~is injective, the free left $R$\+module $R[X]$ is
$u$\+torsion-free for every set~$X$.
 Applying Corollary~\ref{countable-Delta=Lambda} for $M=R[X]$,
we conclude that the map~$\beta_{u,X}$ is an isomorphism.
 By Proposition~\ref{uncountable-orthogonality} and
Lemma~\ref{contramodule-categories-comparison-criterion}, it
follows that the forgetful functor $\R\contra\rarrow R\modl$
is fully faithful and its essential image coincides with
the full subcategory of $u$\+contramodule left $R$\+modules
$R\modl_{u\ctra}\subset R\modl$.
\end{proof}

 The following theorem is our uncountable generalization of
Theorem~\ref{injective-epi-u-contra=R-contra}.
 It is also a partial generalization
of~\cite[Examples~2.4(3) and~2.5(3)]{Pper}.

\begin{thm} \label{faithful-topology-theorem}
 Let $u\:R\rarrow U$ be an injective ring epimorphism such that
$U$ is a flat left $R$\+module of projective dimension
not exceeding\/~$1$.
 Let\/ $\bG$ be the related faithful perfect Gabriel topology
of right ideals in~$R$.
 Assume that\/ $\bG$ satisfies the condition (T$_\omega$) of
Section~\ref{perfect-of-type-lambda-secn} (e.~g., $\bG$ has a base
consisting of two-sided ideals; see Examples~\ref{T-lambda-examples}
for further cases when (T$_\omega$)~is satisfied).
 Let\/ $\R$ be the completion of $R$ with respect to the topology\/
$\bG$, viewed as a complete, separated topological ring in
the projective limit topology\/~$\bfG$.
 Then the left $R$\+module morphism~\eqref{delta-free-contra-compare}
$$
 \beta_{u,X}\:\Delta_u(R[X])\lrarrow\R[[X]]
$$
is an isomorphism for any set~$X$.
 Furthermore, the forgetful functor\/ $\R\contra\rarrow R\modl$ is
fully faithful, and its essential image coincides with the full
subcategory $R\modl_{u\ctra}\subset R\modl$, so there is an equivalence
of abelian categories
$$
 \R\contra\simeq R\modl_{u\ctra}.
$$
\end{thm}

\begin{proof}
 Since $u$~is injective, the free left $R$\+module $R[X]$
is $u$\+torsion-free for every set~$X$, and the right $R$\+module $R$
has $\omega$\+bounded (in fact, zero) $\bG$\+torsion.
 Applying Theorem~\ref{uncountable-Delta=Lambda-theorem} for $M=R[X]$,
we obtain the first assertion of the theorem.
 The remaining assertions follow from the first one by virtue of 
Lemma~\ref{contramodule-categories-comparison-criterion}.
\end{proof}

\begin{rem} \label{two-topological-rings-remark}
 Let $u\:R\rarrow U$ be an injective ring epimorphism.
 Set $K=U/R$, so $K$ is an $R$\+$R$\+bimodule; and denote by
$\sS=\Hom_R(K,K)^\rop$ the opposite ring to the ring of endomorphisms
of the left $R$\+module~$K$.
 So the ring $\sS$ acts in $K$ on the right, making $K$
an $R$\+$\sS$\+bimodule; while the right action of $R$ in $K$
induces a ring homomorphism $R\rarrow\sS$.
 We endow the ring $\sS$ with the right linear topology $\bfF$ with
a base $\bfB$ formed by the annihilators of finitely generated
left $R$\+submodules in~$K$.
 Then $\sS$ is a complete, separated topological
ring~\cite[Theorem~7.1]{PS} and $K$ is a discrete right
$\sS$\+module~\cite[Lemma~7.5]{PS} (see also~\cite[Section~1.13]{BP}).
 The topological ring~$\sS$ is discussed at length
in~\cite[Sections~17 and~19]{BP} (where it is denoted by~$\R$).

 Assume that $U$ is a flat left $R$\+module.
 Let $\bG$ be the perfect Gabriel topology of right ideals in $R$
related to the left flat ring epimorphism~$u$, and let $\R$ be
the completion of $R$ with respect to $\bG$, viewed as a complete,
separated topological ring in its projective limit topology~$\bfG$.
 Then $U/R$ is a discrete right $R$\+module (since it is
a $\bG$\+torsion right $R$\+module, because $U/R\ot_RU=0$), and
consequently $U/R$ also has a discrete right $\R$\+module structure
(see Sections~\ref{prelim-discrete}\+-\ref{prelim-gabriel-topologies}).
 It follows that the right action of $\R$ in $U/R$ induces
a continuous homomorphism of topological rings $\R\rarrow\sS$.
 Hence for every set $X$ we have the induced map of sets $\R[[X]]
\rarrow\sS[[X]]$.
 In fact, we have a commutative triangle diagram of ring homomorphisms
$R\rarrow\R\rarrow\sS$; so the map $\R[[X]]\rarrow\sS[[X]]$ is
a left $R$\+module morphism.

 Now let us assume additionally that $U$ is a left $R$\+module of
projective dimension not exceeding~$1$.
 Then we also have the left $R$\+module morphism
$\beta_{u,X}\:\Delta_u(R[X])\rarrow\R[[X]]$.
 Every left $\sS$\+contramodule, and in particular $\sS[[X]]$, has
the underlying left $\R$\+contramodule structure.
 By Proposition~\ref{uncountable-orthogonality}, $\sS[[X]]$ is
a $u$\+contramodule left $R$\+module.
 Hence, by Lemma~\ref{reflector-delta}, there is a unique
left $R$\+module morphism $\Delta_u(R[X])\rarrow\sS[[X]]$ forming
a commutative triangle diagram with the map~$\beta_{u,X}$ and
the map $R[X]\rarrow\sS[[X]]$ induced by the ring homomorphism
$R\rarrow\sS$.
 The composition of our two maps $\Delta_u(R[X])\rarrow\R[[X]]
\rarrow\sS[[X]]$ has this diagram commutativity property.
 The isomorphism $\Delta_u(R[X])\simeq\sS[[X]]$ constructed
in~\cite[direct proof of Theorem~19.2]{BP} also has the same
commutativity property.
 Thus the composition $\Delta_u(R[X])\rarrow\R[[X]]\rarrow\sS[[X]]$
is an isomorphism.

 According to Theorem~\ref{faithful-topology-theorem},
the map~$\beta_{u,X}$ is an isomorphism provided that the Gabriel
topology $\bG$ on $R$ satisfies the condition~(T$_\omega$).
 Then it follows that the map $\R[[X]]\rarrow\sS[[X]]$ is bijective
for every set~$X$.
 In particular, the associative ring homomorphism $\R\rarrow\sS$
is an isomorphism.
 It still does \emph{not} seem to follow from anything that it is
an isomorphism \emph{of topological rings} (i.~e., that
the topologies $\bfG$ and $\bfF$ on $\R=\sS$ are the same);
but it is a bijective continuous ring homomorphism
(so $\bfF\subset\bfG$) inducing a bijective map $\R[[X]]\rarrow
\sS[[X]]$ for every set~$X$.
\end{rem}

\end{document}


\title{Corrigenda to \\ ``Flat ring epimorphisms of countable type''}
\author{Leonid Positselski}

\address{Institute of Mathematics, Czech Academy of Sciences,
\v Zitn\'a~25, 115~67 Prague~1, Czech Republic}
\email{positselski@math.cas.cz}

\begin{abstract}
 We identify, discuss, and correct two mistakes in~\cite{Pcoun}.
 The first one is located in~\cite[Remark~3.3]{Pcoun} and
slightly affects~\cite[Lemma~3.6]{Pcoun}.
 The second mistake is in the proofs of~\cite[Proposition~5.1]{Pcoun}
and~\cite[Theorem~5.3]{Pcoun} (all the assertions of the proposition
and the theorem remain true, but the proofs need to be modified).
 We also clarify a confusion in~\cite[Remark~11.3]{Pcoun}, leading to
an improvement of~\cite[Theorem~11.2]{Pcoun}.
 Finally, we admit that we do not know whether a claim made at the end
of~\cite[Section~2.4]{Pcoun} (in the preliminary material) is true
in full generality; we prove it under additional assumptions.
 No other results of~\cite{Pcoun} are affected.
\end{abstract}

\maketitle

\Section{Directed Unions of Gabriel Topologies}

There are two assertions in the first paragraph
of~\cite[Remark~3.3]{Pcoun} concerning directed unions
of topologies of right ideals in an associative ring~$R$:
\begin{itemize}
\item[(1)] for any nonempty set $\Xi$ of right linear topologies $\bF$
on $R$ such that for every $\bF_1$, $\bF_2\in\Xi$ there exists
$\bF\in\Xi$ with $\bF_1\cup\bF_2\subset\bF$, the directed union
$\bigcup_{\bF\in\Xi}\bF$ is a right linear topology on~$R$;
\item[\sout{(2)}] in the same context, if $\bF$ is a right Gabriel
topology for every $\bF\in\Xi$, then $\bigcup_{\bF\in\Xi}\bF$ is also
a right Gabriel topology on~$R$.
\end{itemize}
 The first assertion, concerning right linear topologies, is correct.
 The second one, concerning Gabriel topologies, is wrong.

 Nevertheless, the following version of~\sout{(2)} is correct:
\begin{itemize}
\item[(2$^{\mathrm f}$)] in the context of~(1), if $\bF$ is a right
Gabriel topology with a base of finitely generated right ideals for
every $\bF\in\Xi$, then $\bigcup_{\bF\in\Xi}\bF$ is also a right Gabriel
topology with a base of finitely generated right ideals.
\end{itemize}
 Accordingly, the problem with~\sout{(2)} does not affect the second
paragraph of~\cite[Remark~3.3]{Pcoun}, which remains valid as stated:
\begin{itemize}
\item[(3)] in the context of~(1), if $\bF$ is a perfect right Gabriel
topology for every $\bF\in\Xi$, then $\bigcup_{\bF\in\Xi}\bF$ is also
a perfect right Gabriel topology on~$R$.
\end{itemize}

 A discussion of the problem with~\sout{(2)} follows below.

\begin{cex} \label{gabriel-counterex}
 Let $R$ be a commutative ring and $I\subset R$ be an ideal
with a set of generators $s_i\in R$.
 Denote by $\bG_I$ the collection of all ideals $J\subset R$
satisfying the following condition: for every $s\in I$ there exists
$m\ge1$ such that $s^m\in J$, or equivalently, for every index~$i$
there exists $m\ge1$ such that $s_i^m\in J$.
 Then $\bG_I$ is a Gabriel topology on~$R$.
 In fact, the Gabriel topology $\bG_I$ corresponds to the following
torsion class $\sT_I$ in $R\modl$: an $R$\+module $M$ belongs to
$\sT_I$ if for every $b\in M$ and $s\in I$ there exists $m\ge1$ such
that $s^mb=0$ in~$M$.
 Such $R$\+modules are called ``$I$\+torsion''
in~\cite[Sections~6--7]{Pcta}.

 Now let $R=k[x_1,x_2,\dotsc,y_1,y_2,\dots]$ denote the ring of
polynomials in a countably infinite number of variables (separated
into two countably infinite sorts) over a field~$k$.
 Let $J_0=(x_1,x_2,\dotsc)\subset R$ denote the ideal generated by
the variables~$x_i$ in~$R$.
 For every $n\ge1$, let $J_n\subset R$ denote the ideal generated
by the sequence of elements $y_1y_2\dotsm y_nx_i$, \,$i\ge1$.
 Clearly, one has $J_0\supset J_1\supset J_2\supset\dotsb$, hence
$\sT_{J_0}\subset\sT_{J_1}\subset\sT_{J_2}\subset\dotsb\subset R\modl$
and $\bG_{J_0}\subset\bG_{J_1}\subset\bG_{J_2}\subset\dotsb$.

 Let $\Xi$ denote the directed set of Gabriel topologies $\bG_{J_n}$,
\,$n\ge1$, on the ring~$R$, and let $\bH=\bigcup_{n=0}^\infty\bG_{J_n}$
be the union of~$\Xi$.
 We observe that $\bH$ is \emph{not} a Gabriel topology on~$R$.
 Indeed, let $I\subset R$ be the ideal generated by the sequence
of elements $x_iy_i$, \,$i\ge1$.
 Then $I\notin\bH$, since for every $n\ge0$ there exists $k=n+1$
such that for every $m\ge1$ the element $(y_1y_2\dotsm y_n x_k)^m$
does not belong to~$I$.

 Still, we have $J_0\in\bG_0\subset\bH$, and for every element
$s\in J_0$ the colon ideal $(I:s)$ belongs to~$\bH$.
 To check the latter assertion, pick an integer $n\ge1$ such that
$s$~belongs to the ideal $(x_1,\dotsc,x_n)\subset R$.
 Then $y_1y_2\dotsm y_n\in(I:s)$ and $(y_1y_2\dotsm y_n)\supset J_n
\in\bG_{J_n}\subset\bH$.
 So the filter of ideals $\bH$ in $R$ does not satisfy~(T4).
\end{cex}

 The following lemma is to be compared with~\cite[Lemma~3.1]{Pcoun}.

\begin{lem} \label{T4-generators}
 Let\/ $\bF$ be a right linear topology on an associative ring $R$,
let $I$ and $J\subset R$ be two right ideals, and let $s_j\in R$
be a set of generators of the ideal~$J$.
 Then one has $(I:s)\in\bF$ for all $s\in J$ if and only if
$(I:s_j)\in\bF$ for every generator~$s_j$.
\end{lem}

\begin{proof}
 Suppose $s=s_1r_1+\dotsb+s_mr_m$ with $r_i\in R$ and $s_i\in J$.
 Set $K_i=(I:s_i)$ and $H=(K_1:r_1)\cap\dotsb\cap(K_m:r_m)\subset R$.
 Then $sH\subset s_1K_1+\dotsb+s_mK_m\subset I$.
 Assume that $K_i\in\bF$ for every $i=1$,~\dots,~$m$; then
$(K_i:r_i)\in\bF$ by~(T3), hence $H\in\bF$ by~(T2).
 Since $H\subset(I:s)$, it follows that $(I:s)\in\bF$ by~(T1).
\end{proof}

 Let $\lambda$~be an infinite cardinal.
 A poset $\Xi$ is said to be \emph{$\lambda$\+directed} if for any
its subset $\Upsilon\subset\Xi$ of the cardinality less than~$\lambda$
there exists an element $\xi\in\Xi$ such that $\xi\ge\upsilon$ for
all $\upsilon\in\Upsilon$.

\begin{cor} \label{lambda-directed-lambda-generated}
 Let $R$ be an associative ring, $\lambda$~be an infinite cardinal,
and\/ $\Xi$ be a $\lambda$\+directed (by inclusion) set of right Gabriel
topologies on~$R$.
 Assume that every\/ $\bG\in\Xi$ has a base consisting of right ideals
with less than~$\lambda$ generators.
 Then\/ $\bH=\bigcup_{\bG\in\Xi}\bG$ is a right Gabriel topology on $R$
(with a base consisting of right ideals with less than~$\lambda$
generators).
\end{cor}

\begin{proof}
 To check that $\bH$ satisfies~(T4), consider a right ideal $I\subset R$
and a right ideal $J\in\bH$.
 Then there exist $\bG_0\in\Xi$ such that $J\in\bG_0$, and $J'\subset J$
such that $J'$ has less than~$\lambda$ generators~$s_j$ and $J'\in\bG_0$.
 Assume that $(I:s)\in\bH$ for every $s\in J$.
 Then there exist $\bG_j\in\Xi$ such that $(I:s_j)\in\bG_j$ for
every~$j$.
 Since $\Xi$ is $\lambda$\+directed, there is $\bG\in\Xi$ such that
$\bG_0\subset\bG$ and $\bG_j\subset\bG$ for every~$j$.
 Hence $J'\in\bG$ and $(I:s_j)\in\bG$ for every~$j$.
 By Lemma~\ref{T4-generators}, it follows that $(I:s)\in\bG$ for
every $s\in J'$.
 Since $\bG$ is a Gabriel topology, we can conclude that $I\in\bG
\subset\bH$.
\end{proof}

 Specializing to the case of the countable cardinal $\lambda=\omega$,
we see from Corollary~\ref{lambda-directed-lambda-generated}
that the assertion~(2$^{\mathrm f}$) is correct.

 The problem with~\sout{(2)} (demonstrated in
Counterexample~\ref{gabriel-counterex}) slightly
affects~\cite[Lemma~3.6]{Pcoun}, which should be restated as follows.

\begin{lem}[corrected version of~{\cite[Lemma~3.6]{Pcoun}}]
 Let $R$ be an associative ring, $\Xi$ be a directed set of right
Gabriel topologies on $R$, and\/ $\bG=\bigcup_{\bH\in\Xi}\bH$ be their
union.

 Assume that\/ $\bG$ is a tight Gabriel topology on $R$ (e.g., this
always holds when every Gabriel topology\/ $\bH\in\Xi$ has a base of
finitely generated right ideals, or more generally when\/ $\Xi$ is
$\lambda$\+directed and every\/ $\bH\in\Xi$ has a base consisting of
right ideals with less than~$\lambda$ generators).

 Let $N$ be a right $R$\+module such that $t_\bH(N)=t_\bG(N)$ for all\/
$\bH\in\Xi$.
 Then there is a natural isomorphism of right $R$\+modules
$N_\bG\simeq\varinjlim_{\bH\in\Xi}N_\bH$. \qed
\end{lem}

 Accordingly, the assumption that $\bG=\bigcup_{\bH\in\Xi}\bH$ is
a Gabriel topology is also needed in~\cite[Remark~3.7]{Pcoun}.

 In the proof of~\cite[Proposition~3.9]{Pcoun}, we sometimes refrained
from mentioning that the right Gabriel topologies we were dealing with
had to have a base of finitely generated right ideals.
 In view of the problem with~\sout{(2)}, this needs to be mentioned
throughout the proof of~\cite[Proposition~3.9]{Pcoun}.
 With this correction in mind, the proof is valid, and the assertion
of~\cite[Proposition~3.9]{Pcoun} remains unaffected.

 All the other results of~\cite[Section~3]{Pcoun} are likewise
unaffected.

\Section{Additive Kan Extensions}

 Let $R$ be an associative ring and $\bF$ be a right linear topology
on~$R$.
 Let $\sQ_\bF$ be the full subcategory of cyclic discrete right modules
$R/I$, \,$I\in\bF$, in the category of discrete right $R$\+modules
$\discr R$.
 Then
\begin{enumerate}
\item any left exact additive functor $M\:\sQ_\bF^\sop\rarrow\Ab$ (in
the sense of~\cite[Section~5]{Pcoun}) can be extended to a functor
$G_M\:(\discr R)^\sop\rarrow\Ab$ taking colimits in $\discr R$ to limits
in $\Ab$;
\item any right exact additive functor $C\:\sQ_\bF\rarrow\Ab$ (in
the sense of~\cite[Section~5]{Pcoun}) can be extended to a functor
$F_C\:\discr R\rarrow\Ab$ preserving colimits.
\end{enumerate}

 The assertion~(1) is used in the proof (or rather, in one of
the proofs) of~\cite[Proposition~5.1]{Pcoun}.
 The assertion~(2) is used in one of the proofs
of~\cite[Theorem~5.3]{Pcoun}.

 Both the assertions~(1) and~(2) are correct.
 However, the constructions of the functors $G_M$ and $F_C$ given
in~\cite[proofs of Proposition~5.1 and Theorem~5.3]{Pcoun} are wrong.
 They need to be modified as explained below.
 This problem is closely related to~\cite[Example~1.24\,(4)]{AR}.

 For any category $\sA$, a small full subcategory $\sQ\subset\sA$,
a cocomplete category $\sB$, and a functor $F\:\sQ\rarrow\sB$,
the functor $F$ can be extended to a functor $\widetilde F\:\sA
\rarrow\sB$ using the construction of
the Kan extension~\cite[Section~X]{McL}.
 By the definition, for any object $N\in\sA$, we put
$$
 \widetilde F(N)=\varinjlim\nolimits_{Q\to N}F(Q)
$$
where the colimit is taken over the diagram whose vertices are all
the morphisms $Q\rarrow N$ in $\sA$ with $Q\in\sQ$ and arrows
are the commutative triangles $Q'\rarrow Q''\rarrow N$ in $\sA$
with $Q'$, $Q''\in\sQ$.

 Now suppose that $\sA$ and $\sB$ are additive categories (so $\sQ$ is
a preadditive category), and $F$ is an additive functor.
 Then the functor $\widetilde F$ does not need to be additive,
as the following example demonstrates.

\begin{cex}
 Let $R=k$ be a field endowed with the discrete topology
$\bF=\{(0),(1)\}$.
 Let $\sA=\discr R$ be the category of $k$\+vector spaces, $\sB=\Ab$
be the category of abelian groups, and $\sQ=\sQ_\bF\subset\sA$ be
the full subcategory of vector spaces of dimension~$\le1$
(cf.~\cite[Example~1.24\,(4)]{AR}).

 Let $F\:\sQ\rarrow\sB$ be the forgetful functor assigning to a vector
space its underlying abelian group.
 Then the functor $\widetilde F\:\sA\rarrow\sB$ assigns to a vector
space $V$ the underlying abelian group of the $k$\+vector space with
a basis indexed by the set of all one-dimensional vector subspaces
in~$V$.
 So the functor $\widetilde F$ is \emph{not} additive (and \emph{not}
isomorphic to the forgetful functor $\sA\rarrow\sB$).
\end{cex}

 The above example shows that the construction of the functor $F_C$
in~\cite[proof of Theorem~5.3]{Pcoun} is wrong (in that it
does not produce a coproduct-preserving or right exact functor, contrary
to what is claimed).
 The construction of the functor $G_M$ in~\cite[proof of
Proposition~5.1]{Pcoun} is wrong for the same reason.
 The correct constructions are explained below.

\begin{lem}
 Let\/ $\sA$ and\/ $\sB$ be additive categories and\/ $\sQ\subset\sA$
be a full subcategory.
 Assume for simplicity that\/ $0\in\sQ$, and denote by\/
$\sQ^+\subset\sA$ the full subcategory consisting of all the objects
$Q_1\oplus Q_2$, where $Q_1$, $Q_2\in\sQ$.
 Let $F\:\sQ\rarrow\sB$ be an additive functor.
 Then there exists a unique additive functor $F^+\:\sQ^+\rarrow\sB$
such that $F^+|_\sQ=F$, defined by the rule $F^+(Q_1\oplus Q_2)=
F(Q_1)\oplus F(Q_2)$. \qed
\end{lem}

\begin{prop} \label{additive-Kan}
 Let\/ $\sA$ be an additive category, $0\in\sQ\subset\sA$ be a small
full subcategory, $\sB$ be a cocomplete additive category, and
$F\:\sQ\rarrow\sB$ be an additive functor.
 Let $\widetilde{F^+}:\sA\rarrow\sB$ be the Kan extension of
the functor $F^+\:\sQ^+\rarrow\sB$.
 Then the functor $\widetilde{F^+}$ is additive.
\end{prop}

\begin{proof}
 Let $M$ and $N\in\sA$ be two objects, and let $f$, $g\:M
\rightrightarrows N$ be a pair of parallel morphisms.
 We have to check that $\widetilde{F^+}(f+g)=\widetilde{F^+}(f)
+\widetilde{F^+}(g)$.

 The object $\widetilde{F^+}(M)\in\sB$ is the colimit of the objects
$F(Q_1)\oplus F(Q_2)$ taken over the diagram whose vertices are
all the morphisms $Q_1\oplus Q_2\rarrow M$ with $Q_1$, $Q_2\in\sQ$
and whose arrows are all the commutative triangles $Q_1'\oplus Q_2'
\rarrow Q_1''\oplus Q_2''\rarrow M$, where $Q_1'\oplus Q_2'
\rarrow Q_1''\oplus Q_2''$ is a $(2\times2)$\+matrix of morphisms
between objects of~$\sQ$.
 Therefore, it suffices to check that for every morphism $h\:Q_1\oplus
Q_2\rarrow M$ with $Q_1$, $Q_2\in\sQ$ one has
$\widetilde{F^+}((f+g)\circ h)=\widetilde{F^+}(f\circ h)+
\widetilde{F^+}(g\circ h)$.

 In turn, it suffices to check the latter condition in the case of
a pair of objects $(Q_1,Q_2)=(Q,0)$, i.~e., for a morphism
$h\:Q\rarrow M$ with $Q\in\sQ$.
 Now we consider the object $(Q,Q)\in\sQ^+$ and the morphism
$(f\circ h,\>g\circ h)\:(Q,Q)\rarrow N$ in~$\sA$.
 There are three natural morphisms $(1,0)$, $(0,1)$, and $(1,1)\:
Q\rarrow(Q,Q)$ in~$\sQ^+$.
 One has $f\circ h=(f\circ h,\>g\circ h)\circ(1,0)$, and similarly
$g\circ h=(f\circ h,\>g\circ h)\circ(0,1)$ and
$(f+g)\circ h=(f\circ h,\>g\circ h)\circ(1,1)$.

 By construction of the object $\widetilde{F^+}(N)$, it follows that
$\widetilde{F^+}(f\circ h)=\widetilde{F^+}(f\circ h,\>g\circ h)
\circ F^+(1,0)$, and similarly $\widetilde{F^+}(g\circ h)=
\widetilde{F^+}(f\circ h,\>g\circ h)\circ F^+(0,1)$ and
$\widetilde{F^+}((f+g)\circ h)=\widetilde{F^+}(f\circ h,\>g\circ h)
\circ F^+(1,1)$.
 It remains to observe that $F^+(1,1)=F^+(1,0)+F^+(0,1)$, since
the functor $F^+$ is additive.
\end{proof}

 The following lemma does not depend on any additivity assumptions.

\begin{lem} \label{Kan-preserves-colimits}
 Let\/ $\sA$ be a category, $\sQ\subset\sA$ be a small full subcategory,
$\sB$ be a cocomplete category, $F\:\sQ\rarrow\sB$ be a functor,
$\sX$ be a small category, and $D\:\sX\rarrow\sA$ be a diagram which
has a colimit in\/~$\sA$.
 Assume that for every object $Q\in\sQ$ the functor\/ $\Hom_\sA(Q,{-})\:
\sA\rarrow\Sets$ preserves the colimit of the diagram~$D$.
 Then the functor $\widetilde F\:\sA\rarrow\sB$ also preserves
the colimit of the diagram~$D$.
\end{lem}

\begin{proof}
 Let $B\in\sB$ be an object.
 Then, for every object $A\in\sA$, the set of all morphisms
$\widetilde F(A)\rarrow B$ in $\sB$ is naturally bijective to the set of
all rules assigning to every morphism $Q\rarrow A$ in $\sA$ with
an object $Q\in\sQ$ a morphism $F(Q)\rarrow B$ in $\sB$ in a way
compatible with all the morphisms in~$\sQ$.
 It follows that, under the assumptions of the lemma, both the sets of
morphisms $\Hom_\sB(\varinjlim_{x\in\sX} F(D(x)),\>B)$ and
$\Hom_\sB(F(\varinjlim_{x\in\sX}D(x)),B)$ are naturally bijective to
the set of all rules assigning to every morphism $Q\rarrow D(x)$
in $\sA$ with objects $Q\in\sQ$ and $x\in\sX$ a morphism
$F(Q)\rarrow B$ in $\sB$ in a way compatible with all the morphisms
in $\sQ$ and~$\sX$.

 Alternatively, one can argue as follows.
 Consider the category $\sC=\Sets^{\sQ^\sop}$ of presheaves of sets
on the small category~$\sQ$.
 By the Yoneda lemma, $\sQ$ is naturally a full subcategory in~$\sC$.
 There is a natural functor $H\:\sA\rarrow\sC$ assigning to
an object $A\in\sA$ the presheaf $H(A)=\Hom_\sA({-},A)|_\sQ$.
 The functor $H$ forms a commutative triangle diagram with the fully
faithful functors $\sQ\rarrow\sA$ and $\sQ\rarrow\sC$.
 By assumption, the functor $H$ preserves the colimit of the diagram~$D$.

 Let $G\:\sC\rarrow\sB$ be the Kan extension of the functor $F\:\sQ
\rarrow\sB$ with respect to the Yoneda embedding $\sQ\rarrow\sC$.
 The functor $G$ has a right adjoint functor $R$ assigning to
an object $B\in\sB$ the presheaf $R(B)=\Hom_\sB(F({-}),B)$.
 It follows that the functor $G$ preserves all colimits.
 Finally, it remains to observe that $\widetilde F=G\circ H$.
\end{proof}

 Now we can return to the situation at hand.
 Let $\bF$ be a right linear topology on an associative ring~$R$.
 Set $\sA=\discr R$, \ $\sQ=\sQ_\bF\subset\sA$, and
$\sB=\Ab$ or $\Ab^\sop$.

 Given an additive functor $M\:\sQ_\bF\rarrow\Ab^\sop$, we set
$$
 G_M(\N)=\varprojlim\nolimits_{R/I_1\oplus R/I_2\to\N}
 M(R/I_1)\oplus M(R/I_2) \qquad \text{for every $\N\in\discr R$},
$$
where the projective limit is taken over the diagram formed by
all the morphisms of discrete right $R$\+modules $R/I_1\oplus R/I_2
\rarrow\N$ (indexing the vertices of the diagram) and all
the commutative triangles $R/I_1\oplus R/I_2\rarrow R/J_1\oplus
R/J_2\rarrow\N$ (indexing the arrows), where $I_1$, $I_2$, $J_1$,
$J_2\in\bF$ and $R/I_1\oplus R/I_2\rarrow R/J_1\oplus R/J_2$ ranges
over all the $(2\times2)$\+matrices of morphisms in~$\sQ_\bF$.

 By Proposition~\ref{additive-Kan}, the functor $G_M=\widetilde{M^+}\:
\discr R\rarrow\Ab^\sop$ is additive; hence it preserves finite
(co)products.
 Applying Lemma~\ref{Kan-preserves-colimits} for the diagram
representing a coproduct in $\discr R$ as the filtered colimit of
its finite subcoproducts and recalling that the right $R$\+modules
$R/I_1\oplus R/I_2$ are finitely generated, we conclude that
the functor $G_M$ takes coproducts in $\discr R$ to products in~$\Ab$.

 Checking that the functor $G_M$ takes cokernels in $\discr R$ to
kernels in $\Ab$ if and only if the functor $M$ is left exact in
the sense of~\cite[Section~5]{Pcoun} is a straightforward
diagram-chasing exercise.
 This proves the assertion~(1).

 Given an additive functor $C\:\sQ_\bF\rarrow\Ab$, we set
$$
 F_C(\N)=\varinjlim\nolimits_{R/I_1\oplus R/I_2\to\N}
 C(R/I_1)\oplus C(R/I_2)
 \qquad\text{for every $\N\in\discr R$},
$$
where the inductive limit is taken over the same diagram as
in the previuos construction.
 By Proposition~\ref{additive-Kan}, the functor $F_C=\widetilde{C^+}\:
\discr R\rarrow\Ab$ is additive; hence it preserves finite
coproducts.
 The same argument based on Lemma~\ref{Kan-preserves-colimits} as above
shows that the functor $F_C$ preserves infinite coproducts, too.
 Finally, a diagram-chasing exercise shown that the functor $F_C$
preserves cokernels if and only if the functor $C$ is right exact
in the sense of~\cite[Section~5]{Pcoun}.
 This proves the assertion~(2).

\Section{Injective Flat Ring Epimorphisms of Projective Dimension~1}
\label{injective-flat-epi-secn}

 The exposition in~\cite[Remark~11.3]{Pcoun} is correct, but confused.
 Particularly confused is the fourth, last paragraph of that remark
(while the first two paragraphs are OK).
 Let us discuss the situation anew and with updated references.

 Let $u\:R\rarrow U$ be an injective ring epimorphism.
 Following~\cite[first paragraph of Remark~11.3]{Pcoun}, we consider
the $R$\+$R$\+bimodule $K=U/R$, and denote by $\sS=\Hom_R(K,K)^\rop$
the opposite ring to the ring of endomorphisms of the left
$R$\+module~$K$.
 So the ring $\sS$ acts in $K$ on the right, making $K$
an $R$\+$\sS$\+bimodule; and the right action of $R$ in $K$ induces
a ring homomorphism $R\rarrow\sS$.
 We endow the ring $\sS$ with the right linear topology $\bfF$
with a base $\bfB$ formed by the annihilators of finitely generated
left $R$\+submodules in~$K$.
 Then $\sS$ is a complete, separated topological
ring~\cite[Theorem~7.1]{PS1} and $K$ is a discrete right
$\sS$\+module~\cite[Proposition~7.3]{PS1}
(see also~\cite[Section~1.13]{Pproperf}).
 The topological ring $\sS$ is discussed at length
in~\cite[Section~4]{BP2} and~\cite[Sections~14\+-15]{BP3}
(where it is denoted by~$\R$).

 Assume that $U$ is a flat left $R$\+module.
 Following~\cite[second paragraph of Remark~11.3]{Pcoun}, let $\bG$
be the perfect Gabriel topology of all right ideals $I\subset R$
such that $R/I\ot_RU=0$ (or equivalently, $U=IU$).
 Let $\R$ denote the completion of $R$ with respect to $\bG$, viewed
as a complete, separated topological ring in its projective limit
topology~$\bfG$; and let $\rho\:R\rarrow\R$ be the completion map.
 Then $K=U/R$ is a discrete right $R$\+module (with respect to
the topology $\bG$ on~$R$), since $U/R\ot_RU=0$.
 Consequently, the right action of $R$ in $K$ extends uniquely to
a discrete right action of~$\R$.
 Hence the ring homomorphism $R\rarrow\sS$ factorizes as
$R\rarrow\R\rarrow\sS$.
 Since the annihilator of any finite set of elements in $K$ is
an open right ideal in $\R$, the ring homomorphism $\sigma\:
\R\rarrow\sS$ is continuous.

\begin{lem} \label{topology-comparison}
 Any open right ideal $I\in\bG$ of the ring $R$ contains
the preimage $(\sigma\rho)^{-1}(\J)\subset R$ of some open
right ideal\/ $\J\in\bfF$ of the ring\/~$\sS$.
\end{lem}
 
\begin{proof}
 We know that $U=IU$, that is, there exist some elements $s_1$,~\dots,
$s_n\in I$ and $v_1$,~\dots, $v_n\in U$ such that $\sum_{k=1}^n
s_kv_k=1$ in the ring~$U$.
 Denote by $\J\subset\sS$ the annihilator of the finite collection
of cosets $v_1+R$,~\dots, $v_n+R\in U/R=K$.
 Then $\J$ is an open right ideal in~$\sS$ by the definition of
the topology $\bfF$ on~$\sS$.

 The ring homomorphism $\sigma\rho\:R\rarrow\sS$ is induced by
the right action of $R$ in~$U/R$.
 Let $r\in R$ be an element such that $\sigma\rho(r)\in\J$.
 This means that the cosets $v_1+R$,~\dots, $v_n+R$ are
annihilated by the right action of~$r$, that is, the elements
$v_1r$,~\dots, $v_nr$ belong to $R\subset U$.
 Hence we have $r=\sum_{k=1}^n s_kv_kr\in\sum_{k=1}^n s_kR
\subset I$, as desired.
\end{proof}

\begin{cor} \label{injective-induced-topology}
 For any injective left flat ring epimorphism $u\:R\rarrow U$,
the ring homomorphism $\sigma\:\R\rarrow\sS$ is injective, and
the topology\/ $\bfG$ on\/ $\R$ coincides with the topology induced
on\/ the ring\/ $\R$ as a subring in\/ $\sS$ via~$\sigma$.
\end{cor}

\begin{proof}
 The preimages of open right ideals $\J\in\bfF$ of the ring $\sS$
under the map $\sigma\rho\:R\rarrow\sS$ are open right ideals in $R$,
because the ring homomorphism~$\sigma\rho$ is continuous.
 The open right ideals $(\sigma\rho)^{-1}(\J)\subset R$ also form
a base of the topology $\bG$ on the ring $R$ by
Lemma~\ref{topology-comparison}.
 Passing to the projective limit of the directed diagram of injective
maps $R/(\sigma\rho)^{-1}(\J)\rarrow\sS/\J$ indexed by $\J\in\bfF$,
we obtain the map $\sigma\:\R\rarrow\sS$.
 Thus the map~$\sigma$ is injective.

 The topology $\bfG$ on $\R$ is the topology of projective limit of
discrete abelian groups
$\varprojlim_{\J\in\bfF}R/(\sigma\rho)^{-1}(\J)$, while the topology
$\bfF$ on $\sS$ is the topology of projective limit of discrete
abelian groups $\varprojlim_{\J\in\bfF}\sS/\J$.
 The map~$\sigma$ is the projective limit of an injective morphism
from the former diagram of abelian groups to the latter one.
 By the definition of the topology of projective limit, it follows
that the topology on $\R$ is induced from its embedding into $\sS$
via~$\sigma$.
\end{proof}

 The following theorem is a clarified version
of~\cite[Remark~11.3]{Pcoun}, and it also includes an improved version
of~\cite[Theorem~11.2]{Pcoun}.

\begin{thm}
 Let $u\:R\rarrow U$ be an injective ring epimorphism such that
$U$ is a flat left $R$\+module of projective dimension
not exceeding\/~$1$.
 Then the map $\sigma\:\R\rarrow\sS$ is an isomorphism of topological
rings.
 Furthermore, the left $R$\+module morphism
$$
 \beta_{u,X}\:\Delta_u(R[X])\lrarrow\R[[X]]
$$
is an isomorphism for any set~$X$.
 The forgetful functor\/ $\R\contra\rarrow R\modl$ is fully faithful,
and its essential image coincides with the full subcategory
$R\modl_{u\ctra}\subset R\modl$; so there is an equivalence of
abelian categories\/ $\R\contra\simeq R\modl_{u\ctra}$.
\end{thm}

\begin{proof}
 We follow~\cite[second and third paragraphs of Remark~11.3]{Pcoun}
taking into account the discussion above.
 The injective topological ring homomorphism $\sigma\:\R\rarrow\sS$
(see Corollary~\ref{injective-induced-topology}) induces, for every
set $X$, an injective map of sets $\sigma[[X]]\:\R[[X]]\rarrow\sS[[X]]$.
 In fact, we have a commutative triangle diagram of ring homomorphisms
$R\rarrow\R\rarrow\sS$, so $\sigma[[X]]$ is an $R$\+module morphism.

 Every left $\sS$\+contamodule has an underlying left
$\R$\+contramodule structure.
 By~\cite[Proposition~10.2]{Pcoun}, the underlying left $R$\+module
of any left $\R$\+contramodule is a $u$\+contramodule.
 In particular, $\sS[[X]]$ is a $u$\+contramodule left $R$\+module.
 Hence, by~\cite[Lemma~10.1]{Pcoun}, there exists a unique $R$\+module
morphism $\Delta_u(R[X])\rarrow\sS[[X]]$ forming a commutative triangle
diagram with the adjunction map $\delta_{u,R[X]}\:R[X]\rarrow
\Delta_u(R[X])$ and the map $R[X]\rarrow \sS[[X]]$ induced by the ring
homomorphism $\sigma\rho\:R\rarrow\sS$.
 The composition of two maps $\sigma[[X]]\circ\beta_{u,X}\:
\Delta_u(R[X])\rarrow\R[[X]]\rarrow\sS[[X]]$ has this diagram
commutativity property.
 So does the isomorphism $\Delta_u(R[X]])\simeq\sS[[X]]$
constructed in~\cite[direct proof of Theorem~14.2]{BP3} 
(we remind the reader that the topological ring denoted by $\sS$ here
is denoted by $\R$ in~\cite{BP3}).
 Thus the composition $\sigma[[X]]\circ\beta_{u,X}$ is
isomorphism.

 Since the map $\sigma[[X]]$ is injective, it follows that
both the maps $\beta_{u,X}$ and $\sigma[[X]]$ are isomorphisms.
 In particular, the map~$\sigma$ is an isomorphism of rings, and
in view of Corollary~\ref{injective-induced-topology} we can conclude
that $\sigma$~is an isomorphism of topological rings.
 We have proved the first two assertions of the theorem.
 The remaining assertions follow in view of~\cite[Lemma~11.1]{Pcoun}
or~\cite[Theorem~14.2]{BP3}.
\end{proof}

\Section{Completions of Gabriel Topologies are Gabriel?}

 Let $R$ be a topological ring with a right linear topology~$\bG$,
and let $\R=\varprojlim_{I\in\bG}R/I$ be the completion of $R$,
viewed as a topoogical ring in the projective limit topology~$\bfG$.
 It is claimed in~\cite[final paragraph of Section~2.4]{Pcoun} (in
the preliminary material to the paper~\cite{Pcoun}) that if
$\bG$ is a Gabriel topology on $R$, then $\bfG$ is a Gabriel
topology on~$\R$.

 This claim is never used in the main part of the paper, and was
only included in the preliminaries for the sake of completeness of
the exposition.
 In fact, we do not know whether this claim is true.

 Let us have a little discussion of what this claim entails.
 A right linear topology on a ring is Gabriel if and only if the full
subcategory of discrete right modules is closed under extensions in
the category of all right modules.
 Let $0\rarrow\N'\rarrow N\rarrow\N''\rarrow0$ be a short exact sequence
of right $\R$\+modules, where the $\R$\+modules $\N'$ and $\N''$ are
discrete.
 Restricting the scalars, we can view this sequence as a short exact
sequence of right $R$\+modules.
 Then $\N'$ and $\N''$ are discrete right $R$\+modules.
 Since $\bG$ is a Gabriel topology by assumption, it follows that $N$
is a discrete right $R$\+module.
 But how can one conclude that $N$ is a discrete right $\R$\+module?

 Let $b\in N$ be an element, and let $\rho\:R\rarrow\R$ denote
the completion map.
 The annihilator of~$b$ in $\mathfrak R$ is a right ideal $\I$ whose
full preimage $I=\rho^{-1}(\I)$ is an open right ideal in~$R$.
 If we knew that $\I$ is closed in the topology of $\R$, it would
follow that $\I$ is open in~$\R$.
 But does $\I$ need to be closed?

 Alternatively, in order to show that $\bfG$ is Gabriel it suffices to
check condition~(T4) for the ring~$\R$.
 Let $\I\subset\R$ be a right ideal and $\J\in\bfG$ be an open right
ideal in~$\R$.
 Suppose that $(\I:s)\in\bfG$ for all $s\in\J$.
 How does one show that $\I\in\bfG$\,?
 What is the technology for proving openness or closedness of a given
right ideal in $\R$, based on whatever knowledge of right ideals
in~$R$\,?

 The converse implication is also interesting and unclear.
 Let $\R$ be the completion of $R$ with respect to a right linear
topology $\bG$ (as above).
 Suppose that $\bfG$ is a Gabriel topology on~$\R$.
 Does it follow that $\bG$ is a Gabriel topology on~$R$\,?
 
 Let $0\rarrow\N'\rarrow N\rarrow\N''\rarrow0$ be a short exact sequence
of right $R$\+modules, where the $R$\+modules $\N'$ and $\N''$ are
discrete.
 Is the $R$\+module $N$ discrete in this case?
 The discrete $R$\+module structures on $\N'$ and $\N''$ can be uniquely
extended to structures of discrete $\R$\+modules.
 If one could extend the action of $R$ in $N$ to an action of $\R$, then
it would follow that $N$ is a discrete right $\R$\+module (since
the topology $\bfG$ is assumed to be Gabriel); hence $N$ is also
a discrete right $R$\+module.
 But how does one extend the action of $R$ in $N$ to an action of~$\R$\,?

 For example, if $N=R/I$ is a cyclic right $R$\+module, one could
consider the extension $\I=\rho(I)\R\subset\R$ of the right ideal
$I\subset R$ to the ring~$\R$.
 But the equality $\rho^{-1}(\I)=I$ need not hold for an arbitrary
right ideal $I\subset R$, generally speaking.
 For example, if $R=\boZ$ is the ring of integers and $\bG$ is
the $p$\+adic topology on $R$ (where $p$~is a prime number), so
$\R=\boZ_p$ is the ring of $p$\+adic integers, then for any integer~$m$
not divisible by~$p$ and for the ideal $I=(m)$ one has $\I=\R$.
 
 Here is a result which we can actually prove.

\begin{lem}
 Let $R$ be a topological ring with a right Gabriel topology\/ $\bG$
having a countable base consisting of finitely generated right ideals.
 Let\/ $\R=\varprojlim_{I\in\bG}R/I$ be the completion of $R$ with
respect to\/ $\bG$, and let\/ $\bfG$ be the projective limit
topology on\/~$\R$.
 Then\/ $\bfG$ is a right Gabriel topology with a countable base
consisting of finitely generated right ideals in\/~$\R$.
\end{lem}

\begin{proof}
 First of all, there is a natural order isomorphism between the posets
$\bG$ and $\bfG$ (with respect to the inverse inclusion), as per
the discussion in~\cite[Section~4]{Pcoun}.
 In particular, this bijection takes any countable base $\bB$ of
the topology $\bG$ to a countable base $\bfB$ of the topology~$\bfG$.
 Furthermore, let $I\in\bB$ be an open right ideal generated by
a finite set of elements $s_1$,~\dots, $s_m\in R$.
 Let $\I\in\bfB$ be the open right ideal corresponding to~$I$.
 By~\cite[Theorem~6.6]{Pcoun}, the right ideal $I$ is strongly finitely
generated.
 This means, in particular (taking the sigleton set of indices
$X=\{*\}$), that $I\R=\I$; so the ideal $\I$ is generated by
the finite set of elements $\rho(s_1)$,~\dots, $\rho(s_m)\in\R$.

 In order to show that $\bfG$ is a Gabriel topology, let us
use~\cite[Lemma~3.1]{Pcoun}.
 We have to check that the condition (T4$'$) holds for the topology
$\bfG$ on $\R$, that is, for any open right ideal $\J\in\bfB$
the right ideal $\rho(s_1)\J+\dotsb+\rho(s_m)\J$ is open in~$\R$.
 Let $J\in\bB$ be the open right ideal in $R$ corresponding to $\J$,
and let $t_1$,~\dots, $t_n\in R$ be its finite set of generators.
 Since $\bG$ is assumed to be a right Gabriel topology on $R$,
by the other implication in~\cite[Lemma~3.1]{Pcoun} we know that
$K=s_1J+\dotsb+s_mJ$ is an open right ideal in~$R$.
 Clearly, $K$ is also a finitely generated right ideal, with
the set of generators $(s_jt_i)_{1\le j\le m,1\le i\le n}$.
 Then the argument in the previous paragraph tells that the right
ideal generated by the finite set of elements
$(\rho(s_jt_i))_{1\le j\le m,1\le i\le n}$ is open in~$\R$.
 Thus $\rho(s_1)\J+\dotsb+\rho(s_m)\J$ is an open right ideal
in~$\R$.
\end{proof}

\bigskip\noindent
\textbf{Acknowledgement.}
 The author is grateful to J.~\v St\!'ov\'\i\v cek for suggesting
the above alternative proof of Lemma~\ref{Kan-preserves-colimits}.
 Section~\ref{injective-flat-epi-secn} owes its existence to
S.~Bazzoni, who kindly invited me to visit her in Padova in
November--December~2019 and clarified my confusion by suggesting
Lemma~\ref{topology-comparison} and its proof.